\documentclass[a4paper,12pt,reqno]{amsart}
\usepackage[T1]{fontenc}
\usepackage[utf8]{inputenc}
\usepackage[english]{babel}
\usepackage{amsmath,amsthm,amssymb}
\usepackage{enumitem}
\usepackage{fullpage}
\usepackage{siunitx}
\usepackage{tikz,tikz-3dplot}
\usetikzlibrary{calc,spy}
\usepackage{pgfplotstable}
\pgfplotsset{tick label style = {font = \tiny}}
\usepackage[width=0.95\textwidth, font = small]{caption}
\usepackage[width=0.85\textwidth]{subcaption}
\usepackage{chemmacros}
\usepackage{hyperref}
\makeatletter
\def\@seccntformat#1{%
  \protect\textup{\protect\@secnumfont
    \ifnum\pdfstrcmp{subsection}{#1}=0 \bfseries\fi
    \csname the#1\endcsname
    \protect\@secnumpunct
  }%
}
\makeatother
\newtheorem{theorem}{Theorem}[section]
\newtheorem{proposition}[theorem]{Proposition}
\newtheorem{lemma}[theorem]{Lemma}
\newtheorem{algorithm}[theorem]{Algorithm}
\newtheorem{definition}[theorem]{Definition}
\newtheorem{remark}[theorem]{Remark}
\renewcommand{\aa}{\boldsymbol{a}}
\newcommand\ff{\boldsymbol{f}}
\newcommand\hh{\boldsymbol{h}}
\newcommand\mm{\boldsymbol{m}}
\newcommand\nn{\boldsymbol{n}}

\newcommand\uu{\boldsymbol{u}}
\newcommand\vv{\boldsymbol{v}}

\newcommand\xx{\boldsymbol{x}}

\newcommand\zz{\boldsymbol{z}_h}
\newcommand\CC{\boldsymbol{C}}
\newcommand\HH{\boldsymbol{H}}

\newcommand\LL{\boldsymbol{L}}

\newcommand\WW{\boldsymbol{W}}
\newcommand\M{\mathcal{M}}
\newcommand\E{\mathcal{E}}
\newcommand\K{\mathcal{K}}

\newcommand\N{\mathcal{N}}

\newcommand\T{\mathcal{T}}

\newcommand\Kh{\boldsymbol{\K}_h}
\newcommand\Mh{\boldsymbol{\M}_h}
\newcommand\Nh{\N_h}

\newcommand\Th{\T_h}
\newcommand\pphi{\boldsymbol{\phi}}
\newcommand\vvphi{\boldsymbol{\vphi}}
\newcommand\ppsi{\boldsymbol{\psi}}
\newcommand\eps{\varepsilon}

\newcommand\vphi{\varphi}

\newcommand\interp{\mathcal{I}_h}
\newcommand\Interp{\boldsymbol{\mathcal{I}}_h}
\newcommand\mmh{\mm_h}
\newcommand\mmhk{\mm_{hk}}
\newcommand\mmhkt{\de_t\mm_{hk}}

\newcommand\pphih{\pphi_h}

\newcommand\ppsih{\ppsi_h}
\newcommand\vvh{\vv_h}
\newcommand\vvhk{\vv_{hk}}
\newcommand\0{\boldsymbol{0}}
\newcommand\sphere{\mathbb{S}^2}
\newcommand\curl{\nabla\times}
\renewcommand\div{\nabla\cdot}
\newcommand\grad{\nabla}
\newcommand\Grad{\boldsymbol{\nabla}}
\newcommand\lapl{\Delta}
\newcommand\Lapl{\boldsymbol{\Delta}}
\newcommand\real{\mathbb{R}}
\renewcommand{\vec}[1]{\mathbf{#1}}
\newcommand{\abs}[1]{\left\lvert #1 \right\rvert}
\newcommand{\edual}[2]{\langle\hspace*{-1mm}\langle #1,#2 \rangle\hspace*{-1mm}\rangle}
\newcommand{\inner}[3][]{\langle #2,#3 \rangle_{#1}}
\newcommand{\norm}[2][]{\left\lVert #2 \right\rVert_{#1}}
\DeclareMathOperator{\diam}{diam}

\DeclareMathOperator{\spann}{span}
\DeclareMathOperator{\trace}{tr}
\newcommand\ddt{\frac{\mathrm{d}}{\mathrm{d}t}}
\newcommand\de{\partial}
\newcommand\dt{\mathrm{d}t}
\newcommand\dx{\mathrm{d}\xx}

\newcommand\mmt{\de_t \mm}
\newcommand\Heff{\HH_{\mathrm{eff}}}
\newcommand\Hext{\HH_{\mathrm{ext}}}
\newcommand\Hstray{\HH_{\mathrm{s}}}
\newcommand\ldm{\ell_{\mathrm{dm}}}
\newcommand\lex{\ell_{\mathrm{ex}}}
\newcommand\Ms{M_{\mathrm{s}}} 
\newcommand\heff{\hh_{\mathrm{eff}}}
\newcommand\Cinv{C_{\mathrm{inv}}}
\newcommand\Cnorm{C_{\mathrm{norm}}}
\newcommand\Cgeo{C_{\mathrm{geo}}}
\newcommand{\Hcurl}[1]{\HH(\mathrm{curl},#1)}
\newcommand{\bigO}[1]{\mathcal{O}(#1)}
\newcommand\weakto{\rightharpoonup}
\newcommand\weakstarto{\stackrel{*}{\rightharpoonup}}
\begin{document}
\title{Convergent tangent plane integrators for the simulation of chiral magnetic skyrmion dynamics}
\author{Gino~Hrkac}
\author{Carl-Martin~Pfeiler}
\author{Dirk~Praetorius}
\author{Michele~Ruggeri}
\author{Antonio~Segatti}
\author{Bernhard~Stiftner}
\address[Gino Hrkac]{College of Engineering, Mathematics and Physical Sciences, University of Exeter, North Park Road, EX4 4QF, Exeter, UK}
\email{G.Hrkac@exeter.ac.uk}
\address[Carl-Martin Pfeiler, Dirk Praetorius, Bernhard Stiftner]{Institute for Analysis and Scientific Computing, TU Wien, Wiedner Hauptstra{\ss}e 8--10, 1040, Vienna, Austria}
\email{carl-martin.pfeiler@asc.tuwien.ac.at}
\email{dirk.praetorius@asc.tuwien.ac.at}
\email{bernhard.stiftner@asc.tuwien.ac.at}
\address[Michele Ruggeri]{Faculty of Mathematics, University of Vienna, Oskar-Morgenstern-Platz 1, 1090 Vienna, Austria}
\email[Corresponding author]{michele.ruggeri@univie.ac.at}
\address[Antonio Segatti]{Dipartimento di Matematica ``F. Casorati'', Universit\`a di Pavia, Via Ferrata 5, 27100 Pavia, Italy}
\email{antonio.segatti@unipv.it}
\date{\today}
\thanks{\emph{Acknowledgements.}
This research has been supported by the Vienna Science and Technology Fund (WWTF) through the project \emph{Thermally controlled magnetization dynamics} (grant MA14-44), by the Austrian Science Fund (FWF) through the doctoral school \emph{Dissipation and dispersion in nonlinear PDEs} (grant W1245) and the special research program \emph{Taming complexity in partial differential systems} (grant SFB F65), and by the Engineering and Physical Sciences Research Council (EPSRC) through the projects \emph{Picosecond dynamics of magnetic exchange springs} (grant EP/P02047X/1) and 
\emph{Coherent spin waves for emerging nanoscale magnonic logic architectures} (grant EP/L019876/1).
The authors also thank S.\ Komineas (University of Crete, Heraklion, Greece) for an informal and stimulating discussion on the topic of this work.}
\begin{abstract}
We consider the numerical approximation of the Landau--Lifshitz--Gilbert equation, which describes the dynamics of the magnetization in ferromagnetic materials.
In addition to the classical micromagnetic contributions, the energy comprises the Dzyaloshinskii--Moriya interaction, which is the most important ingredient for the enucleation and the stabilization of chiral magnetic skyrmions.
We propose and analyze three tangent plane integrators, for which we prove (unconditional) convergence of the finite element solutions towards a weak solution of the problem.
The analysis is constructive and also establishes existence of weak solutions.
Numerical experiments demonstrate the applicability of the methods for the simulation of practically relevant problem sizes.
\end{abstract}
\maketitle
\section{Introduction}
\subsection{State of the art}
Magnetic skyrmions are topologically protected vortex-like magnetization configurations~\cite{nt2013,fbtck2016,wiesendanger2016}, which have been theoretically predicted~\cite{by1989,bh1994,br2001,rbf2006} and experimentally observed~\cite{mbjprngb2009,rhmbwbkw2013} in several magnetic systems.
The most important ingredient for the enucleation and the stabilization of magnetic skyrmions is the so-called Dzyaloshinskii--Moriya interaction (DMI); see~\cite{dzyaloshinskii1958,moriya1960}.
It is a short-range effect, sometimes also referred to as antisymmetric exchange, which exerts a torque on the magnetization inducing neighboring spins to be perpendicular to each other.
It is thus in direct competition with the classical Heisenberg exchange interaction, which conversely favors uniform configurations.
The DMI is modeled by an energy contribution, which is linear in the first spatial derivatives of the magnetization and is added to the micromagnetic energy for chiral ferromagnets.
Magnetic skyrmions are currently subject of intense scientific research, which includes theoretical, computational, and experimental studies; see, e.g., \cite{hbmbkwbb2011,scrtf2013,kp2015a,hkkykclksky2016,babcwcvhcsmf2017}.
As for the mathematical literature, the existence of isolated skyrmions emerging as energy minimizers of two-dimensional micromagnetic models and their dynamic stability have been investigated in~\cite{melcher2014,dm2017}, whereas chiral domain walls in ultrathin ferromagnetic films have been studied in~\cite{ms2017}.
The growing interest in skyrmions in the magnetic storage and magnetic logic community is connected with their potential as possible candidate to store the bits of future devices, with the information being encoded as presence/absence of a skyrmion; see, e.g., \cite{fcs2013,tmztcf2014} for the proposal of skyrmion racetrack memories, which are believed to overcome the original domain-wall-based device of~\cite{pht2008} and pave new ways in magnetic data logic~\cite{hkbb2015}.
\par
A well-accepted model for the magnetization dynamics is the Landau--Lifshitz--Gilbert equation (LLG)~\cite{ll1935,gilbert1955}.
The numerical approximation of LLG poses several challenges:
nonlinearities, a nonconvex pointwise constraint, an intrinsic energy law, which resembles the one of a gradient flow and combines conservative and dissipative effects, and the possible coupling with other partial differential equations (PDEs), e.g., the Maxwell equations.
\par
The numerical integration of LLG has been the subject of several mathematical studies; see, e.g., \cite{prohl2001,kp2006,garciacervera2007}.
A well-established approach is represented by the integrators usually referred to as tangent plane schemes.
These methods are based on equivalent reformulations of the equation in the tangent space.

The integrator proposed in~\cite{alouges2008a}, which considers the case in which the energy only comprises the exchange contribution, requires only the solution of one linear system per time-step, is formally of first order in time, and is unconditionally convergent towards a weak solution of the problem, i.e., the numerical analysis of the scheme does not require to impose any restrictive CFL-type coupling condition on the time-step size and the spatial mesh size.
The pointwise constraint is enforced by applying the nodal projection to the computed solution at each time-step.
The scheme generalizes the explicit scheme proposed in~\cite{aj2006} and analyzed in~\cite{bkp2008}.
Implicit-explicit approaches of the algorithm of~\cite{alouges2008a} for the full effective field were independently introduced and analyzed in~\cite{akt2012,bffgpprs2014}.
Extensions of the scheme for the discretization of the coupling of LLG with other PDEs were studied in~\cite{lt2013,bppr2013,lppt2015,bpp2015}.
Inspired by~\cite{bartels2016}, the projection-free version of the algorithm of~\cite{alouges2008a}, which avoids the use of the nodal projection, was introduced, analyzed, and applied to the decoupled integration of the coupling of LLG with a spin diffusion equation for the spin accumulation in~\cite{ahpprs2014}.
The violation of the constraint at the nodes of the mesh occurring in this case is uniformly controlled by the time-step size.
The projection-free tangent plane scheme of~\cite{ahpprs2014} was combined with a FEM-BEM coupling method for the discretization of the coupling of LLG with the magnetoquasistatic Maxwell equations in full space in~\cite{ft2017}.
There, assuming the existence of a unique sufficiently smooth solution, the authors proved optimal first-order convergence rates of the method.
A tangent plane scheme characterized by an enhanced convergence order in time was proposed in~\cite{akst2014}.
The method is unconditionally convergent and formally of (almost) second order in time.
A more efficient implicit-explicit version of this method has been proposed in~\cite{dpprs2017}.
Adapting ideas from~\cite{bp2006,alouges2008a}, the recent work~\cite{kw2018} proposes a similar predictor-corrector scheme based on a linear mass-lumped variational formulation of LLG.
\subsection{Contributions and general outline of the present work}
In this work, as a novel contribution, we introduce and analyze three tangent plane schemes for LLG in the presence of DMI.
The integrators extend to this case the first-order scheme of~\cite{alouges2008a} (Algorithm~\ref{alg:tps1}), its projection-free variant from~\cite{ahpprs2014} (Algorithm~\ref{alg:pftps1}), and the (almost) second-order scheme of~\cite{akst2014} (Algorithm~\ref{alg:tps2}).
For any algorithm, we prove that the sequence of finite element solutions, upon extraction of a subsequence, converges towards a weak solution of the problem.
For the projection-free algorithm, we prove that the convergence is even unconditional, while the stability analysis requires a mild CFL-type condition on the discretization parameters and a geometric restriction on the underlying mesh for the other two approaches.
The present extension of the LLG analysis is not straightforward, since the DMI term involves magnetization derivatives, is neither self-adjoint nor positive definite, and requires to impose different boundary conditions on LLG, which entail a careful treatment.
A by-product of our constructive analysis is the proof of existence of weak solutions, which to our knowledge was missing in the literature.
Finally, numerical experiments show that our approach can be used to study enucleation processes, stability, and dynamics of magnetic skyrmions.
\par
The remainder of the work is organized as follows:
For the convenience of the reader, we conclude this section by collecting the notation used throughout the paper.
In Section~\ref{sec:modeling}, we propose an organic presentation of the physical background and the mathematical framework of the problem under consideration.
In Section~\ref{sec:algorithms}, we derive three tangent plane schemes and state the convergence result (Theorem~\ref{thm:main}).
Section~\ref{sec:numerics} is devoted to numerical experiments.
Finally, in Section~\ref{sec:convergence}, we present the convergence analysis of the algorithms and, in particular, we establish the proof of Theorem~\ref{thm:main}.
\subsection{Notation}
We use the standard notation for Lebesgue, Sobolev, and Bochner spaces and norms; see, e.g., \cite[Chapter~5]{evans2010} or \cite[Chapter~2]{bbf2013}.
In the case of (spaces of) vector-valued or matrix-valued functions, we use bold letters, e.g., for any domain $U$, we denote both $L^2(U;\real^3)$ and $L^2(U;\real^{3 \times 3})$ by $\LL^2(U)$.
For the differential operators, we use the following notation:
For a scalar function $f$, we denote by $\grad f$ the gradient and by $\lapl f$ the Laplace operator.
For a vector-valued function $\ff$, we denote by $\div\ff$ the divergence, by $\curl\ff$ the curl, by $\Grad\ff$ the Jacobian, and by $\Lapl\ff$ the vector-valued Laplace operator.
Given another vector-valued function $\hh$, we also define $(\ff\cdot\grad)\hh$ by $[(\ff\cdot\grad)\hh]_i = \ff\cdot\grad h_i$ for all $1 \leq i \leq 3$.
We denote the unit sphere by $\sphere = \{ \vec{x} \in \real^3 : \abs{\vec{x}} = 1 \}$ and by $\{\vec{e}_i\}_{1 \leq i \leq 3} \in \real^3$ the standard basis of $\real^3$, i.e., $(\vec{e}_i)_j = \delta_{ij}$ for all $1 \leq i,j \leq 3$.
Given a vector $\vec{b} \in \real^3$ and a matrix $\vec{A} \in \real^{3 \times 3}$ (with columns $\vec{a}_i \in \real^3$ for all $1 \leq i \leq 3$), we denote by $\vec{A}\times\vec{b} \in \real^{3 \times 3}$ the matrix whose columns are $\vec{a}_i \times \vec{b}$ for all $1 \leq i \leq 3$.
By $C>0$ we always denote a generic constant, which is independent of the discretization parameters, but not necessarily the same at each occurrence.
We also use the notation $\lesssim$ to denote \emph{smaller than or equal to up to a multiplicative constant}, i.e., we write $A \lesssim B$ if there exists a constant $C>0$, which is clear from the context and always independent of the discretization parameters, such that $A \leq C B$.
\section{Mathematical model} \label{sec:modeling}
\subsection{Physical background}
Let $\Omega\subset\real^3$ be a bounded domain with boundary $\Gamma := \partial\Omega$.
The dynamics of the normalized magnetization $\mm = (m_1,m_2,m_3) \in \sphere$ is governed by LLG, which in the so-called Gilbert form reads 
\begin{equation} \label{eq:llg:physical}
\mmt
= -\gamma_0 \, \mm \times \Heff(\mm)
+ \alpha \, \mm \times \mmt;
\end{equation}
see~\cite{ll1935,gilbert1955}.
Here, $\gamma_0 \approx$ \SI{2.21e5}{\meter\per\ampere\per\second} is the rescaled gyromagnetic ratio of the electron, $0<\alpha\leq 1$ is the dimensionless Gilbert damping parameter, and $\Heff$ is the energy-based effective field (in \si{\ampere\per\meter}), i.e., it holds that
\begin{equation} \label{eq:functionalDerivative}
\mu_0 \Ms \, \Heff(\mm) = - \frac{\delta\E(\mm)}{\delta \mm},
\end{equation}
where $\E(\cdot)$ is the total energy, $\mu_0 =$ \SI{4 \pi e-7}{\newton\per\square\ampere} is the vacuum permeability, and $\Ms>0$ is the saturation magnetization (in \si{\ampere\per\meter}).
In micromagnetics, the total energy is usually the sum of the following standard terms:
\begin{itemize}
\item the Heisenberg exchange contribution
\begin{equation*}
\E_{\mathrm{ex}} (\mm) = A \int_{\Omega}\abs{\Grad\mm}^2 \dx,
\end{equation*}
where $A>0$ denotes the exchange stiffness constant (in~\si{\joule\per\meter});
\item the magnetocrystalline anisotropy contribution, which for the uniaxial case reads
\begin{equation*}
\E_{\mathrm{ani}}(\mm) = K \int_{\Omega}\left[1-(\aa\cdot\mm)^2\right] \dx,
\end{equation*}
where $K>0$ denotes the anisotropy constant (in~\si{\joule\per\meter\cubed}) and $\aa\in\sphere$ is the easy axis,
\item the Zeeman contribution
\begin{equation*}
\E_{\mathrm{ext}}(\mm) = - \mu_0 \Ms \int_{\Omega}\Hext\cdot\mm \, \dx,
\end{equation*}
where $\Hext$ denotes an applied external field (in~\si{\ampere\per\meter});
\item the magnetostatic contribution
\begin{equation*}
\E_{\mathrm{mag}}(\mm)
= \frac{\mu_0}{2} \int_{\real^3} \abs{\Hstray(\mm)}^2 \dx,
\end{equation*}
where $\Hstray(\mm) = - \grad u$ denotes the stray field (in~\si{\ampere\per\meter}), with $u$ being the magnetostatic potential (in~\si{\ampere}), which solves the full-space transmission problem
\begin{subequations} \label{eq:transmission}
\begin{alignat}{2}
- \lapl u^{\mathrm{int}} &= - \Ms \, \div \mm &\quad& \text{in } \Omega, \\
- \lapl u^{\mathrm{ext}} &= 0 && \text{in } \real^3\setminus\overline{\Omega}, \\
u^{\mathrm{ext}}-u^{\mathrm{int}} &= 0 && \text{on } \Gamma, \\
(\grad u^{\mathrm{ext}}- \grad u^{\mathrm{int}}) \cdot \nn &= - \Ms \, \mm \cdot \nn && \text{on } \Gamma, \\
u(\xx) &= \bigO{1/\abs{\xx}} && \text{as } \abs{\xx}\to\infty.
\end{alignat}
\end{subequations}
Here, $\nn:\Gamma\to\sphere$ denotes the outward-pointing unit normal vector to $\Gamma$.
\end{itemize}
In this work, the total energy $\E(\cdot)$ in~\eqref{eq:functionalDerivative} also comprises a term associated with the DMI~\cite{dzyaloshinskii1958,moriya1960}.
This energy contribution is phenomenologically introduced for systems with a broken symmetry due to different interface crystal configurations as a linear combination of the so-called Lifshitz invariants, i.e., the components
of the chirality tensor $\vec{A} = (a_{ij})_{1 \leq i,j \leq 3} = \Grad\mm \times \mm$; see~\cite{by1989,bh1994}.
The choice of the appropriate DMI form depends on the crystal structure of the material and on the geometry of the sample under consideration.
This choice in turn determines the specific expression of the effective field (according to~\eqref{eq:functionalDerivative}) and the boundary conditions on $\Gamma$, which are chosen in agreement with those satisfied by the solution of the Euler--Lagrange equations associated with the energy minimization problem
\begin{equation} \label{eq:minimization}
\min_{\abs{\mm}=1} \E(\mm).
\end{equation}
In micromagnetics, two main DMI forms are usually considered:
\begin{itemize}
\item For helimagnetic materials~\cite{moriya1960}, the DMI is obtained by taking as Lifshitz invariants the trace of the matrix $\vec{A}$, i.e.,
\begin{equation*}
\trace\vec{A} = \sum_{1 \leq i \leq 3} a_{ii} = (\curl\mm)\cdot\mm.
\end{equation*}
The so-called \emph{bulk} DMI energy contribution then takes the form
\begin{equation} \label{eq:bulkDMI}
\E_{\mathrm{bDMI}}(\mm)
= D \int_{\Omega} (\curl\mm)\cdot\mm \, \dx,
\end{equation}
so that the resulting effective field term and boundary conditions on $\Gamma$ are given by
\begin{equation*}
\HH_{\mathrm{eff,bDMI}}(\mm)
= - \frac{2D}{\mu_0 \Ms} \curl\mm
\quad \text{and} \quad
2A \, \de_{\nn} \mm + D \, \mm\times\nn = \0.
\end{equation*}
\item Another type of DMI is due to the interfaces between different materials which break inversion symmetry~\cite{cl1998}.
For a magnetic thin film aligned with the $x_1x_2$-plane, the Lifshitz invariants are given by
\begin{equation*}
a_{12} - a_{21}
= m_3(\de_1 m_1 + \de_2 m_2) - (m_1 \, \de_1 m_3 + m_2 \, \de_2 m_3).
\end{equation*}
The so-called \emph{interfacial} DMI energy contribution thus takes the form
\begin{equation} \label{eq:interDMI}
\E_{\mathrm{iDMI}}(\mm)
= D \int_{\Omega} [m_3(\de_1 m_1 + \de_2 m_2) - (m_1 \, \de_1 m_3 + m_2 \, \de_2 m_3)] \, \dx,
\end{equation}
so that the resulting effective field term and boundary conditions on $\Gamma$ are given by
\begin{equation*}
\HH_{\mathrm{eff,iDMI}}(\mm)
= - \frac{2D}{\mu_0 \Ms}
\begin{pmatrix}
- \de_1 m_3 \\
- \de_2 m_3 \\
\de_1 m_1 + \de_2 m_2
\end{pmatrix}
\quad \text{and} \quad
2A \, \de_{\nn} \mm + D (\vec{e}_3 \times \nn) \times \mm = \0.
\end{equation*}
\end{itemize}
In~\eqref{eq:bulkDMI}--\eqref{eq:interDMI}, the constant $D \in \real$ is the DMI constant (in \si{\joule\per\square\meter})\footnote{Consider the energy $w_D$ in~\cite[equation~(4)]{bh1994} for the crystallographic class $C_n$ ($n=3,4,6$). The bulk DMI energy~\eqref{eq:bulkDMI} corresponds to the last two terms of $w_D$, i.e., $D_1 = 0$ and $D_2 = D_3 = -D$ (crystallographic subclass $D_n$). The interfacial DMI energy~\eqref{eq:interDMI} corresponds to the first term of $w_D$, i.e., $D_1=D$ and $D_2 = D_3 = 0$ (crystallographic subclass $C_{nv}$).}.
The sign of $D$ determines the chirality of the system, which, in the case of a skyrmion state, defines the sense of rotation of the magnetization along the skyrmion diameter~\cite{scrtf2013,kp2015a}.
\subsection{Problem formulation} \label{sec:model}
To simplify the notation, in our analysis, we restrict ourselves to the case in which the energy solely comprises exchange and DMI.
We refer to~\cite{akt2012,bffgpprs2014,ahpprs2014,akst2014,dpprs2017} for the design and the analysis of effective tangent plane integrators for the standard energy terms (exchange, anisotropy, Zeeman, and magnetostatic).
Moreover, without loss of generality, we assume that $D>0$ and consider the prototypical case of bulk DMI, because it is characterized by a short notation involving the curl operator.
The same approach and the same results hold for all possible choices of the Lifshitz invariants.
\par
After a suitable nondimensionalization\footnote{We rescale the time according to the transformation $t'=\gamma_0 \Ms t$. We define the rescaled effective field by $\heff=\Heff/\Ms$ and the rescaled energy by $\E'=\E/(\mu_0 \Ms^2)$. However, to simplify the notation, we neglect all $'$-superscripts.}, the initial boundary value problem in which we are interested takes the form
\begin{subequations} \label{eq:llg:IBVP}
\begin{alignat}{2}
\label{eq:llg:IBVP:llg}
\mmt & = - \mm \times \heff(\mm) + \alpha \, \mm \times \mmt &\quad& \text{in } \Omega\times(0,\infty), \\
\label{eq:llg:IBVP:bc}
2 \lex^2 \, \de_{\nn} \mm & = - \ldm \, \mm\times\nn  && \text{on } \Gamma\times(0,\infty), \\
\mm(0) & = \mm^0 && \text{in } \Omega,
\end{alignat}
\end{subequations}
where the effective field $\heff(\mm)$ is determined by the energy functional
\begin{equation} \label{eq:llg:energy}
\E(\mm)
= \frac{\lex^2}{2} \norm[\LL^2(\Omega)]{\Grad\mm}^2
+ \frac{\ldm}{2} \inner{\curl\mm}{\mm}
\end{equation}
according to the relation
\begin{equation} \label{eq:llg:effective}
\heff(\mm) = - \frac{\delta\E(\mm)}{\delta \mm} = \lex^2 \Lapl\mm - \ldm \curl\mm.
\end{equation}
The positive quantities $\lex=\sqrt{2A/(\mu_0 \Ms^2)}$ and $\ldm=2D/(\mu_0 \Ms^2)$ denote the exchange length and the DMI length (both measured in~\si{\meter}), respectively.
\par
Since $\norm[\LL^2(\Omega)]{\curl\mm} \leq \sqrt{2} \norm[\LL^2(\Omega)]{\Grad\mm}$, using the weighted Young inequality
\begin{equation} \label{eq:young}
ab \leq \frac{\eps a^2}{2} + \frac{b^2}{2 \eps} \quad \text{for any } a,b \in \real \text{ and } \eps>0,
\end{equation}
it is easy to see that the energy~\eqref{eq:llg:energy} satisfies the condition
\begin{equation} \label{eq:energyNormEquivalence}
\frac{\lex^2}{4} \norm[\LL^2(\Omega)]{\Grad\mm}^2 - \frac{\ldm^2}{2 \lex^2} \norm[\LL^2(\Omega)]{\mm}^2
\leq \E(\mm)
\leq \frac{\lex^2 + \ldm^2}{2} \norm[\LL^2(\Omega)]{\Grad\mm}^2 + \frac{1}{4} \norm[\LL^2(\Omega)]{\mm}^2.
\end{equation}
For any $T>0$, we define the space-time cylinder by $\Omega_T := \Omega \times (0,T)$.
Moreover, we denote by $\inner{\cdot}{\cdot}$ the scalar product in $\LL^2(\Omega)$, by $\edual{\cdot}{\cdot}$ the duality pairing between $\HH^{-1/2}(\Gamma)$ and $\HH^{1/2}(\Gamma)$, and by $\gamma_T: \Hcurl{\Omega} \to \HH^{-1/2}(\Gamma)$ the tangential trace operator, which satisfies $\gamma_T[\uu] = \uu\times\nn\vert_{\Gamma}$ for any smooth function $\uu$ as well as the Green formula
\begin{equation} \label{eq:green}
\edual{\gamma_T[\uu]}{\pphi}
= \inner{\uu}{\curl\pphi} - \inner{\curl\uu}{\pphi}
\quad \text{for all } \uu \in \Hcurl{\Omega} \text{ and } \pphi \in \HH^1(\Omega);
\end{equation}
see, e.g., \cite[Lemma~2.1.4]{bbf2013}.
\par
We conclude this section by extending the notion of a weak solution introduced in~\cite{as1992} to the present setting.
\begin{definition} 
Let $\mm^0 \in \HH^1(\Omega)$ satisfy $\abs{\mm}=1$ a.e.\ in $\Omega$.
A vector field $\mm:\Omega \times (0,\infty) \to \real^3$ is called a weak solution of~\eqref{eq:llg:IBVP} if, for any $T>0$, the following properties are satisfied:
\begin{itemize}
\item[\rm(i)] $\mm\in \HH^1(\Omega_T) \cap L^{\infty}(0,T;\HH^1(\Omega))$ with $\abs{\mm}=1$ a.e.\ in $\Omega_T$;
\item[\rm(ii)] $\mm(0)=\mm^0$ in the sense of traces;
\item[\rm(iii)] For all $\vvphi\in\HH^1(\Omega_T)$, it holds that
\begin{equation} \label{eq:weak:variational}
\begin{split}
& \int_0^T \inner{\mmt(t)}{\vvphi(t)} \, \dt \\
& \quad = \lex^2 \int_0^T \inner{\mm(t)\times\Grad\mm(t)}{\Grad\vvphi(t)} \, \dt
- \ldm \int_0^T \inner{\curl\mm(t)}{\mm(t)\times\vvphi(t)} \, \dt \\
& \qquad - \frac{\ldm}{2} \int_0^T \edual{\gamma_T[\mm(t)]}{\mm(t)\times\vvphi(t)} \, \dt
+ \alpha \int_0^T \inner{\mm(t)\times\mmt(t)}{\vvphi(t)} \, \dt;
\end{split}
\end{equation}
\item[\rm(iv)] It holds that
\begin{equation} \label{eq:weak:energyLaw}
\E(\mm(T)) + \alpha \int_0^T \norm[\LL^2(\Omega)]{\mmt(t)}^2 \dt
\leq \E(\mm^0).
\end{equation}
\end{itemize}
\end{definition}
The variational formulation~\eqref{eq:weak:variational} comes from a weak formulation of~\eqref{eq:llg:IBVP:llg} in the space-time domain.
The boundary conditions~\eqref{eq:llg:IBVP:bc} are enforced as natural boundary conditions.
In particular, the term with $\edual{\cdot}{\cdot}$ arises from integrating by parts the exchange contribution and using~\eqref{eq:llg:IBVP:bc}.
The energy inequality~\eqref{eq:weak:energyLaw} is a weak counterpart of the dissipative energy law
\begin{equation*}
\ddt\E(\mm(t)) = - \alpha \norm[\LL^2(\Omega)]{\mmt(t)}^2 \leq 0
\quad \text{for all } t>0
\end{equation*}
satisfied by any sufficiently smooth solution of~\eqref{eq:llg:IBVP}.
\begin{remark}
Taking the scalar product of~\eqref{eq:llg:IBVP:llg} with $\mm$, we deduce that $\mm\cdot\mmt = 0$.
In particular, since $\de_t \abs{\mm}^2 = 2 \, \mm\cdot\mmt= 0$, it follows that a sufficiently smooth solution of~\eqref{eq:llg:IBVP:llg} satisfies the constraint $\abs{\mm}=1$, provided that it is satisfied by the initial condition.
\end{remark}
\section{Numerical algorithms and main result} \label{sec:algorithms}
In this section, we introduce three algorithms for the numerical approximation of the problem discussed in Section~\ref{sec:model} and we state the main convergence result.
\subsection{Preliminaries}
For the time discretization, given an integer $N>0$ and a final time $T>0$, we consider a uniform partition of the time interval $(0,T)$ with time-step size $k := T/N$, i.e., $t_i := ik$ for all $0 \leq i \leq N$.
For the spatial discretization, we assume $\Omega$ to be a polyhedral domain with Lipschitz boundary and consider a $\kappa$-quasi-uniform family $\{ \Th \}_{h>0}$ of regular tetrahedral meshes of $\Omega$ parametrized by the mesh size $h>0$, i.e., there exists $\kappa \geq 1$, independent of $h$, such that $\Th$ is $\kappa$-shape-regular and $\kappa^{-1} h \leq \diam(K)$ for all $K \in \Th$.
We denote by $\Nh$ the set of vertices of $\Th$.
For any $K \in \Th$, we denote by $\mathcal{P}^1(K)$ the space of linear polynomials on $K$.
We consider the space
\begin{equation*}
\mathcal{S}^1(\Th) = \left\{v_h \in C^0(\overline{\Omega}): v_h \vert_K \in \mathcal{P}^1(K) \text{ for all } K \in \Th \right\}
\end{equation*}
of piecewise linear and globally continuous functions from $\Omega$ to $\real$.
The classical basis for this finite-dimensional linear space is given by the set of the nodal hat functions $\left\{\vphi_{\zz}\right\}_{\zz\in\Nh}$, which satisfy $\vphi_{\zz}(\zz')=\delta_{\zz,\zz'}$ for all $\zz,\zz'\in\Nh$.
We assume that all off-diagonal entries of the so-called stiffness matrix are nonpositive, i.e.,
\begin{equation} \label{eq:angleCondition}
\inner{\grad\vphi_{\zz}}{\grad\vphi_{\zz'}} \leq 0
\quad \text{for all } \zz, \zz' \in \N_h \text{ with } \zz \neq \zz'.
\end{equation}
This requirement is usually referred to as \emph{angle condition}, since it is satisfied if the measure of all dihedral angles of all tetrahedra of the mesh is less than or equal to $\pi/2$.
\par
Any solution of LLG is characterized by the nonconvex pointwise constraint $\abs{\mm}=1$ and by the orthogonality property $\mm\cdot\mmt = 0$.
To mimic these properties at the discrete level, we require them to be satisfied only at the nodes of the mesh.
To this end, we introduce the \emph{set of admissible discrete magnetizations}
\begin{equation*}
\Mh := \left\{\pphih \in \mathcal{S}^1(\Th)^3: \abs{\pphih(\zz)}=1 \text{ for all } \zz \in \Nh \right\}
\end{equation*}
and, for $\ppsih \in \mathcal{S}^1(\Th)^3$, the linear space
\begin{equation} \label{eq:discreteTangentSpace}
\Kh(\ppsih) := \left\{\pphih \in \mathcal{S}^1(\Th)^3 : \ppsih(\zz) \cdot \pphih(\zz) = 0 \text{ for all } \zz \in \Nh \right\},
\end{equation}
which we call the \emph{discrete tangent space} of $\ppsih$.
\subsection{Three tangent plane integrators}
Using the well-known formula
\begin{equation*}
\vec{a}\times(\vec{b}\times\vec{c})=(\vec{a}\cdot\vec{c})\vec{b} - (\vec{a}\cdot\vec{b})\vec{c} \quad \text{for all } \vec{a},\vec{b},\vec{c} \in\real^3,
\end{equation*}
\eqref{eq:llg:IBVP:llg} can be formally rewritten in the form
\begin{equation} \label{eq:llg:alternative}
\alpha \, \mmt
+ \mm\times\mmt
= \heff(\mm) - [\heff(\mm)\cdot\mm]\mm.
\end{equation}
Observing that this equation is linear with respect to the time derivative $\mmt$, we introduce the free variable $\vv=\mmt$.
For any $t \in (0,T)$, $\vv(t)$ belongs to the tangent space of $\sphere$ at $\mm(t)$.
Taking this orthogonality and the expression~\eqref{eq:llg:effective} of the effective field into account, we obtain the following variational formulation:
Find $\vv(t) \in \LL^2(\Omega)$ with $\mm(t)\cdot\vv(t)=0$ a.e.\ in $\Omega$ such that
\begin{equation} \label{eq:tps1continuous}
\begin{split}
& \alpha \inner{\vv(t)}{\pphi}
+ \inner{\mm(t)\times\vv(t)}{\pphi} \\
& \quad = - \lex^2 \inner{\Grad\mm(t)}{\Grad\pphi}
- \frac{\ldm}{2} \inner{\curl\mm(t)}{\pphi}
- \frac{\ldm}{2} \inner{\mm(t)}{\curl\pphi}
\end{split}
\end{equation}
for all $\pphi\in \HH^1(\Omega)$ satisfying $\mm(t)\cdot\pphi=0$ a.e.\ in $\Omega$.
To obtain~\eqref{eq:tps1continuous}, the boundary integral which arises from integrating by parts the exchange contribution and using the boundary conditions~\eqref{eq:llg:IBVP:bc} is rewritten as a volume integral by using~\eqref{eq:green}.
Note that, since the test function $\pphi$ belongs to the tangent space of the sphere at $\mm(t)$, in~\eqref{eq:tps1continuous} the term corresponding to the last term (strongly nonlinear in $\mm$) on the right-hand side of~\eqref{eq:llg:alternative} vanishes.
\par
For any time-step $0 \leq i \leq N-1$, given the approximate current magnetization $\mmh^i \approx \mm(t_i)$, we compute $\vvh^i \approx \vv(t_i)$ by a Galerkin discretization of~\eqref{eq:tps1continuous} based on the discrete tangent space $\Kh(\mmh^i)$ introduced in~\eqref{eq:discreteTangentSpace}.
The computed quantity $\vvh^i \in \Kh(\mmh^i)$ is then used to update the current magnetization $\mmh^i \approx \mm(t_i)$ to the new value $\mmh^{i+1} \approx \mm(t_{i+1})$ via a first-order time-stepping.
To ensure that the discrete magnetization belongs to the set of admissible discrete magnetizations $\Mh$, the nodal projection is applied.
\par
The resulting scheme, summarized in the following algorithm, extends the method proposed by~\cite{alouges2008a} to the present situation.
\begin{algorithm}[first-order tangent plane scheme, TPS1] \label{alg:tps1}
Input:
$\mmh^0 \in \Mh$. \\
Loop:
For all $0 \leq i \leq N-1$, iterate:
\begin{enumerate} [label={\rm(\roman*)}]
\item Compute $\vvh^i \in \Kh(\mmh^i)$ such that, for all $\pphih \in \Kh(\mmh^i)$, it holds that
\begin{equation} \label{eq:tps1}
\begin{split}
& \alpha \inner{\vvh^i}{\pphih}
+ \inner{\mmh^i\times\vvh^i}{\pphih}
+ \lex^2 \theta k \inner{\Grad\vvh^i}{\Grad\pphih} \\
& \quad = - \lex^2 \inner{\Grad\mmh^i}{\Grad\pphih}
- \frac{\ldm}{2} \inner{\curl\mmh^i}{\pphih}
- \frac{\ldm}{2} \inner{\mmh^i}{\curl\pphih}.
\end{split}
\end{equation}
\item Define $\mmh^{i+1} \in \Mh$ by $\displaystyle \mmh^{i+1}(\zz) := \frac{\mmh^i(\zz) + k \vvh^i(\zz)}{\abs{\mmh^i(\zz) + k \vvh^i(\zz)}}$ for all $\zz\in\Nh$.
\label{item:tps1-2}
\end{enumerate}
Output:
Sequence of discrete functions $\left\{(\vvh^i,\mmh^{i+1})\right\}_{0 \leq i \leq N-1}$.
\end{algorithm}
In~\eqref{eq:tps1}, the parameter $0 \leq \theta \leq 1$ modulates the `degree of implicitness' of the method in the treatment of the leading-order exchange contribution of the effective field.
\par
In the following algorithm, we state a projection-free variant of Algorithm~\ref{alg:tps1}, where step~\ref{item:tps1-2} is replaced by a simple linear first-order time-stepping.
Note that, omitting the nodal projection, the pointwise constraint $\abs{\mm}=1$ is not explicitly enforced by the numerical scheme.
\begin{algorithm}[projection-free first-order tangent plane scheme, PF-TPS1] \label{alg:pftps1}
Input:
$\mmh^0 \in \mathcal{S}^1(\Th)^3$. \\
Loop:
For all $0 \leq i \leq N-1$, iterate:
\begin{enumerate} [label={\rm(\roman*)}]
\item Compute $\vvh^i \in \Kh(\mmh^i)$ such that, for all $\pphih \in \Kh(\mmh^i)$, it holds that
\begin{equation} \label{eq:pftps1}
\begin{split}
& \alpha \inner{\vvh^i}{\pphih}
+ \inner{\mmh^i\times\vvh^i}{\pphih}
+ \lex^2 \theta k \inner{\Grad\vvh^i}{\Grad\pphih} \\
& \quad = - \lex^2 \inner{\Grad\mmh^i}{\Grad\pphih}
- \frac{\ldm}{2} \inner{\curl\mmh^i}{\pphih}
- \frac{\ldm}{2} \inner{\mmh^i}{\curl\pphih}.
\end{split}
\end{equation}
\item Define $\mmh^{i+1} := \mmh^i + k \vvh^i \in \mathcal{S}^1(\Th)^3$.
\label{item:pftps1-2}
\end{enumerate}
Output:
Sequence of discrete functions $\left\{(\vvh^i,\mmh^{i+1})\right\}_{0 \leq i \leq N-1}$.
\end{algorithm}
The idea of removing the nodal projection from the tangent plane scheme goes back to~\cite{ahpprs2014} for LLG and has been inspired by~\cite{bartels2016}, where the same principle is applied to a certain class of geometrically constrained PDEs, e.g., the harmonic map heat flow.
\par
In~\cite{akst2014}, the authors extend the tangent plane scheme of~\cite{alouges2008a} to improve the formal convergence order in time of the method.
If the tangential update $\vv$ takes the form
\begin{equation} \label{eq:tps2:v}
\vv(t) = \mmt(t) + \frac{k}{2} \mathbb{P}_{\mm(t)} [\partial_{tt}\mm(t)],
\end{equation}
where $\mathbb{P}_{\uu} : \real^3 \to \spann(\uu)^{\perp}$ denotes the orthogonal projection onto $\spann(\uu)^{\perp}$ for any $\uu\in\sphere$, a sufficiently smooth solution $\mm$ of LLG satisfies the Taylor expansion
\begin{equation} \label{eq:tps2:normalization}
\mm(t+k) = \frac{\mm(t) + k \vv(t)}{\abs{\mm(t) + k \vv(t)}} + \mathcal{O}(k^3).
\end{equation}
Differentiating~\eqref{eq:llg:alternative} with respect to time and proceeding as in~\cite[Section~6]{akst2014}, one obtains that, up to a residual term of order $\mathcal{O}(k^2)$, the tangential update~\eqref{eq:tps2:v} can be characterized as the solution of the following variational problem:
Find $\vv(t) \in \HH^1(\Omega)$ with $\mm(t)\cdot\vv(t)=0$ a.e.\ in $\Omega$ such that
\begin{equation} \label{eq:tps2continuous}
\begin{split}
& \alpha \inner{\vv(t)}{\pphi}
+ \frac{1}{2} k \inner{[-\lex^2 \abs{\Grad\mm(t)}^2 - \ldm(\curl\mm(t))\cdot\mm(t)]\vv(t)}{\pphi} \\
& \ + \inner{\mm(t)\times\vv(t)}{\pphi}
+ \frac{\lex^2}{2} k \inner{\Grad\vv(t)}{\Grad\pphi}
+ \frac{\ldm}{4} k \inner{\curl\vv(t)}{\pphi}
+ \frac{\ldm}{4} k \inner{\vv(t)}{\curl\pphi}
\\
& \quad = - \lex^2 \inner{\Grad\mm(t)}{\Grad\pphi}
- \frac{\ldm}{2} \inner{\curl\mm(t)}{\pphi}
- \frac{\ldm}{2} \inner{\mm(t)}{\curl\pphi}
\end{split}
\end{equation}
for all $\pphi\in \HH^1(\Omega)$ satisfying $\mm(t)\cdot\pphi=0$ a.e.\ in $\Omega$.
To obtain an effective numerical method, we use the same predictor-corrector approach used for Algorithms~\ref{alg:tps1}--\ref{alg:pftps1}:
For any time-step $0 \leq i \leq N-1$, given the approximation $\mmh^i \approx \mm(t_i)$, we compute $\vvh^i \approx \vv(t_i)$ by a Galerkin discretization of~\eqref{eq:tps2continuous} based on $\Kh(\mmh^i)$.
Then, with~\eqref{eq:tps2:normalization} in mind, we define $\mmh^{i+1} \approx \mm(t_{i+1})$ in $\Mh$ as the nodal projection of $\mmh^i + k \vvh^i$.
\par
However, in order to obtain a well-defined scheme, following~\cite{akst2014}, we perform two higher-order modifications of~\eqref{eq:tps2continuous}.
Firstly, to ensure the well-posedness of the variational problem, we proceed as follows:
Given $M>0$, we define the cut-off function $W_M: \real \to \real$ by
\begin{equation*}
W_M(s) =
\begin{cases}
\displaystyle
\alpha + k \min\{ s, M \} / 2 & \text{if } s \geq 0, \\
\displaystyle
2\alpha^2 / (2 \alpha + k \min\{ -s, M \}) & \text{if } s < 0.
\end{cases}
\end{equation*}
By construction, it holds that
\begin{equation} \label{eq:propertyCutOff}
W_M(s) \geq 2 \alpha^2 / (2 \alpha + M k)
\quad \text{and} \quad
\big\lvert W_M(s) - \alpha\big\rvert \leq M k/2
\quad \text{for all } s \in \real;
\end{equation}
see, e.g., \cite[Lemma~12]{dpprs2017}.
In the variational formulation~\eqref{eq:tps2continuous}, we then replace
\begin{equation*}
\alpha \inner{\vv(t)}{\pphi}
+ \frac{1}{2} k \inner{[-\lex^2 \abs{\Grad\mm(t)}^2 - \ldm(\curl\mm(t))\cdot\mm(t)]\vv(t)}{\pphi}
\text{ by }
\inner{W_M(\lambda(\mm(t)))\vv(t)}{\pphi},
\end{equation*}
where
\begin{equation*}
\lambda(\mm)
= \heff(\mm)\cdot\mm
= -\lex^2 \abs{\Grad\mm}^2 - \ldm(\curl\mm)\cdot\mm.
\end{equation*}
Note that the function $\lambda(\mm)$ is also the Lagrange multiplier associated with the constraint $\abs{\mm}=1$ in the constrained minimization problem~\eqref{eq:minimization}.
If $M>0$ is sufficiently large, this modification introduces a consistency error of order $\mathcal{O}(k^2)$ in~\eqref{eq:tps2continuous}.
In particular, to ensure this, we define $M: \real_{>0} \to \real_{>0}$ by $M(k) := \abs{k \log k}^{-1}$ for all $k>0$.
Note that $M$ satisfies the convergence properties
\begin{equation} \label{eq:propertyOfM}
M(k) \to \infty \quad \text{and} \quad M(k)k \to 0 \quad \text{as } k \to 0.
\end{equation}
Secondly, in the variational formulation~\eqref{eq:tps2continuous}, we replace
\begin{equation*}
\frac{\lex^2}{2} k \inner{\Grad\vv(t)}{\Grad\pphi}
\quad \text{by} \quad
\frac{\lex^2}{2} [1 + \rho(k)] k \inner{\Grad\vv(t)}{\Grad\pphi},
\end{equation*}
where the stabilization function $\rho: \real_{>0} \to \real_{>0}$ is defined by $\rho(k) := \abs{k \log k}$ for all $k>0$.
This artificial stabilization introduces a formal consistency error of order $\mathcal{O}(k^{2-\varepsilon})$ for any $0<\varepsilon<1$.
\par
In the following algorithm, we summarize the proposed extension of the tangent plane scheme of~\cite{akst2014} to the present setting.
\begin{algorithm}[(almost) second-order tangent plane scheme, TPS2] \label{alg:tps2}
Input:
$\mmh^0 \in \Mh$. \\
Loop:
For all $0 \leq i \leq N-1$, iterate:
\begin{enumerate} [label={\rm(\roman*)}]
\item Set $\lambda_h^i = -\lex^2 \abs{\Grad\mmh^i}^2 - \ldm(\curl\mmh^i)\cdot\mmh^i$.
\item Compute $\vvh^i \in \Kh(\mmh^i)$ such that, for all $\pphih \in \Kh(\mmh^i)$, it holds that
\begin{equation} \label{eq:tps2}
\begin{split}
& \inner{W_{M(k)}(\lambda_h^i)\vvh^i}{\pphih}
+ \inner{\mmh^i\times\vvh^i}{\pphih}
+ \frac{\lex^2}{2} k [1 + \rho(k)] \inner{\Grad\vvh^i}{\Grad\pphih} \\
& \quad + \frac{\ldm}{4} k \inner{\vvh^i}{\curl\pphih}
+ \frac{\ldm}{4} k \inner{\curl\vvh^i}{\pphih} \\
& \qquad = - \lex^2 \inner{\Grad\mmh^i}{\Grad\pphih}
- \frac{\ldm}{2} \inner{\curl\mmh^i}{\pphih}
- \frac{\ldm}{2} \inner{\mmh^i}{\curl\pphih}.
\end{split}
\end{equation}
\item Define $\mmh^{i+1} \in \Mh$ by $\displaystyle \mmh^{i+1}(\zz) := \frac{\mmh^i(\zz) + k \vvh^i(\zz)}{\abs{\mmh^i(\zz) + k \vvh^i(\zz)}}$ for all $\zz\in\Nh$.
\end{enumerate}
Output:
Sequence of discrete functions $\left\{(\vvh^i,\mmh^{i+1})\right\}_{0 \leq i \leq N-1}$.
\end{algorithm}
For the proof that the three proposed algorithms are well-posed (if the time-step size is sufficiently small in the case of TPS2), we refer to Proposition~\ref{prop:wellposed} below.
\begin{remark}
The natural starting point for a hypothetical projection-free version of TPS2 would be the expansion
\begin{equation*}
\mm(t+k) = \mm(t) + k \vv(t) + \mathcal{O}(k^3),
\end{equation*}
for which a \emph{nontangential} update of the form
\begin{equation*}
\vv(t) = \mmt(t) + \frac{k}{2} \partial_{tt}\mm(t)
\end{equation*}
would be required.
In particular, it is not clear how to apply the tangent plane paradigm to this situation, where the update $\vv(t)$ has a nonzero component parallel to $\mm(t)$ in general, i.e., $\mm(t) \cdot \vv(t) = \frac{k}{2} \partial_{tt}\mm(t) \cdot \mm(t) \neq 0$, which needs to be taken into account in order to achieve a second-order accuracy.
\end{remark}
\subsection{Convergence result}
From any algorithm, we obtain two sequences of discrete functions $\{\mmh^i\}_{0 \leq i \leq N}$ and $\{\vvh^i\}_{0 \leq i \leq N-1}$.
We define the piecewise linear time reconstruction $\mmhk$ and the piecewise constant time reconstructions $\mmhk^\pm$ and $\vvhk^-$ by
\begin{equation} \label{eq:timeApprox}
\begin{split}
& \mmhk(t) := \frac{t-t_i}{k}\mmh^{i+1} + \frac{t_{i+1} - t}{k}\mmh^i, \\
& \mmhk^-(t) := \mmh^i,
\quad
\mmhk^+(t) := \mmh^{i+1},
\quad
\text{and}
\quad
\vvhk^-(t) := \vvh^i
\end{split}
\end{equation}
for all $0 \leq i \leq N-1$ and $t \in [t_i,t_{i+1})$.
The following theorem, which is the main result of the paper, states that the time reconstructions $\mmhk$ obtained by the three algorithms converge in an appropriate sense towards a weak solution of~\eqref{eq:llg:IBVP} as $h,k \to 0$.
\begin{theorem} \label{thm:main}
Let the approximate initial condition satisfy the convergence property
\begin{equation} \label{eq:convergenceMh0}
\mmh^0 \to \mm^0 \quad \text{in } \HH^1(\Omega) \quad \text{as } h \to 0.
\end{equation}
Moreover, for each algorithm, consider the following specific assumptions:
\begin{itemize}
\item For TPS1 (Algorithm~\ref{alg:tps1}), assume that the angle condition~\eqref{eq:angleCondition} is satisfied, that $1/2 \leq \theta \leq 1$, and that it holds that $k/h \to 0$ as $h,k \to 0$.
\item For PF-TPS1 (Algorithm~\ref{alg:pftps1}), assume that $1/2 < \theta \leq 1$.
\item For TPS2 (Algorithm~\ref{alg:tps2}), assume that the angle condition~\eqref{eq:angleCondition} is satisfied and that it holds that $k/h \to 0$ as $h,k \to 0$.
\end{itemize}
Then, for each algorithm, there exist a weak solution $\mm$ of~\eqref{eq:llg:IBVP} and a subsequence of $\{ \mmhk \}$ which converges weakly in $\HH^1(\Omega_T)$ towards $\mm$ as $h,k \to 0$.
\end{theorem}
\begin{remark} \label{rem:mainThm}
{\rm (i)} For the sake of brevity, we have considered the case in which the energy consists of only exchange and DMI.
Adopting the implicit-explicit approaches of~\cite{akt2012,bffgpprs2014,ahpprs2014,dpprs2017}, the schemes and the convergence result of Theorem~\ref{thm:main} can be extended to the case in which also the standard lower-order energy terms (and their discretizations) are included into the setting.
\par
{\rm (ii)}
One important aspect of the research on numerical integrators for LLG is related to the development of unconditionally convergent methods, for which the numerical analysis does not require to impose any CFL-type condition on $h$ and $k$.
Theorem~\ref{thm:main} states that this goal is achieved by PF-TPS1.
For TPS1 and TPS2, our analysis requires a mild CFL condition, which arises from the use of the nodal projection and the presence of the DMI.
If the energy comprises only the standard micromagnetic contributions (exchange, uniaxial anisotropy, Zeeman, and magnetostatic), but no DMI, then the convergence towards a weak solution of LLG is unconditional also for TPS1 (for $1/2 < \theta \leq 1$) and TPS2; see~\cite{akt2012,bffgpprs2014,akst2014,dpprs2017}.
\par
{\rm (iii)}
Since the treatment of the DMI requires the CFL condition $k/h \to 0$ as $h,k \to 0$ in our analysis, the result of Theorem~\ref{thm:main} holds for TPS2 also without artificial stabilization, i.e., $\rho \equiv 0$; see~\cite{akst2014,dpprs2017} for more details.
In this case, TPS2 is of full second order in time.
\par
{\rm (iv)}
In~Theorem~\ref{thm:main}, we state the best result that we are able to prove in terms of stability, i.e., with the weakest CFL condition on the discretization parameters.
The same convergence result can be also established at the price of more severe restrictions.
In particular, the result of Theorem~\ref{thm:main} holds
\begin{itemize}
\item[{\rm(a)}] without angle condition~\eqref{eq:angleCondition} for TPS1 and TPS2, if $k/h^2 \to 0$ as $h,k \to 0$;
\item[{\rm(b)}] also for $0 \leq \theta < 1/2$ for TPS1 and PF-TPS1, if $k/h^2 \to 0$ as $h,k \to 0$;
\item[{\rm(c)}] also for $\theta = 1/2$ for PF-TPS1, if $k/h \to 0$ as $h,k \to 0$.
\end{itemize}
\end{remark}
\section{Numerical experiments} \label{sec:numerics}
Before proceeding with the convergence analysis, we aim to show the effectivity of the proposed algorithms with three numerical experiments.
The computations presented in this section have been performed with our micromagnetic software Commics~\cite{commics,prsehhsmp2018}.
Our Python code is based on the open-source finite element library Netgen/NGSolve~\cite{ngsolve}.
The computation of the stray field, i.e., the numerical solution of the transmission problem~\eqref{eq:transmission}, is based on the hybrid FEM-BEM method of~\cite{fk1990}, which requires the evaluation of the double-layer integral operator associated with the Laplace equation; see, e.g., \cite[Section~4.4.1]{bffgpprs2014} or~\cite[Algorithm~12]{prs2018}.
This part of the code exploits the open-source Galerkin boundary element library BEM++~\cite{sbaps2015}.
For all three schemes, to discretize the classical lower-order contributions (anisotropy, Zeeman, and magnetostatic), we follow the explicit approaches of~\cite{bffgpprs2014,ahpprs2014,dpprs2017}.
Magnetization configurations are visualized with ParaView~\cite{agl2005}.
\subsection{Comparison of the integrators} \label{sec:numerics0}
We discretize the rescaled form~\eqref{eq:llg:IBVP} of LLG for a rectangular cuboid $\Omega$ of dimensions $\SI{80}{\nm}\times\SI{80}{\nm}\times\SI{10}{\nm}$, material parameters $\lex =$ \SI{10}{\nano\meter}, $\ldm =$ \SI{20}{\nano\meter}, and $\alpha=$ \num{0.08}, dimensionless final time $T=200$, and constant initial condition $\mm^0 \equiv (q,-q,\sqrt{1-2q^2})$ with $q=0.01$.
\begin{figure}[t]
\captionsetup[subfigure]{labelformat=empty}
\centering
\begin{subfigure}[b]{0.11\textwidth}
\includegraphics[width=\textwidth]{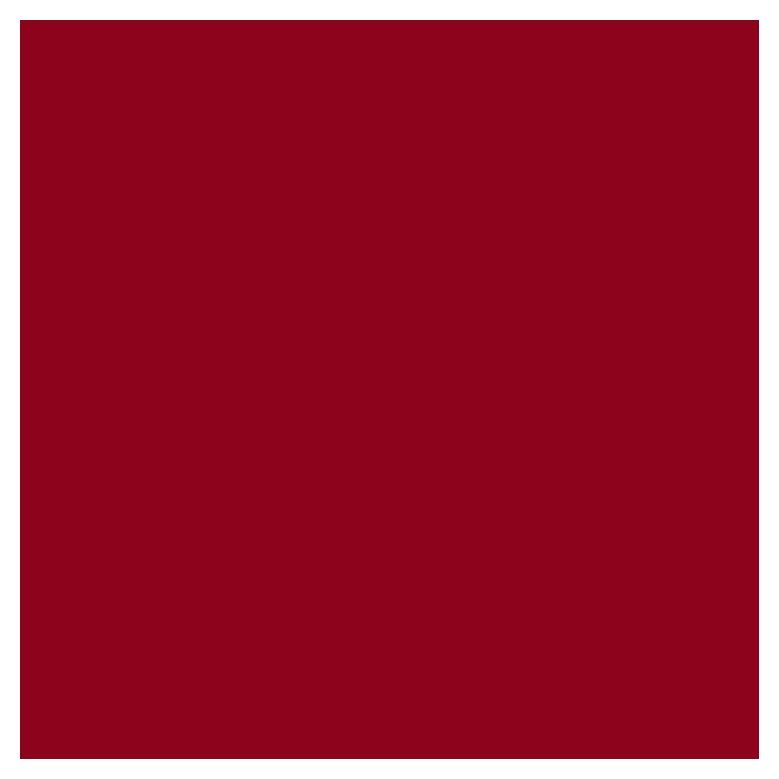}
\caption{$t=0$}
\end{subfigure}
\begin{subfigure}[b]{0.11\textwidth}
\includegraphics[width=\textwidth]{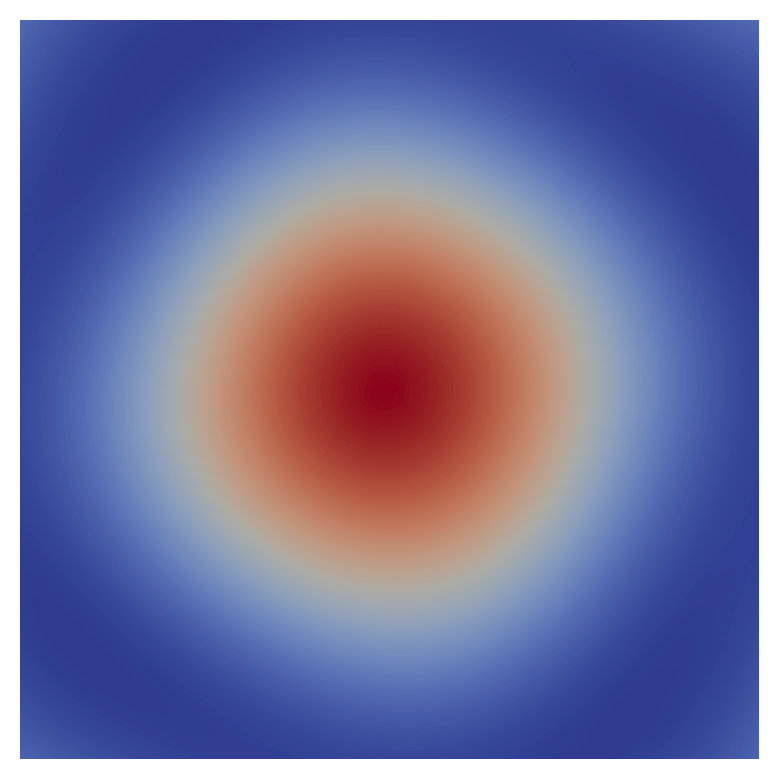}
\caption{$t=15$}
\end{subfigure}
\begin{subfigure}[b]{0.11\textwidth}
\includegraphics[width=\textwidth]{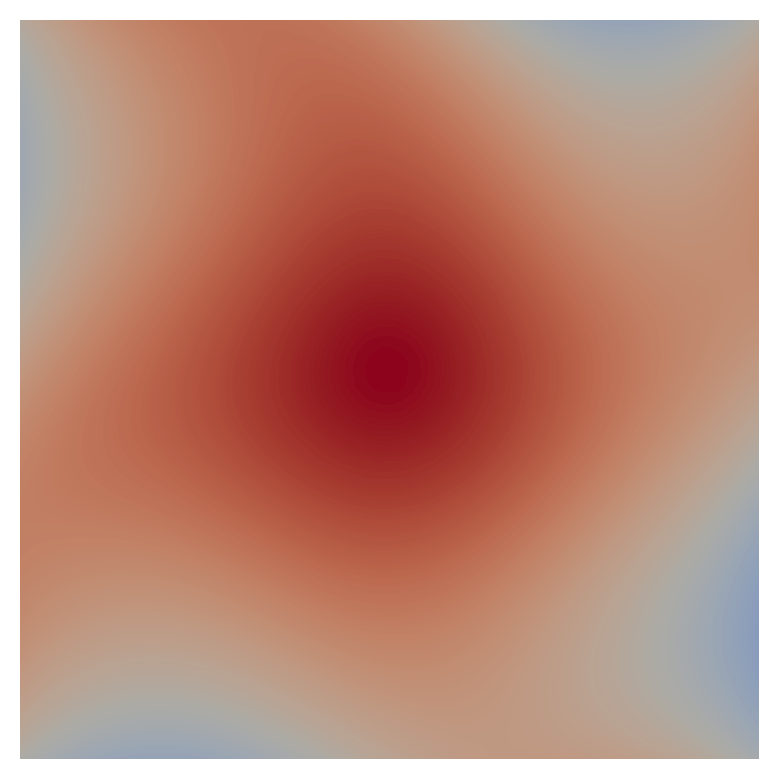}
\caption{$t=27$}
\end{subfigure}
\begin{subfigure}[b]{0.11\textwidth}
\includegraphics[width=\textwidth]{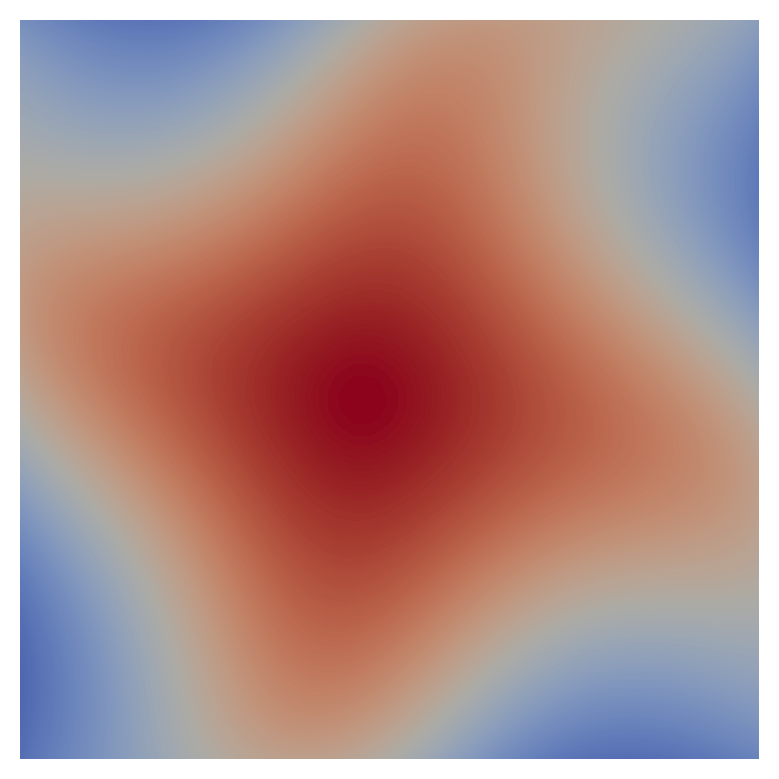}
\caption{$t=38$}
\end{subfigure}
\begin{subfigure}[b]{0.11\textwidth}
\includegraphics[width=\textwidth]{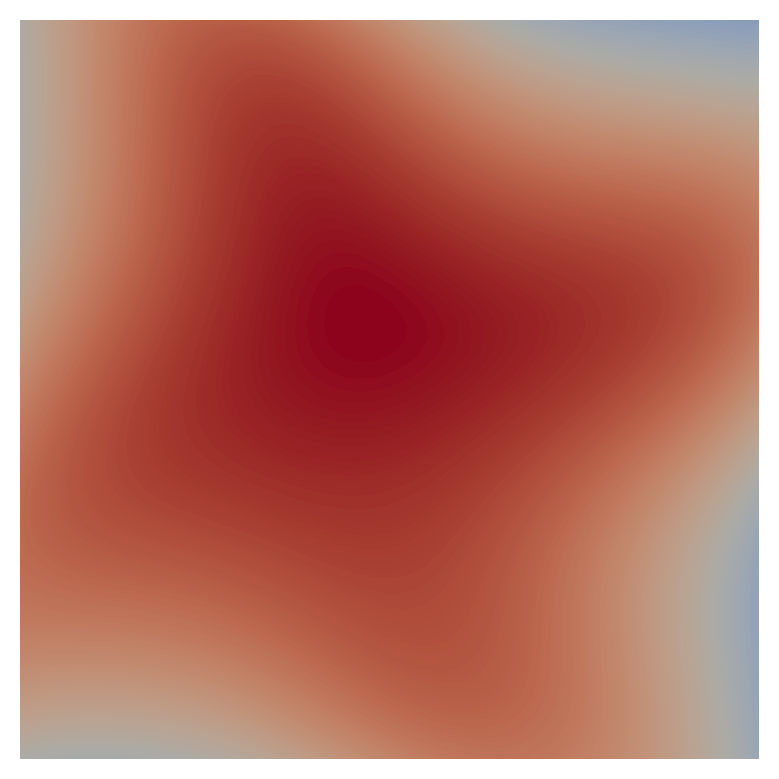}
\caption{$t=51$}
\end{subfigure}
\begin{subfigure}[b]{0.11\textwidth}
\includegraphics[width=\textwidth]{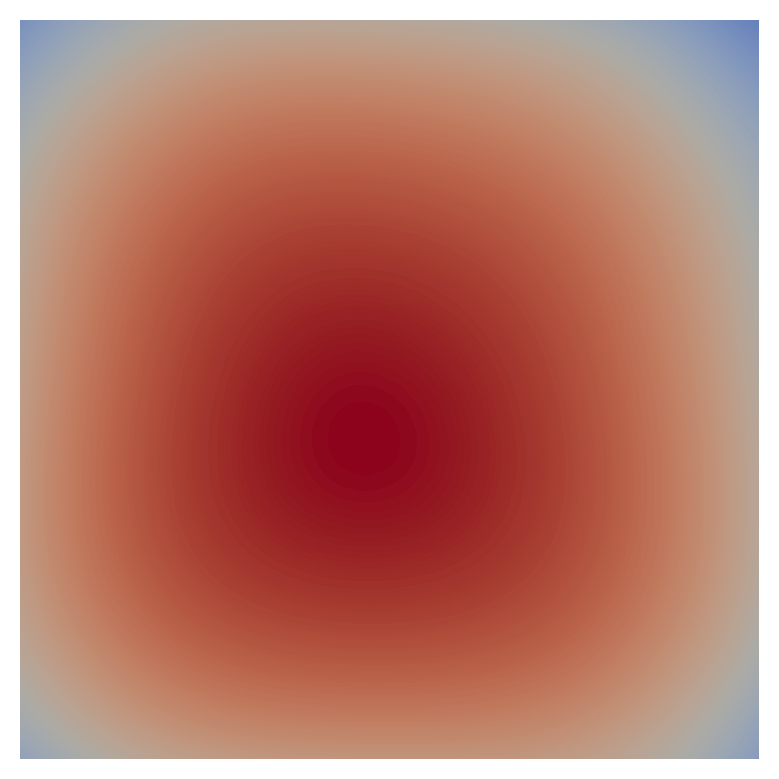}
\caption{$t=100$}
\end{subfigure}
\begin{subfigure}[b]{0.11\textwidth}
\includegraphics[width=\textwidth]{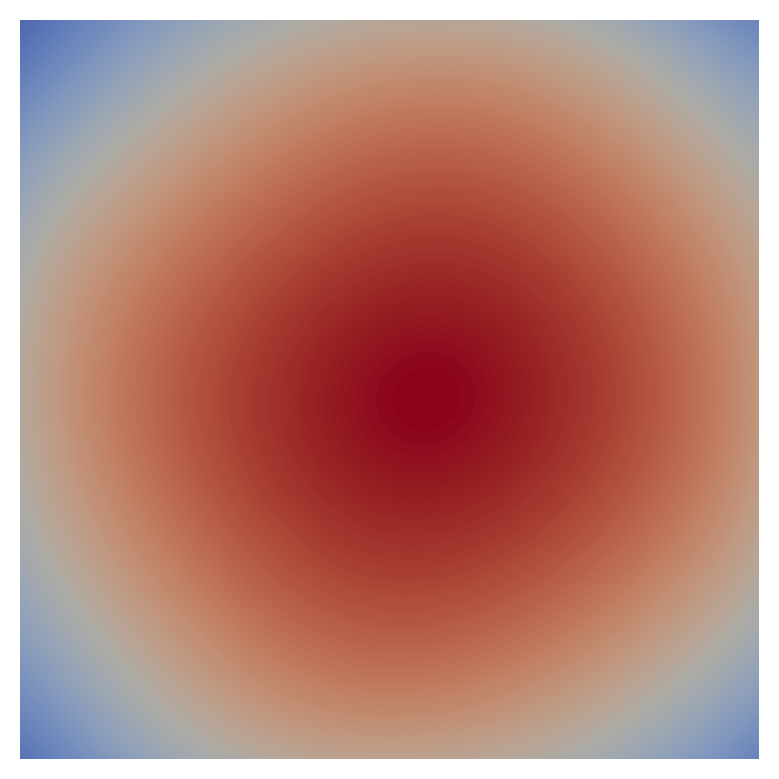}
\caption{$t=150$}
\end{subfigure}
\begin{subfigure}[b]{0.11\textwidth}
\includegraphics[width=\textwidth]{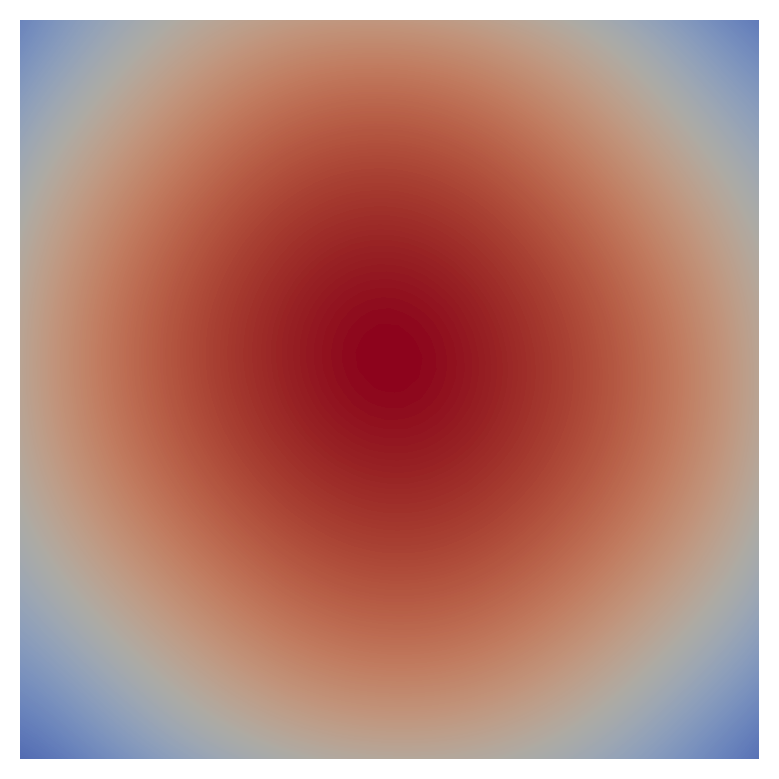}
\caption{$t=200$}
\end{subfigure}
\begin{subfigure}[b]{0.05\textwidth}
\includegraphics[width=\textwidth]{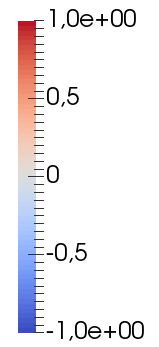}
\caption{}
\end{subfigure}
\caption{Experiment of Section~\ref{sec:numerics0}.
Snapshots of the magnetization dynamics.
The color scale refers to the third component $m_3$ of the magnetization.}
\label{fig:ex0:screenshots}
\end{figure}
For snapshots of the resulting magnetization dynamics, we refer to Figure~\ref{fig:ex0:screenshots}.
\par
\begin{figure}[ht]
\captionsetup[subfigure]{labelfont=rm}
\tdplotsetmaincoords{75}{-35}
\pgfmathsetmacro{\shift}{0.15}
\centering
\begin{subfigure}[b]{0.3\textwidth}
\centering
\begin{tikzpicture}[scale=2.6,tdplot_main_coords]
\coordinate (x0) at (0,0,0);
\coordinate (x1) at (1,0,0);
\coordinate (x2) at (0,1,0);
\coordinate (x3) at (0,0,1);
\coordinate (x4) at (1,1,0);
\coordinate (x5) at (1,0,1);
\coordinate (x6) at (0,1,1);
\coordinate (x7) at (1,1,1);
\coordinate (T1shift) at ($\shift*(x0)+\shift*(x1)$);
\coordinate (T2shift) at ($\shift*(x5)+\shift*(x1)$);
\coordinate (T3shift) at ($\shift*(x5)+\shift*(x7)$);
\coordinate (T4shift) at ($\shift*(x6)+\shift*(x7)$);
\coordinate (T5shift) at ($\shift*(x6)+\shift*(x2)$);
\coordinate (T6shift) at ($\shift*(x0)+\shift*(x2)$);
\draw[fill=yellow, fill opacity=.8] ($(x7)+(T4shift)$)--($(x6)+(T4shift)$)--($(x3)+(T4shift)$)--cycle;
\draw[fill=yellow, fill opacity=.5] ($(x7)+(T4shift)$)--($(x4)+(T4shift)$)--($(x3)+(T4shift)$)--cycle;
\draw[fill=yellow, fill opacity=.5] ($(x4)+(T4shift)$)--($(x6)+(T4shift)$)--($(x3)+(T4shift)$)--cycle;
\draw[dashed] ($(x4) + (T3shift)$)--($(x7) + (T3shift)$);
\draw[fill=blue!50!green, fill opacity=.8] ($(x7)+(T3shift)$)--($(x5)+(T3shift)$)--($(x3)+(T3shift)$)--cycle;
\draw[fill=blue!50!green, fill opacity=.5] ($(x4)+(T3shift)$)--($(x5)+(T3shift)$)--($(x3)+(T3shift)$)--cycle;
\draw[dashed] ($(x4) + (T2shift) $)--($(x5) + (T2shift) $);
\draw[fill=blue, fill opacity=.8] ($(x1)+(T2shift)$)--($(x5)+(T2shift)$)--($(x3)+(T2shift)$)--cycle;
\draw[fill=blue, fill opacity=.5] ($(x1)+(T2shift)$)--($(x4)+(T2shift)$)--($(x3)+(T2shift)$)--cycle;
\draw[dashed] ($(x4) + (T5shift) $)--($(x6) + (T5shift) $);
\draw[fill=red!50!yellow, fill opacity=.8] ($(x2)+(T5shift)$)--($(x6)+(T5shift)$)--($(x3)+(T5shift)$)--cycle;
\draw[fill=red!50!yellow, fill opacity=.5] ($(x2)+(T5shift)$)--($(x4)+(T5shift)$)--($(x3)+(T5shift)$)--cycle;
\draw[dashed] ($(x4) + (T6shift) $)--($(x2) + (T6shift) $);
\draw[fill=red, fill opacity=.8] ($ (x2) + (T6shift) $)--($(x0) + (T6shift) $)--($(x3) + (T6shift) $)--cycle;
\draw[fill=red, fill opacity=.5] ($(x4)+(T6shift)$)--($(x0)+(T6shift)$)--($(x3)+(T6shift)$)--cycle;
\foreach \i in {x0,x1,x3}
    \draw[dashed] ($(x4) + (T1shift) $)--($(\i) + (T1shift) $);
\draw[fill=black!60!blue, fill opacity=.8] ($ (x0) + (T1shift) $)--($(x1) + (T1shift) $)--($(x3) + (T1shift) $)--cycle;
\end{tikzpicture}
\caption{}
\end{subfigure}
\hspace*{2mm}
\begin{subfigure}[b]{0.3\textwidth}
\centering
\begin{tikzpicture}[scale=2.4,tdplot_main_coords]
\pgfmathsetmacro{\smallShift}{0.2}
\coordinate (x0) at (0,0,0);
\coordinate (x1) at (1,0,0);
\coordinate (x2) at (0,1,0);
\coordinate (x3) at (0,0,1);
\coordinate (x4) at (1,1,0);
\coordinate (x5) at (1,0,1);
\coordinate (x6) at (0,1,1);
\coordinate (x7) at (1,1,1);
\coordinate (x8) at (0.5,0.5,0.5);
\coordinate (diag) at ($(x3) - (x4)$);
\coordinate (T1shift)  at ($\shift*(x0)+\shift*(x1)-\smallShift*(diag)$);
\coordinate (T2shift)  at ($\shift*(x0)+\shift*(x1)+\smallShift*(diag)$);
\coordinate (T3shift)  at ($\shift*(x1)+\shift*(x5)-\smallShift*(diag)$);
\coordinate (T4shift)  at ($\shift*(x1)+\shift*(x5)+\smallShift*(diag)$);
\coordinate (T5shift)  at ($\shift*(x7)+\shift*(x5)-\smallShift*(diag)$);
\coordinate (T6shift)  at ($\shift*(x7)+\shift*(x5)+\smallShift*(diag)$);
\coordinate (T7shift)  at ($\shift*(x6)+\shift*(x7)-\smallShift*(diag)$);
\coordinate (T8shift)  at ($\shift*(x6)+\shift*(x7)+\smallShift*(diag)$);
\coordinate (T9shift)  at ($\shift*(x6)+\shift*(x2)-\smallShift*(diag)$);
\coordinate (T10shift) at ($\shift*(x6)+\shift*(x2)+\smallShift*(diag)$);
\coordinate (T11shift) at ($\shift*(x0)+\shift*(x2)-\smallShift*(diag)$);
\coordinate (T12shift) at ($\shift*(x0)+\shift*(x2)+\smallShift*(diag)$);
\draw[fill=blue, fill opacity=.5] ($(x7)+(T7shift)$)--($(x6)+(T7shift)$)--($(x8)+(T7shift)$)--cycle;
\draw[fill=blue, fill opacity=.5] ($(x7)+(T7shift)$)--($(x4)+(T7shift)$)--($(x8)+(T7shift)$)--cycle;
\draw[fill=blue, fill opacity=.5] ($(x4)+(T7shift)$)--($(x6)+(T7shift)$)--($(x8)+(T7shift)$)--cycle;
\draw[dashed] ($(x4) + (T9shift) $)--($(x6) + (T9shift) $);
\draw[fill=pink, fill opacity=.5] ($(x2)+(T9shift)$)--($(x6)+(T9shift)$)--($(x8)+(T9shift)$)--cycle;
\draw[fill=pink, fill opacity=.5] ($(x2)+(T9shift)$)--($(x4)+(T9shift)$)--($(x8)+(T9shift)$)--cycle;
\draw[fill=yellow, fill opacity=.8] ($(x7)+(T8shift)$)--($(x6)+(T8shift)$)--($(x3)+(T8shift)$)--cycle;
\draw[fill=yellow, fill opacity=.5] ($(x7)+(T8shift)$)--($(x8)+(T8shift)$)--($(x3)+(T8shift)$)--cycle;
\draw[fill=yellow, fill opacity=.5] ($(x8)+(T8shift)$)--($(x6)+(T8shift)$)--($(x3)+(T8shift)$)--cycle;
\draw[dashed] ($(x4) + (T5shift)$)--($(x7) + (T5shift)$);
\draw[fill=magenta, fill opacity=.5] ($(x7)+(T5shift)$)--($(x5)+(T5shift)$)--($(x8)+(T5shift)$)--cycle;
\draw[fill=magenta, fill opacity=.5] ($(x4)+(T5shift)$)--($(x5)+(T5shift)$)--($(x8)+(T5shift)$)--cycle;
\draw[dashed] ($(x8) + (T6shift)$)--($(x7) + (T6shift)$);
\draw[fill=blue!50!green, fill opacity=.8] ($(x7)+(T6shift)$)--($(x5)+(T6shift)$)--($(x3)+(T6shift)$)--cycle;
\draw[fill=blue!50!green, fill opacity=.5] ($(x8)+(T6shift)$)--($(x5)+(T6shift)$)--($(x3)+(T6shift)$)--cycle;
\draw[dashed] ($(x8) + (T10shift) $)--($(x6) + (T10shift) $);
\draw[fill=orange, fill opacity=.8] ($(x2)+(T10shift)$)--($(x6)+(T10shift)$)--($(x3)+(T10shift)$)--cycle;
\draw[fill=orange, fill opacity=.5] ($(x2)+(T10shift)$)--($(x8)+(T10shift)$)--($(x3)+(T10shift)$)--cycle;
\draw[dashed] ($(x4) + (T3shift) $)--($(x5) + (T3shift) $);
\draw[fill=red!50!green, fill opacity=.5] ($(x1)+(T3shift)$)--($(x5)+(T3shift)$)--($(x8)+(T3shift)$)--cycle;
\draw[fill=red!50!green, fill opacity=.5] ($(x1)+(T3shift)$)--($(x4)+(T3shift)$)--($(x8)+(T3shift)$)--cycle;
\draw[dashed] ($(x4) + (T11shift) $)--($(x2) + (T11shift) $);
\draw[fill=green, fill opacity=.5] ($ (x2) + (T11shift) $)--($(x0) + (T11shift) $)--($(x8) + (T11shift) $)--cycle;
\draw[fill=green, fill opacity=.5] ($(x4)+(T11shift)$)--($(x0)+(T11shift)$)--($(x8)+(T11shift)$)--cycle;
\foreach \i in {x0,x1,x8}
    \draw[dashed] ($(x4) + (T1shift) $)--($(\i) + (T1shift) $);
\draw[fill=yellow!50!orange, fill opacity=.5] ($ (x0) + (T1shift) $)--($(x1) + (T1shift) $)--($(x8) + (T1shift) $)--cycle;
\draw[dashed] ($(x8) + (T12shift) $)--($(x2) + (T12shift) $);
\draw[fill=red, fill opacity=.8] ($ (x2) + (T12shift) $)--($(x0) + (T12shift) $)--($(x3) + (T12shift) $)--cycle;
\draw[fill=red, fill opacity=.5] ($(x8)+(T12shift)$)--($(x0)+(T12shift)$)--($(x3)+(T12shift)$)--cycle;
\draw[dashed] ($(x8) + (T4shift) $)--($(x5) + (T4shift) $);
\draw[fill=blue, fill opacity=.8] ($(x1)+(T4shift)$)--($(x5)+(T4shift)$)--($(x3)+(T4shift)$)--cycle;
\draw[fill=blue, fill opacity=.5] ($(x1)+(T4shift)$)--($(x8)+(T4shift)$)--($(x3)+(T4shift)$)--cycle;
\foreach \i in {x0,x1,x3}
    \draw[dashed] ($(x8) + (T2shift) $)--($(\i) + (T2shift) $);
\draw[fill=black!50!blue, fill opacity=.8] ($ (x0) + (T2shift) $)--($(x1) + (T2shift) $)--($(x3) + (T2shift) $)--cycle;
\end{tikzpicture}
\caption{}
\end{subfigure}
\hspace*{2mm}
\begin{subfigure}[b]{0.3\textwidth}
\centering
\begin{tikzpicture}[spy using outlines={circle,red, magnification=3,size=1.75cm, connect spies}]
\node {\pgfimage[width=0.95\textwidth]{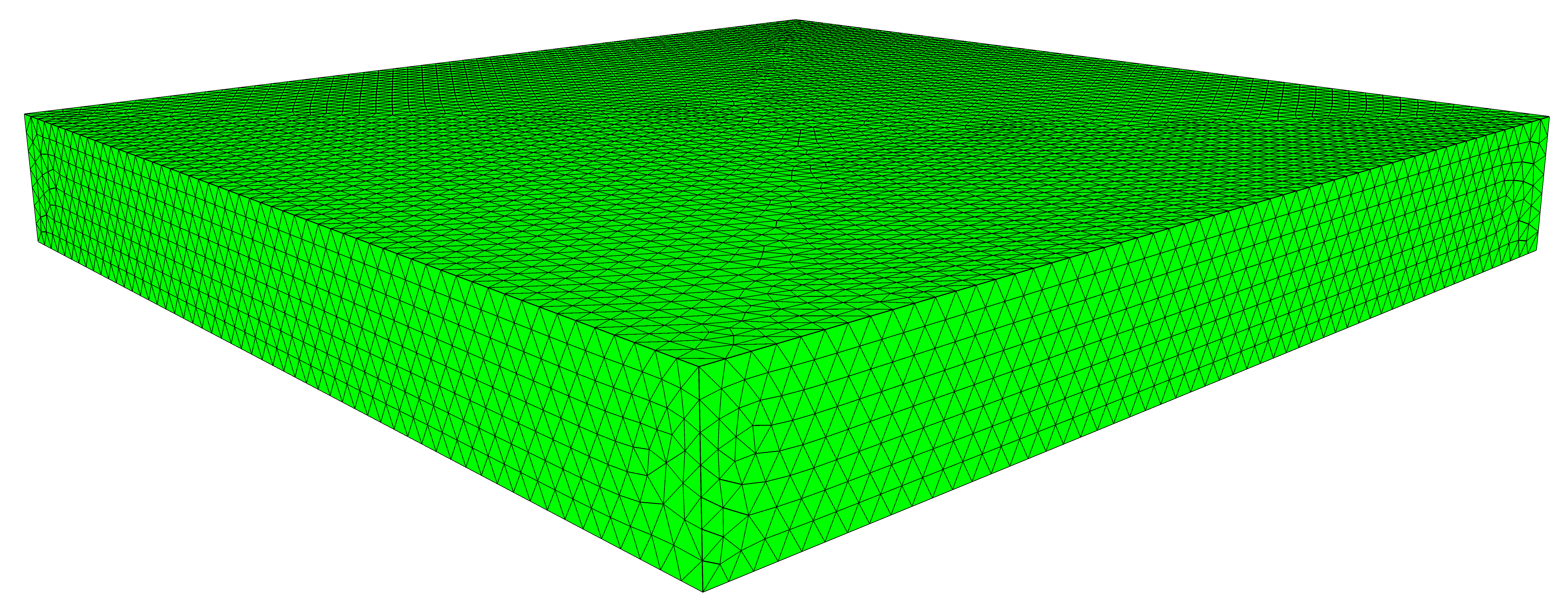}};
\spy[every spy on node/.append style={thick}, size=2.7cm] on (-0.25,-0.45) in node [left] at (2.3,-2.2);
\end{tikzpicture}
\caption{}
\end{subfigure}
\caption{Experiment of Section~\ref{sec:numerics0}.
Mesh types:
(a) Type~I;
(b) Type~II;
(c) Type~III.}
\label{fig:meshes}
\end{figure}
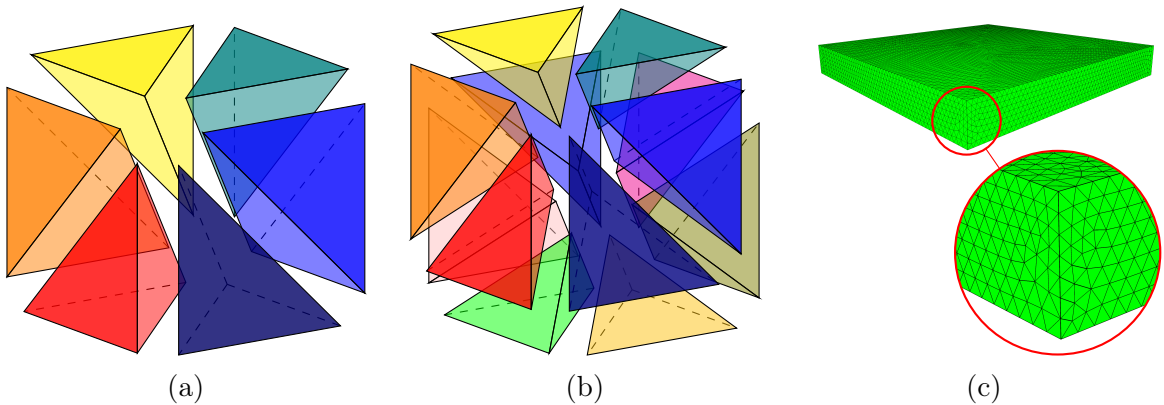
For the spatial discretization, we consider three types of tetrahedral meshes; see Figure~\ref{fig:meshes}.
For meshes of type~I, the domain is first uniformly decomposed into cubes.
Then, each cube is split into six tetrahedra in such a way that any tetrahedron has three mutually perpendicular edges.
Any mesh of this type satisfies the angle condition~\eqref{eq:angleCondition}; see \cite[Lemma~3.5]{bartels2005}.
Meshes of type~II are obtained from the previous one bisecting the longest edge of each of the six tetrahedra
i.e., the main diagonal of the original cube.
As a result the cube is uniformly split into twelve tetrahedra.
Meshes of this type do not satisfy~\eqref{eq:angleCondition}.
For type~III, we consider unstructured meshes obtained with Netgen,
which are generated with the advancing front method (see~\cite{schoeberl1997} for details)
and in general do not satisfy~\eqref{eq:angleCondition}.
\par
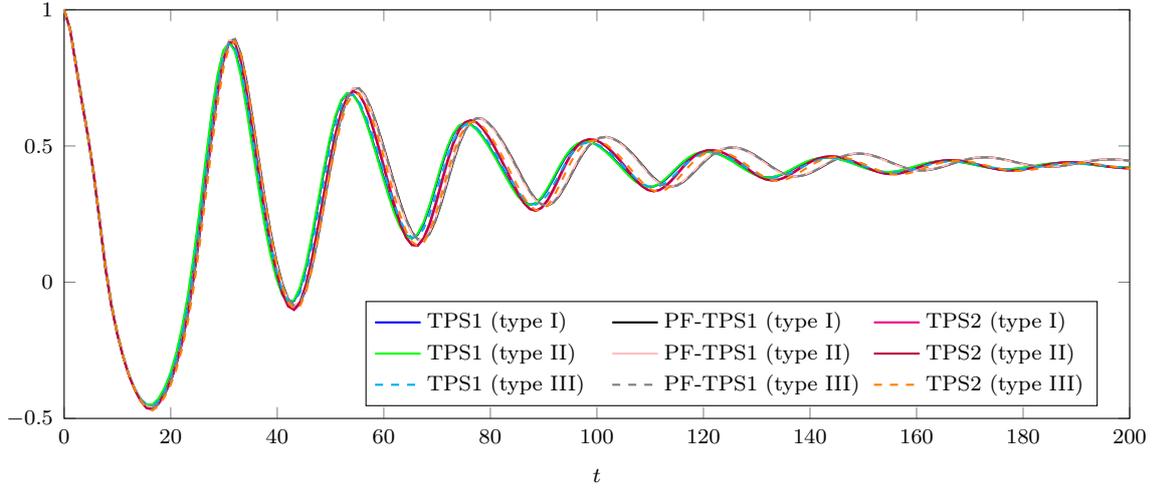
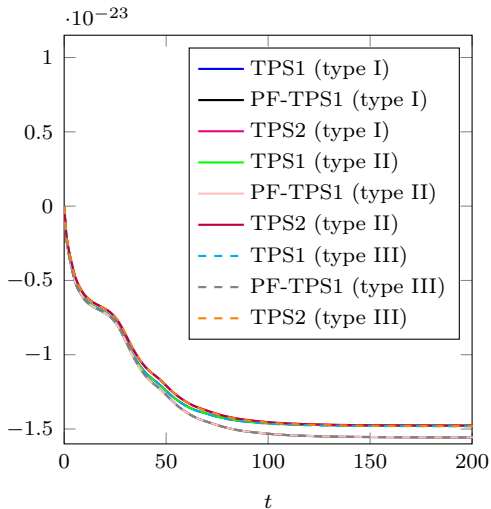
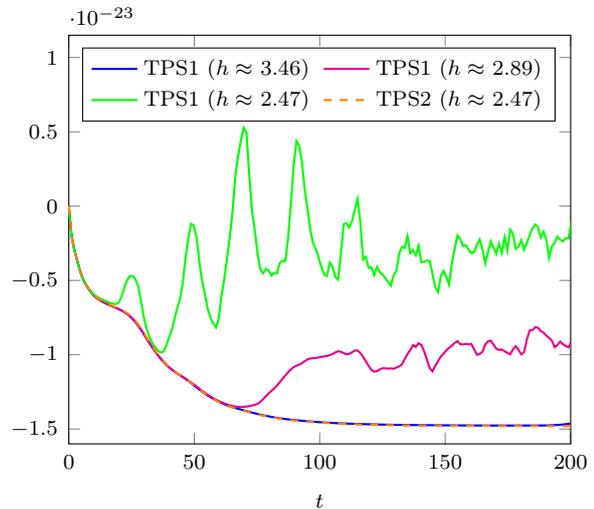
\begin{figure}[ht]
\captionsetup[subfigure]{labelfont=rm}
\centering
\begin{subfigure}[b]{\textwidth}
\begin{tikzpicture}
\pgfplotstableread{pics/ex0/plot1/paper_struct_TPS1_averages.dat}{\structTPSone}
\pgfplotstableread{pics/ex0/plot1/paper_struct_TPS2_averages.dat}{\structTPStwo}
\pgfplotstableread{pics/ex0/plot1/paper_obtuse_TPS1_averages.dat}{\obtuseTPSone}
\pgfplotstableread{pics/ex0/plot1/paper_obtuse_TPS2_averages.dat}{\obtuseTPStwo}
\pgfplotstableread{pics/ex0/plot1/paper_unstruct_TPS1_averages.dat}{\unstructTPSone}
\pgfplotstableread{pics/ex0/plot1/paper_unstruct_TPS2_averages.dat}{\unstructTPStwo}
\pgfplotstableread{pics/ex0/plot1/paper_struct_PFTPS1_averages.dat}{\structPFTPSone}
\pgfplotstableread{pics/ex0/plot1/paper_obtuse_PFTPS1_averages.dat}{\obtusePFTPSone}
\pgfplotstableread{pics/ex0/plot1/paper_struct_PFTPS1_averages.dat}{\unstructPFTPSone} 
\begin{axis}[
width = 0.98\textwidth,
height = 70mm,
xlabel={\tiny $t$},
xmin=0,
xmax=200,
ymin=-0.5,
ymax=1,
legend style={
legend pos=south east,
legend cell align = left,
legend columns = 3,
/tikz/column 2/.style={column sep=5pt},
font = \tiny,
}
]
\addplot[blue,thick] table[x=t, y=mz]{\structTPSone};
\addplot[black,thick] table[x=t, y=mz]{\structPFTPSone};
\addplot[magenta,thick] table[x=t, y=mz]{\structTPStwo};
\addplot[green,thick] table[x=t, y=mz]{\obtuseTPSone};
\addplot[pink,thick] table[x=t, y=mz]{\obtusePFTPSone};
\addplot[purple,thick] table[x=t, y=mz]{\obtuseTPStwo};
\addplot[cyan,thick,dashed] table[x=t, y=mz]{\unstructTPSone};
\addplot[gray,thick,dashed] table[x=t, y=mz]{\unstructPFTPSone};
\addplot[orange,thick,dashed] table[x=t, y=mz]{\unstructTPStwo};
\legend{
\tiny
TPS1 (type~I) \ ,
PF-TPS1 (type~I) \ ,
TPS2 (type~I),
TPS1 (type~II) \ ,
PF-TPS1 (type~II) \ ,
TPS2 (type~II),
TPS1 (type~III) \ ,
PF-TPS1 (type~III) \ ,
TPS2 (type~III)
}
\end{axis}
\end{tikzpicture}
\caption{}
\label{fig:ex0:averages}
\end{subfigure}
\\
\vspace*{2mm}
\begin{subfigure}[b]{0.45\textwidth}
\begin{tikzpicture}
\pgfplotstableread{pics/ex0/plot2/paper_struct_TPS1_energies.dat}{\structTPSone}
\pgfplotstableread{pics/ex0/plot2/paper_struct_TPS2_energies.dat}{\structTPStwo}
\pgfplotstableread{pics/ex0/plot2/paper_obtuse_TPS1_energies.dat}{\obtuseTPSone}
\pgfplotstableread{pics/ex0/plot2/paper_obtuse_TPS2_energies.dat}{\obtuseTPStwo}
\pgfplotstableread{pics/ex0/plot2/paper_unstruct_TPS1_energies.dat}{\unstructTPSone}
\pgfplotstableread{pics/ex0/plot2/paper_unstruct_TPS2_energies.dat}{\unstructTPStwo}
\pgfplotstableread{pics/ex0/plot2/paper_obtuse_PFTPS1_energies.dat}{\unstructPFTPSone}
\pgfplotstableread{pics/ex0/plot2/paper_struct_PFTPS1_energies.dat}{\structPFTPSone}
\pgfplotstableread{pics/ex0/plot2/paper_obtuse_PFTPS1_energies.dat}{\obtusePFTPSone}
\pgfplotstableread{pics/ex0/plot2/paper_struct_PFTPS1_energies.dat}{\unstructPFTPSone} 
\begin{axis}[
width = 0.97\textwidth,
height=70mm,
xlabel={\tiny $t$},
xmin=0,
xmax=200,
ymin=-1.6e-23,
ymax=1.15e-23,
legend style={legend pos=north east, legend cell align = left, font = \tiny}
]
\addplot[blue,thick] table[x=t, y=energy]{\structTPSone};
\addplot[black,thick] table[x=t, y=energy]{\structPFTPSone};
\addplot[magenta,thick] table[x=t, y=energy]{\structTPStwo};
\addplot[green,thick] table[x=t, y=energy]{\obtuseTPSone};
\addplot[pink,thick] table[x=t, y=energy]{\obtusePFTPSone};
\addplot[purple,thick] table[x=t, y=energy]{\obtuseTPStwo};
\addplot[cyan,thick,dashed] table[x=t, y=energy]{\unstructTPSone};
\addplot[gray,thick,dashed] table[x=t, y=energy]{\unstructPFTPSone};
\addplot[orange,thick,dashed] table[x=t, y=energy]{\unstructTPStwo};
\legend{
TPS1~(type~I),
PF-TPS1~(type~I),
TPS2~(type~I),
TPS1~(type~II),
PF-TPS1~(type~II),
TPS2~(type~II),
TPS1~(type~III),
PF-TPS1~(type~III),
TPS2~(type~III)
}
\end{axis}
\end{tikzpicture}
\caption{}
\label{fig:ex0:energies1}
\end{subfigure}
\hfill
\begin{subfigure}[b]{0.53\textwidth}
\begin{tikzpicture}
\pgfplotstableread{pics/ex0/plot3/paper_TPS1_theta05_coarse.dat}{\TPSoneThetaHalfCoarse}
\pgfplotstableread{pics/ex0/plot3/paper_TPS1_theta05_medium.dat}{\TPSoneThetaHalfMedium}
\pgfplotstableread{pics/ex0/plot3/paper_TPS1_theta05_fine.dat}{\TPSoneThetaHalfFine}
\pgfplotstableread{pics/ex0/plot3/paper_TPS2_fine.dat}{\TPStwoFine}
\begin{axis}[
width = 0.97\textwidth,
height=70mm,
xlabel={\tiny $t$},
xmin=0,
xmax=200,
ymin=-1.6e-23,
ymax=1.15e-23,
legend style={legend pos=north east, legend cell align = left, font = \tiny, legend columns = 2}
]
\addplot[blue,thick] table[x=t, y=energy]{\TPSoneThetaHalfCoarse};
\addplot[magenta,thick] table[x=t, y=energy]{\TPSoneThetaHalfMedium};
\addplot[green,thick] table[x=t, y=energy]{\TPSoneThetaHalfFine};
\addplot[orange,thick,dashed] table[x=t, y=energy]{\TPStwoFine};
\legend{
TPS1 ($h \approx$ \num{3.46}) \ ,
TPS1 ($h \approx$ \num{2.89}),
TPS1 ($h \approx$ \num{2.47}) \ ,
TPS2 ($h \approx$ \num{2.47})
}
\end{axis}
\end{tikzpicture}
\caption{}
\label{fig:ex0:energies2}
\end{subfigure}
\caption{Experiment of Section~\ref{sec:numerics0}.
Time evolutions of $\langle m_3 \rangle$ and $\E(\mm)$:
(a) Time evolution of $\langle m_3 \rangle$ for all schemes, all mesh types, and a fixed mesh size of $h \approx$ \SI{3.46}{\nano\meter};
(b) Time evolution of $\E(\mm)$ for all schemes, all mesh types; and a fixed mesh size of $h \approx$ \SI{3.46}{\nano\meter};
(c) Time evolution of $\E(\mm)$ for TPS1 ($\theta=1/2$) and TPS2, a mesh of type~I, and different mesh sizes.
}
\label{fig:ex0:energies+averages}
\end{figure}
In Figure~\ref{fig:ex0:energies+averages},
we plot the time evolutions
of the third component of the spatially averaged magnetization of the sample,
i.e., $\langle m_3(t) \rangle = \abs{\Omega}^{-1} \int_{\Omega} m_3(\xx,t)\, \dx$,
and the energy~\eqref{eq:llg:energy}
obtained with the three algorithms for different mesh types and sizes,
and a constant time-step size $k = \gamma_0 \num{e-7} \approx$ \num{0.0221}.
In Figure~\ref{fig:ex0:averages}--\ref{fig:ex0:energies1}, we compare the results
obtained with TPS1 ($\theta=1$), PF-TPS1 ($\theta=1$), and TPS2 for the mesh types I--III.
For each mesh type, we consider a mesh size of $h \approx$ \SI{3.46}{\nano\meter}.
Note that the meshes of types~II--III violate~\eqref{eq:angleCondition}.
In Figure~\ref{fig:ex0:energies2}, we compare the results obtained with TPS1 ($\theta=1/2$) and TPS2.
We consider a structured mesh of type~I and compare the results obtained for different
mesh sizes.

Although the convergence result of Theorem~\ref{thm:main} does not cover meshes
of types~II--III for TPS1 and TPS2, the numerical results show that, in terms of
stability, the methods behave identically, independently of the mesh type used.
To better understand this aspect, we also monitored \textsl{a~posteriori} the validity of the inequality
\begin{equation*}
\norm[\LL^2(\Omega)]{\Grad\mm_h^{i+1}}
\leq
\norm[\LL^2(\Omega)]{\Grad(\mm_h^i + k \vv_h^i)},
\end{equation*}
which is the inequality effectively used in the stability analysis of TPS1 and TPS2;
see Proposition~\ref{prop:tps1:energy} and Proposition~\ref{prop:tps2:energy} below,
respectively.
It turned out that the inequality is always satisfied, even for meshes
violating~\eqref{eq:angleCondition}.

The omission of the nodal projection in PF-TPS1 manifests itself as a phase error 
in the evolution of $\langle m_3 \rangle$ accumulating over time
(see Figure~\ref{fig:ex0:averages}), and as a lower energy level of the final
magnetization configuration (see Figure~\ref{fig:ex0:energies1}).
However, the overall qualitative outcome of the experiment is preserved.

The results also show that TPS1 with $\theta=1$ is more dissipative than TPS2.
However, the choice of $\theta=1/2$, which would be favorable from an energetic point of view
(no artificial damping), is not feasible, because it affects the stability of the scheme;
see Figure~\ref{fig:ex0:energies1}.
The instability is more severe for smaller mesh sizes, giving numerical evidence
of the CFL-condition required for stability in this case; see Remark~\ref{rem:mainThm}(ii).
\par
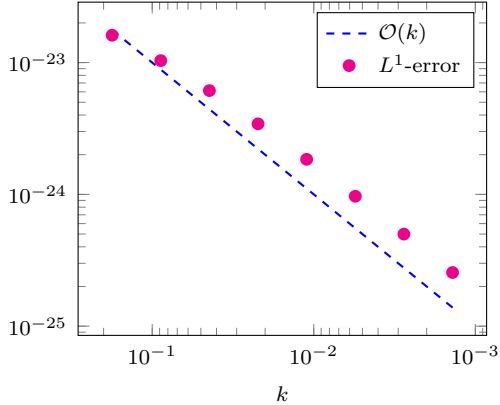
\begin{figure}[t]
\centering
\begin{tikzpicture}
\pgfplotstableread{pics/ex0/errL1.dat}{\error}
\begin{loglogaxis}[
height = 60mm,
xlabel={\tiny $k$},
legend style={legend pos=north east, legend cell align = left, font = \tiny},
x dir=reverse
]
\addplot[blue,dashed,thick] table[x=k, y expr={(\thisrow{k})*10^(-22)}]{\error};
\addplot[magenta,thick,only marks] table[x=k, y=normL1]{\error};
\legend{
$\mathcal{O}(k)$,
$L^1$-error}
\end{loglogaxis}
\end{tikzpicture}
\caption{
Experiment of Section~\ref{sec:numerics0}.
Empirical convergence rate as $k \to 0$ for the error $\big\Vert\interp\big[\abs{\mmhk^+(T)}^2\big]-1\big\Vert_{L^1(\Omega)}$
for PF-TPS1 ($\theta=1$).
}
\label{fig:ex0:pf-error}
\end{figure}
Finally, in Figure~\ref{fig:ex0:pf-error}, we study the violation of the unit-length constraint which occurs for PF-TPS1.
We consider a structured mesh of type~I with mesh size $h \approx$ \SI{4.33}{\nano\meter}
and plot the error $\big\Vert\interp\big[\abs{\mmhk^+(T)}^2\big]-1\big\Vert_{L^1(\Omega)}$
for different time-step sizes $k$.
Note that this error is identically zero for TPS1 and TPS2, because of the nodal projection.
We observe a linear dependence of the error on $k$ which is in total agreement with the theory,
see estimate~\eqref{eq:ConstraintLinearDecay}
in the proof of Proposition~\ref{prop:convergenceSubsequences} below.

For numerical experiments testing the experimental convergence rates in time of the schemes (in the absence of DMI),
we refer to~\cite[Section~6.2.1]{ruggeri2016} (for TPS1 and PF-TPS1)
and~\cite[Section~7.1]{dpprs2017} (for TPS1 and TPS2).
There, the observed rates with respect to a reference solution match the formal consistency error
of the schemes, i.e., first-order convergence for TPS1 and PF-TPS1, second-order convergence for TPS2.
A similar numerical study for the present model problem (which includes DMI) confirms
the first-order convergence for TPS1 and PF-TPS1,
but does not reveal a full second-order convergence for TPS2.
We believe that this is due to a lack of regularity of the solution in time.
\subsection{Stability of isolated skyrmions in nanodisks} \label{sec:numerics1}
We reproduce a numerical experiment from~\cite{scrtf2013}.
We investigate the relaxed states of a thin nanodisk of diameter \SI{80}{\nano\meter} (aligned with $x_1 x_2$-plane) and thickness \SI{0.4}{\nano\meter} ($x_3$-direction) centered at $(0,0,0)$ for different values of the DMI constant and initial conditions.
The effective field in~\eqref{eq:llg:physical} consists of exchange interaction, perpendicular uniaxial anisotropy, interfacial DMI, and stray field, i.e.,
\begin{equation*}
\Heff(\mm)
= \frac{2A}{\mu_0 \Ms} \Lapl\mm
+ \frac{2 K}{\mu_0 \Ms} (\aa\cdot\mm)\aa
- \frac{2D}{\mu_0 \Ms}
\begin{pmatrix}
- \de_1 m_3 \\
- \de_2 m_3 \\
\de_1 m_1 + \de_2 m_2
\end{pmatrix}
+ \Hstray(\mm).
\end{equation*}
The values of the involved material parameters mimic those of cobalt:
$\Ms=$ \SI{5.8e5}{\ampere\per\meter}, $\alpha=$ \num{0.3}, $A=$ \SI{1.5e-11}{\joule\per\meter}, $K=$ \SI{8e5}{\joule\per\meter\cubed}, and $\aa=(0,0,1)$.
For the DMI constant, we consider the values  $D=$ \num{0}, \num{1}, \dots, \num{8} \si{\milli\joule\per\square\meter}.
We test two different initial magnetization configurations:
\begin{itemize}
\item[\rm(i)] a uniform out-of-plane ferromagnetic state, i.e., $\mm^0 \equiv (0,0,1)$,
\item[\rm(ii)] a skyrmion-like state, i.e., given $r = \sqrt{x_1^2 + x_2^2}$, we define $\mm^0(\xx)=(0,0,-1)$ if $r \in [0,15]$ \si{\nano\meter} and $\mm^0(\xx)=(0,0,1)$ if $r \in (15,40]$ \si{\nano\meter}.
\end{itemize}
For all simulations, we choose $T=$ \SI{2}{\nano\second} for $D=$ \num{0}, \dots, \num{6} \si{\milli\joule\per\square\meter} and $T=$ \SI{5}{\nano\second} for $D=$ \num{7}, \num{8} \si{\milli\joule\per\square\meter}, which experimentally turn out to be sufficiently large times to relax the system.
The computational domain is discretized by a regular partition consisting of \num{32575} tetrahedra (mesh size of \SI{1}{\nano\meter}).
For the time discretization, we consider a uniform partition of the time interval $(0,T)$ with a time-step size of \SI{0.1}{\pico\second}.
\begin{figure}[t]
\captionsetup[subfigure]{labelformat=empty}
\centering
\begin{subfigure}[b]{0.095\textwidth}
\includegraphics[width=\textwidth]{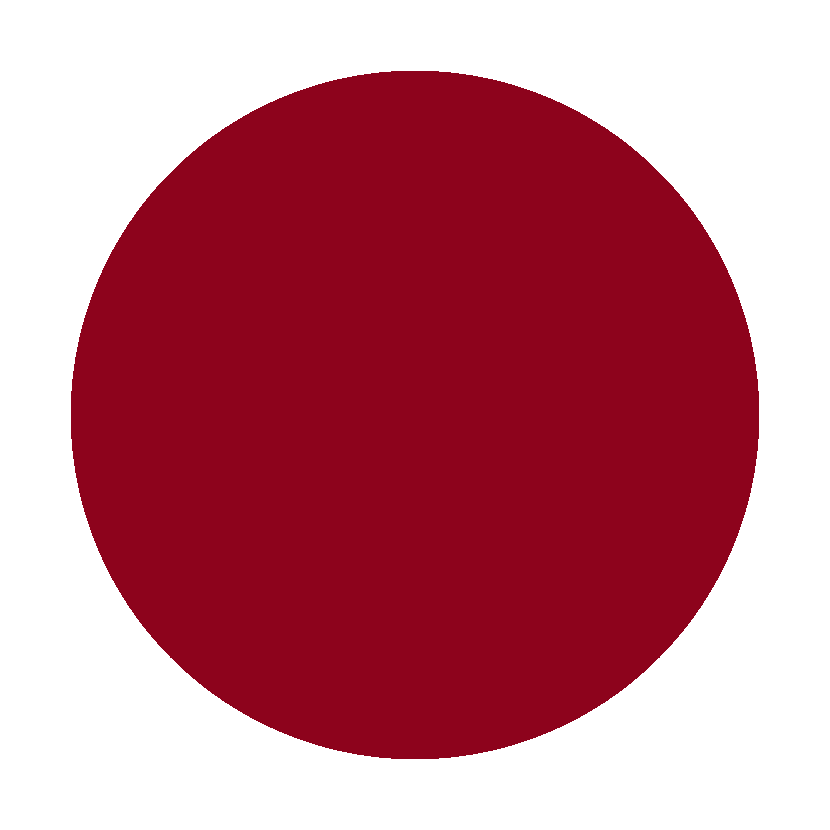}
\caption{$D=0$}
\end{subfigure}
\begin{subfigure}[b]{0.095\textwidth}
\includegraphics[width=\textwidth]{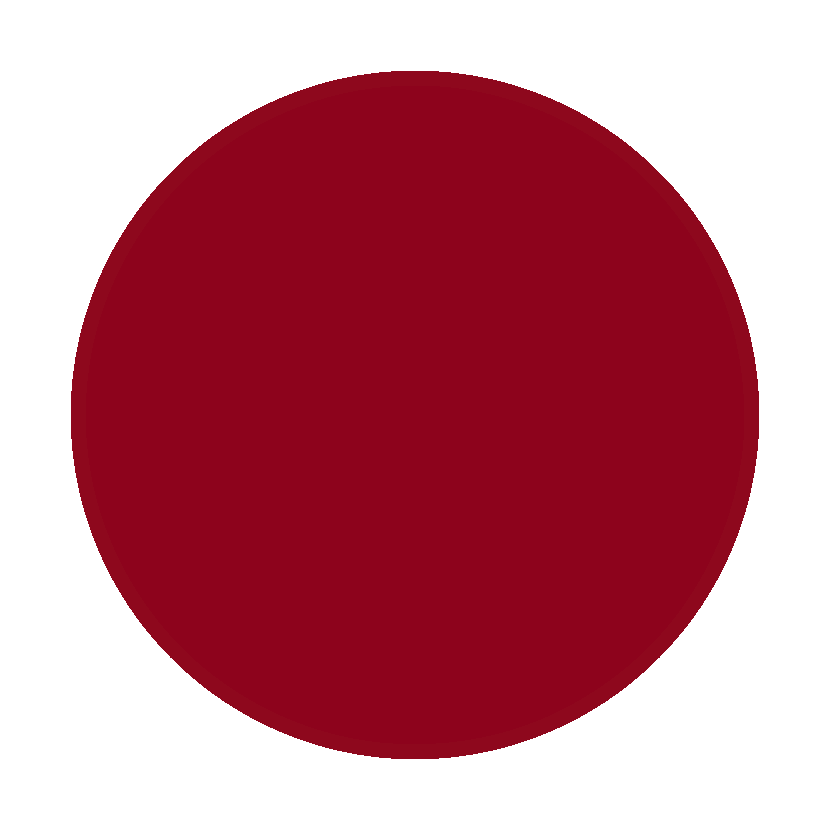}
\caption{$D=1$}
\end{subfigure}
\begin{subfigure}[b]{0.095\textwidth}
\includegraphics[width=\textwidth]{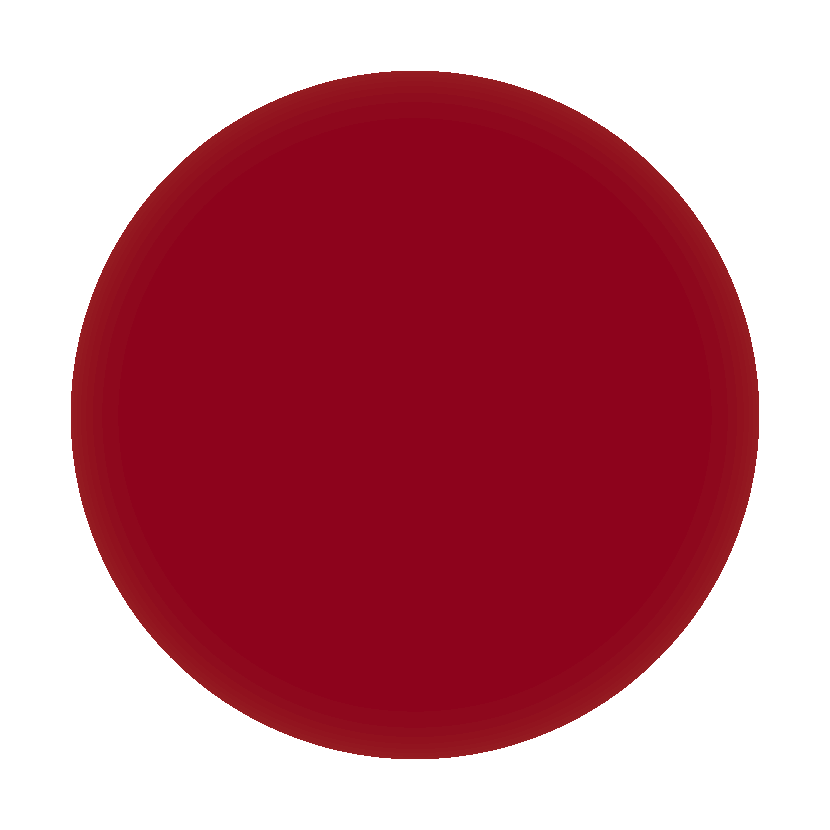}
\caption{$D=2$}
\end{subfigure}
\begin{subfigure}[b]{0.095\textwidth}
\includegraphics[width=\textwidth]{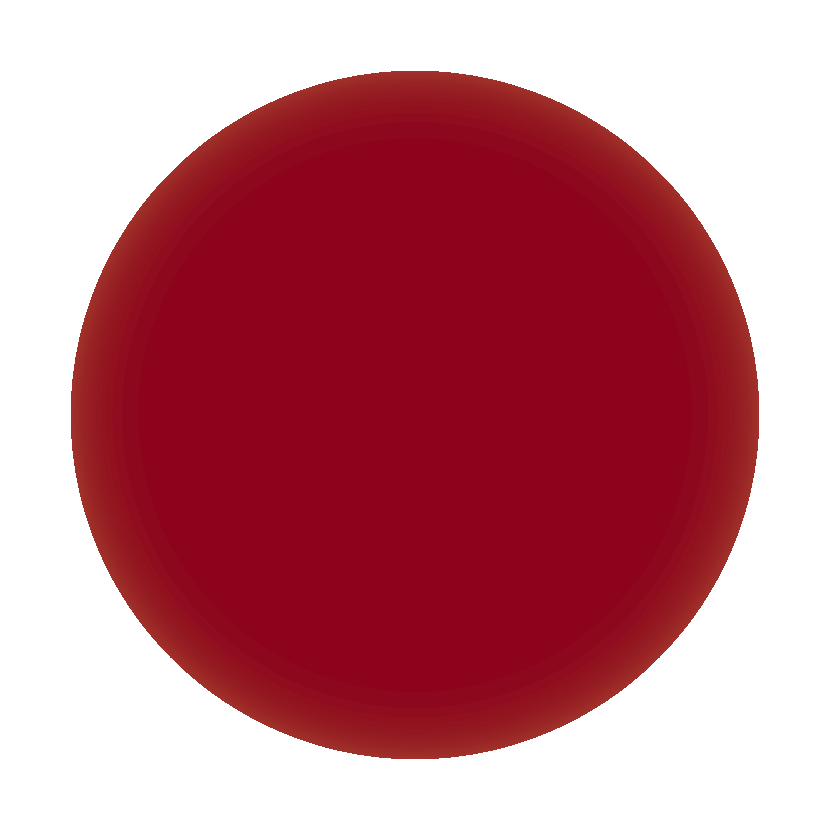}
\caption{$D=3$}
\end{subfigure}
\begin{subfigure}[b]{0.095\textwidth}
\includegraphics[width=\textwidth]{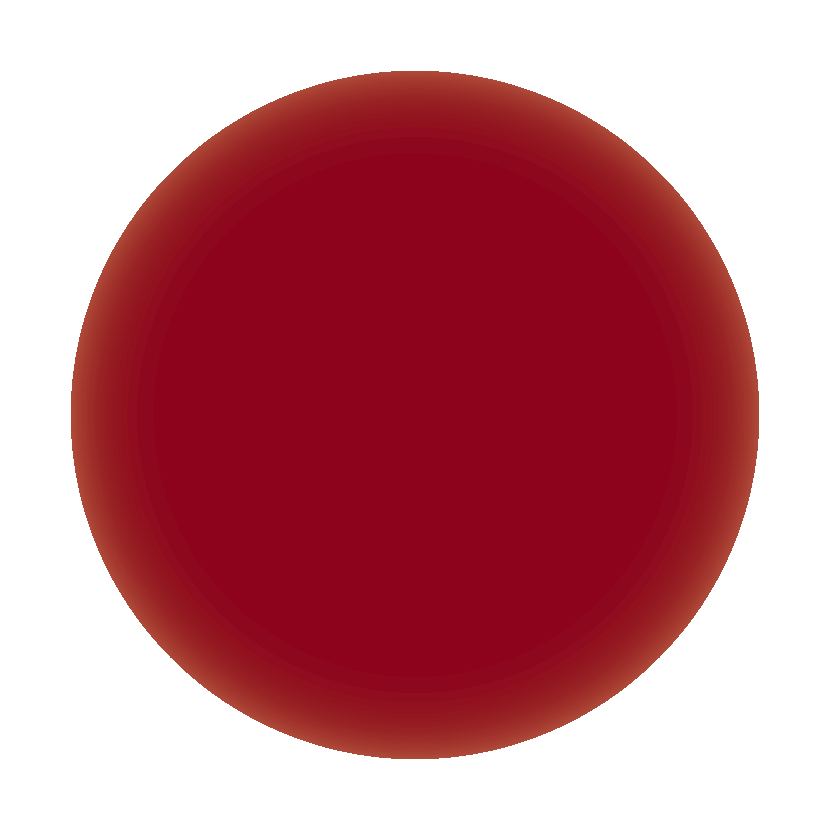}
\caption{$D=4$}
\end{subfigure}
\begin{subfigure}[b]{0.095\textwidth}
\includegraphics[width=\textwidth]{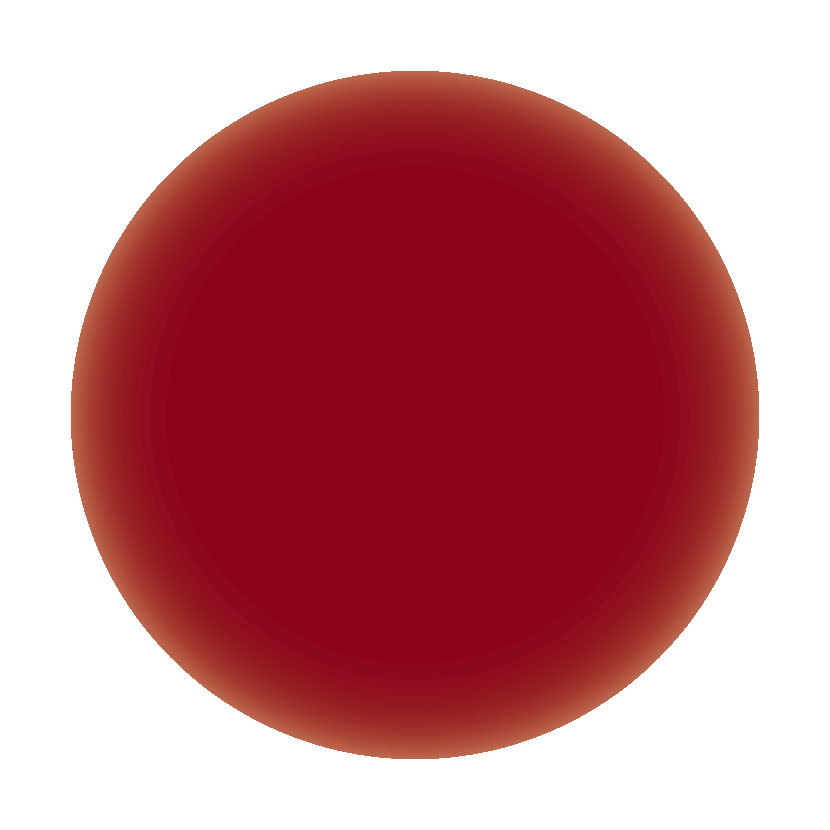}
\caption{$D=5$}
\end{subfigure}
\begin{subfigure}[b]{0.095\textwidth}
\includegraphics[width=\textwidth]{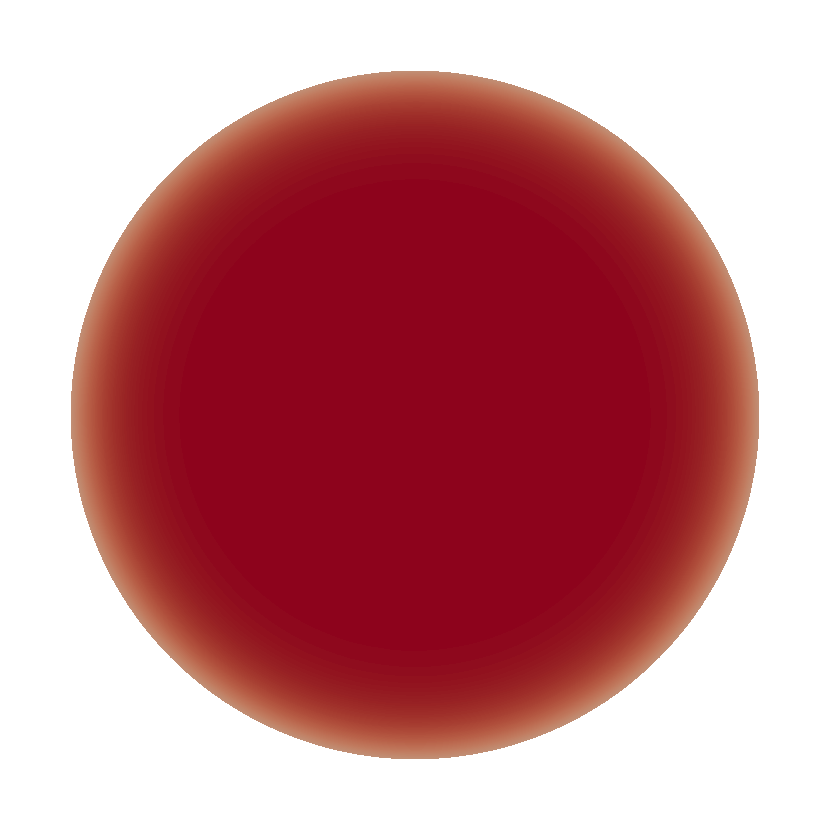}
\caption{$D=6$}
\end{subfigure}
\begin{subfigure}[b]{0.095\textwidth}
\includegraphics[width=\textwidth]{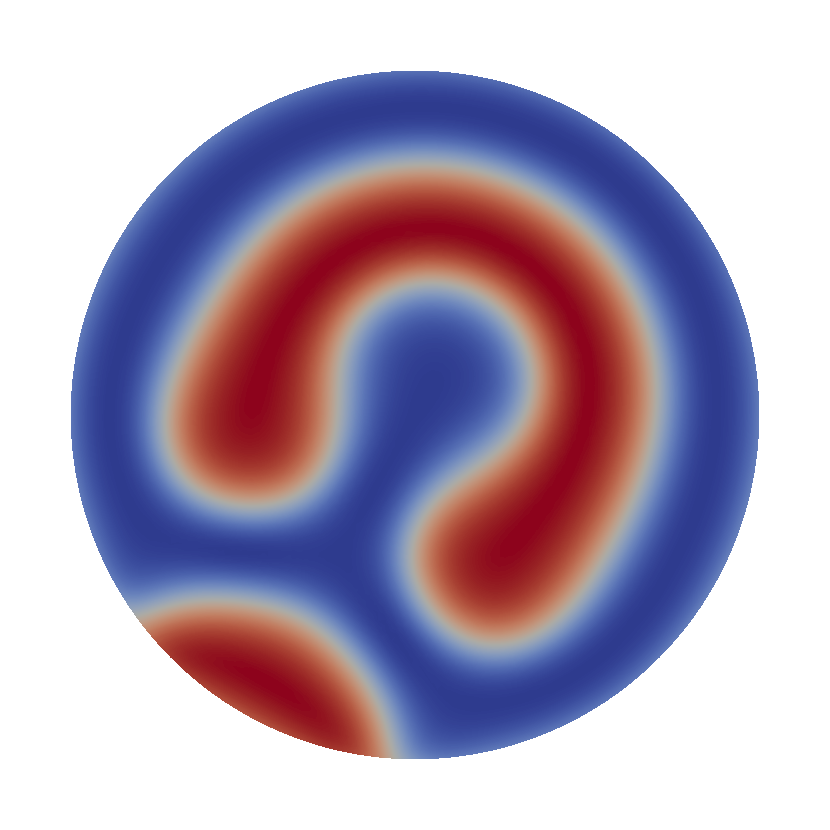}
\caption{$D=7$}
\end{subfigure}
\begin{subfigure}[b]{0.095\textwidth}
\includegraphics[width=\textwidth]{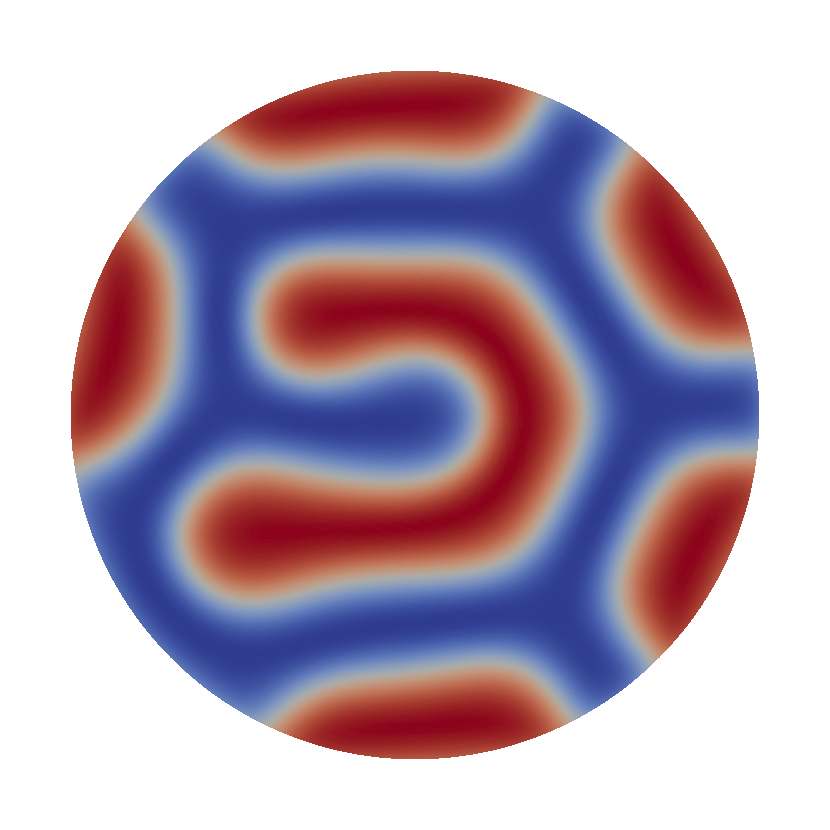}
\caption{$D=8$}
\end{subfigure}
\begin{subfigure}[b]{0.04\textwidth}
\includegraphics[width=\textwidth]{pics/legend.png}
\caption{}
\end{subfigure}
\caption{Experiment of Section~\ref{sec:numerics1}.
Magnetization $m_3$ of the relaxed state for the uniform out-of-plane initial condition~{\rm(i)} and different values of the DMI constant (in \si{\milli\joule\per\square\meter}).
The pictures refer to the states computed with TPS1 ($\theta=1$).}
\label{fig:relaxed:fm}
\end{figure}
\begin{figure}[t]
\captionsetup[subfigure]{labelformat=empty}
\centering
\begin{subfigure}[b]{0.095\textwidth}
\includegraphics[width=\textwidth]{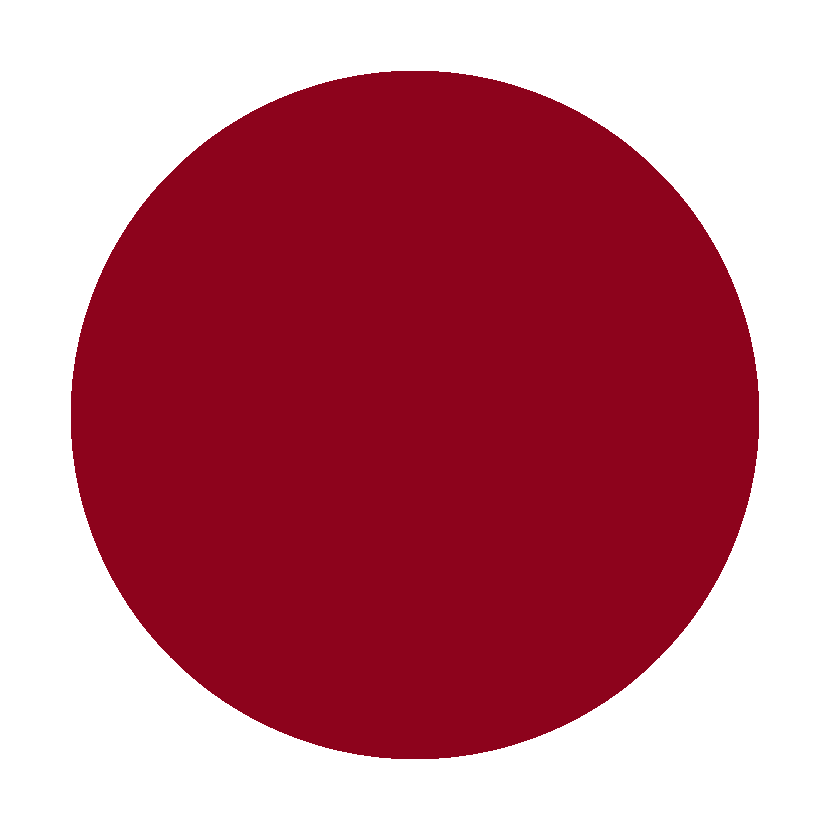}
\caption{$D=0$}
\end{subfigure}
\begin{subfigure}[b]{0.095\textwidth}
\includegraphics[width=\textwidth]{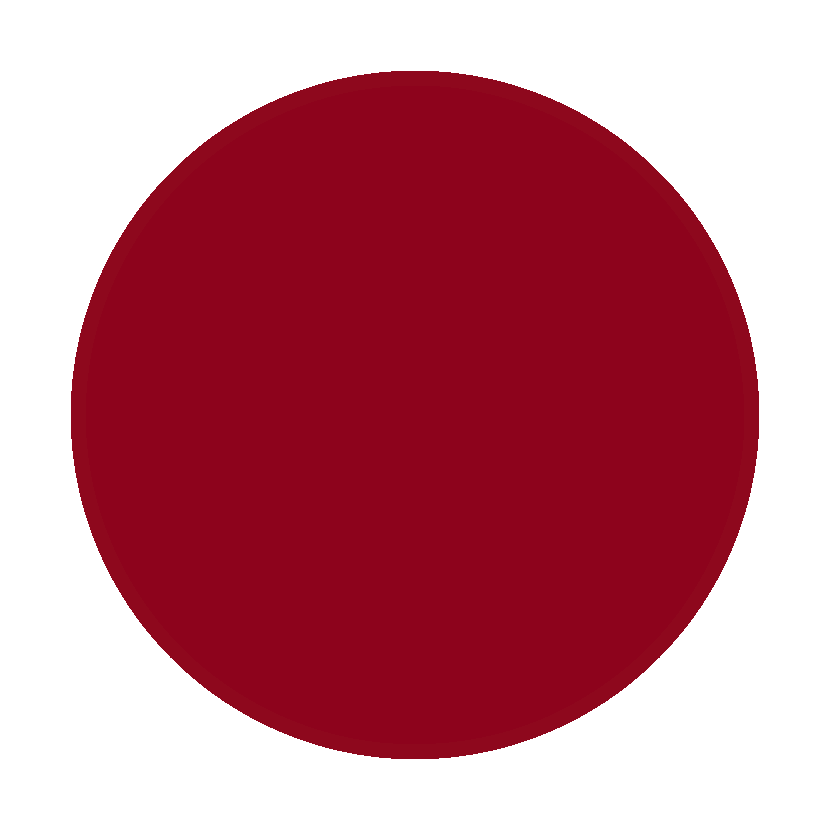}
\caption{$D=1$}
\end{subfigure}
\begin{subfigure}[b]{0.095\textwidth}
\includegraphics[width=\textwidth]{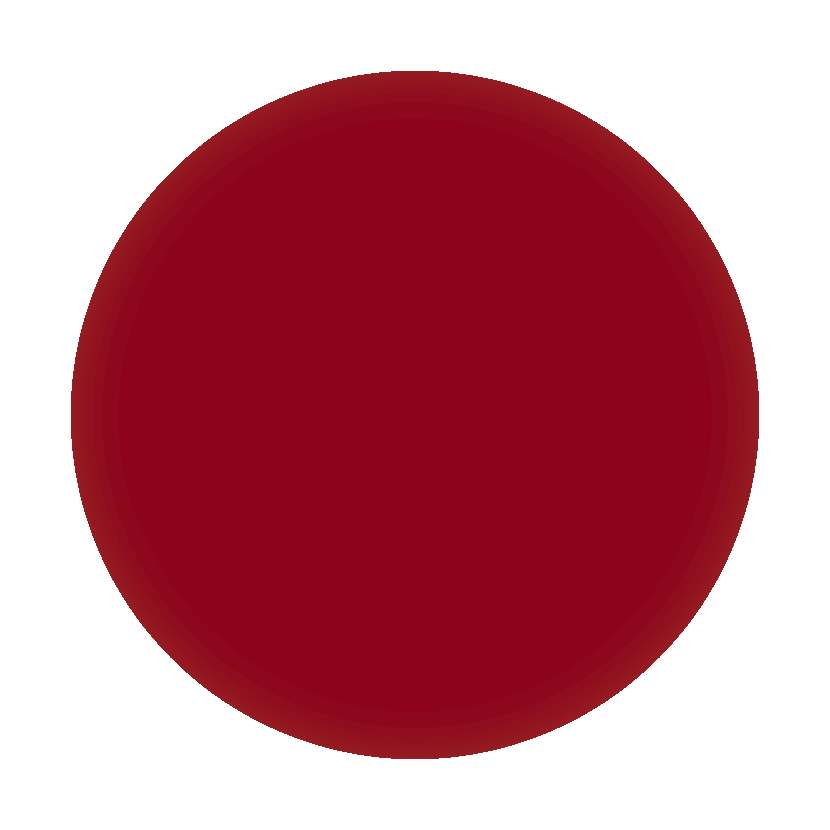}
\caption{$D=2$}
\end{subfigure}
\begin{subfigure}[b]{0.095\textwidth}
\includegraphics[width=\textwidth]{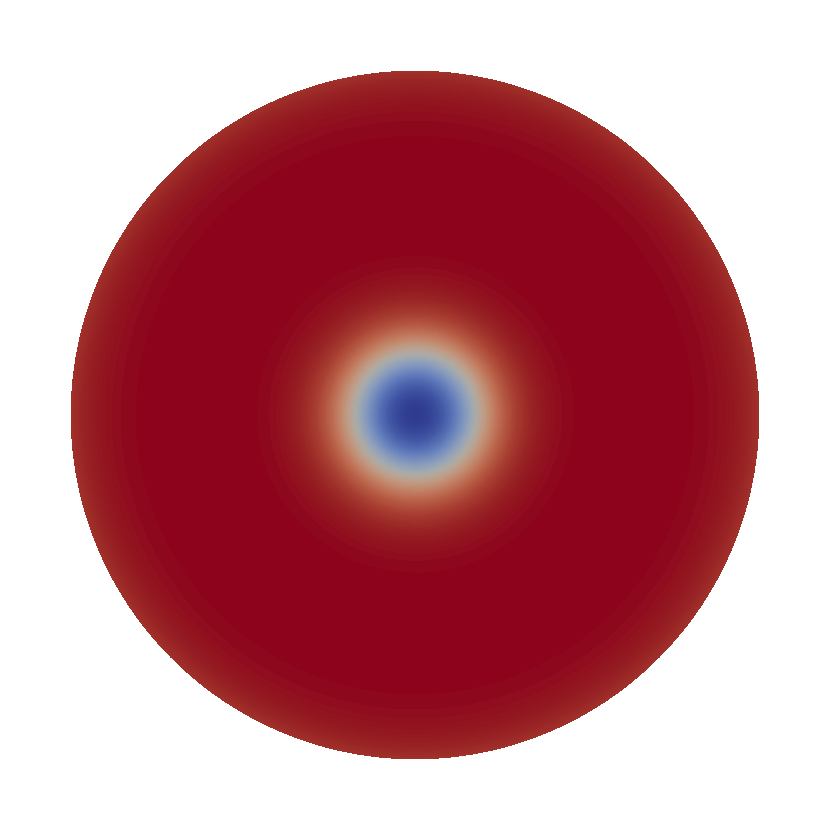}
\caption{$D=3$}
\end{subfigure}
\begin{subfigure}[b]{0.095\textwidth}
\includegraphics[width=\textwidth]{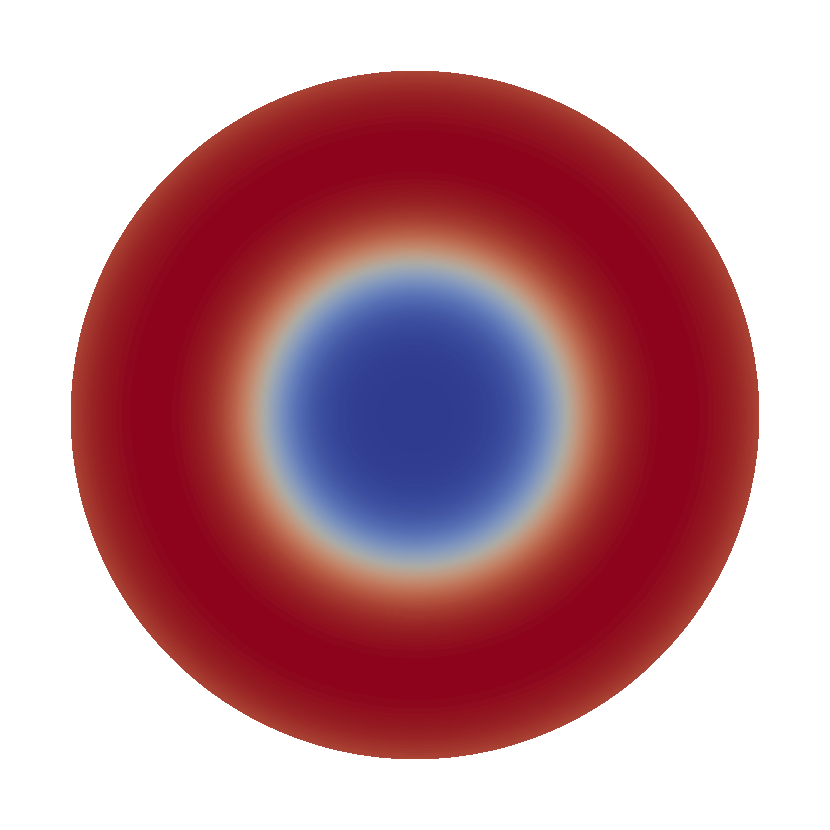}
\caption{$D=4$}
\end{subfigure}
\begin{subfigure}[b]{0.095\textwidth}
\includegraphics[width=\textwidth]{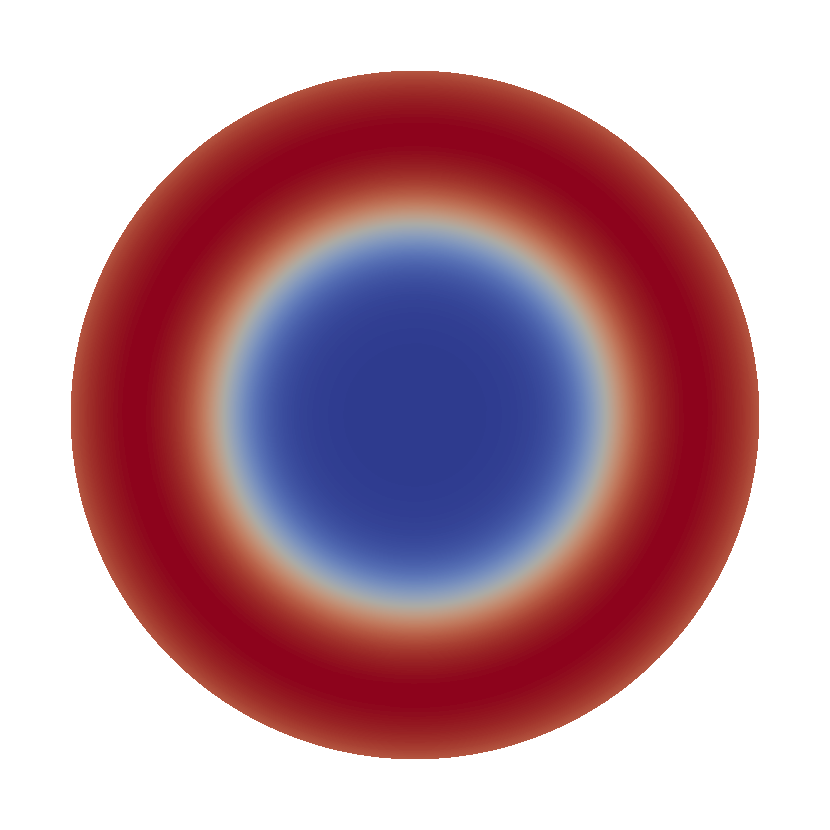}
\caption{$D=5$}
\end{subfigure}
\begin{subfigure}[b]{0.095\textwidth}
\includegraphics[width=\textwidth]{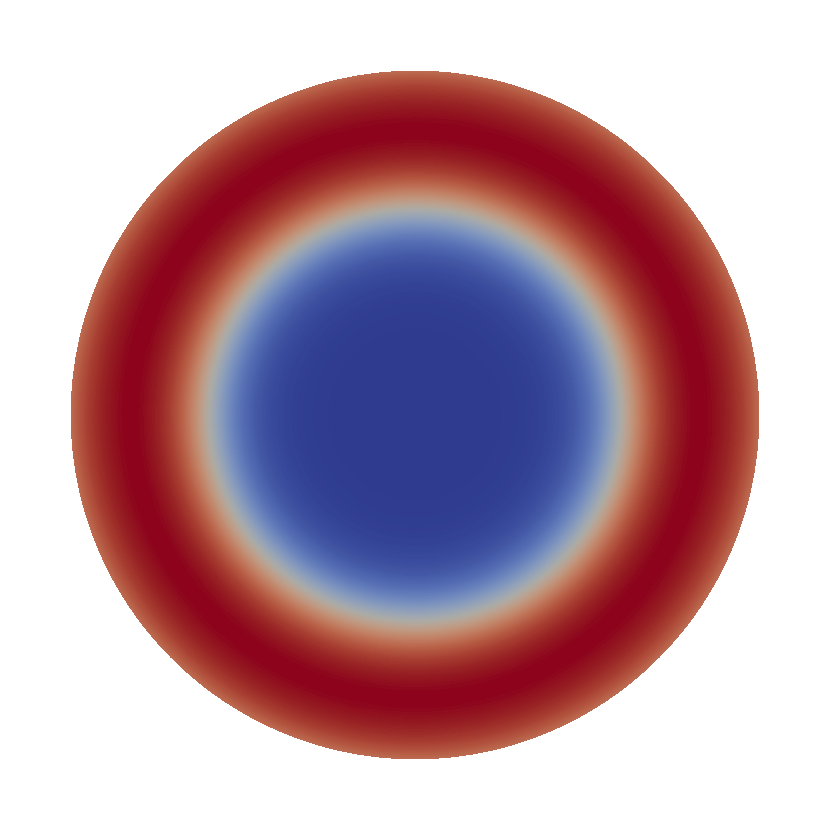}
\caption{$D=6$}
\end{subfigure}
\begin{subfigure}[b]{0.095\textwidth}
\includegraphics[width=\textwidth]{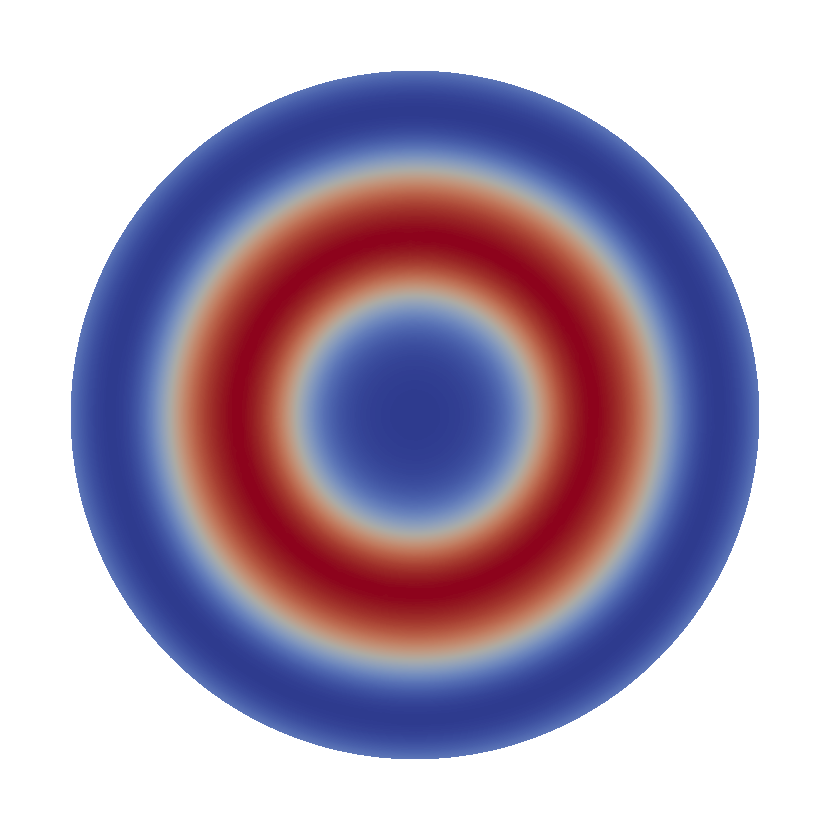}
\caption{$D=7$}
\end{subfigure}
\begin{subfigure}[b]{0.095\textwidth}
\includegraphics[width=\textwidth]{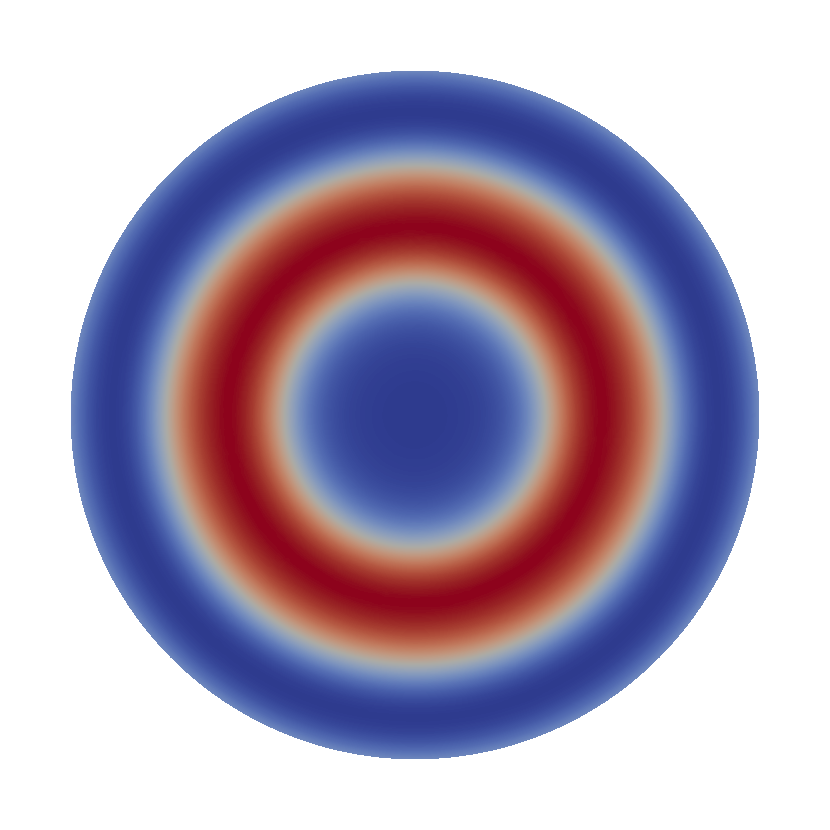}
\caption{$D=8$}
\end{subfigure}
\begin{subfigure}[b]{0.04\textwidth}
\includegraphics[width=\textwidth]{pics/legend.png}
\caption{}
\end{subfigure}
\caption{Experiment of Section~\ref{sec:numerics1}.
Magnetization $m_3$ of the relaxed state for the skyrmion-like initial condition~{\rm(ii)} and different values of the DMI constant (in \si{\milli\joule\per\square\meter}).
The pictures refer to the states computed with TPS1 ($\theta=1$).}
\label{fig:relaxed:skyrmion-like}
\end{figure}
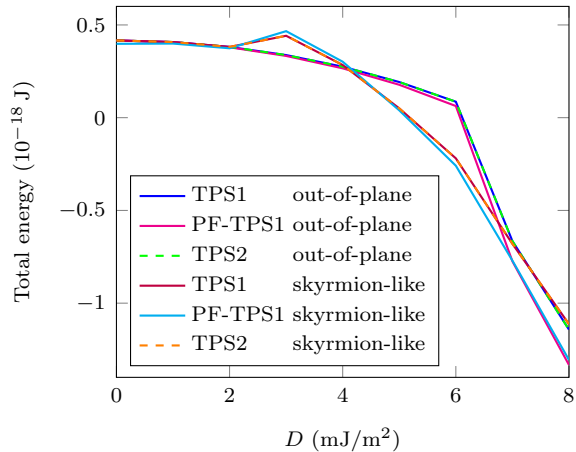
\begin{figure}[b]
\centering
\begin{tikzpicture}
\pgfplotstableread{pics/ex1/energy_fm_tps1.dat}{\fm}
\pgfplotstableread{pics/ex1/energy_sl_tps1.dat}{\sl}
\pgfplotstableread{pics/ex1/energy_fm_pftps1.dat}{\fmpf}
\pgfplotstableread{pics/ex1/energy_sl_pftps1.dat}{\slpf}
\pgfplotstableread{pics/ex1/energy_fm_tps2.dat}{\fmtwo}
\pgfplotstableread{pics/ex1/energy_sl_tps2.dat}{\sltwo}
\begin{axis}[
height = 65mm,
xlabel={\tiny $D$ (\si{\milli\joule\per\square\meter})},
ylabel={\tiny Total energy (\SI{e-18}{\joule})},
xmin=0,
xmax=8,
ymin=-1.4,
ymax=0.6,
legend style={legend pos=south west, legend cell align = left, font = \tiny}
]
\addplot[blue, thick] table[x=D, y=energy]{\fm};
\addplot[magenta, thick] table[x=D, y=energy]{\fmpf};
\addplot[green, dashed, thick] table[x=D, y=energy]{\fmtwo};
\addplot[purple, thick] table[x=D, y=energy]{\sl};
\addplot[cyan, thick] table[x=D, y=energy]{\slpf};
\addplot[orange, dashed, thick] table[x=D, y=energy]{\sltwo};
\legend{
TPS1 \ \ \ \ \ out-of-plane, 
PF-TPS1 out-of-plane,
TPS2 \ \ \ \ \ out-of-plane,
TPS1 \ \ \ \ \ skyrmion-like,
PF-TPS1 skyrmion-like,
TPS2 \ \ \ \ \ skyrmion-like}
\end{axis}
\end{tikzpicture}
\caption{Experiment of Section~\ref{sec:numerics1}.
Total energy of the relaxed state as a function of the DMI constant $D$ for the two considered initial conditions and the three proposed algorithms.}
\label{fig:energyVSdm}
\end{figure}
\par
In the case of the uniform out-of-plane initial condition, the stable state remains a quasi-uniform ferromagnetic state for the values $D=$ \num{0}, \dots, \num{6} \si{\milli\joule\per\square\meter} and turns into a multidomain state for the values $D=$ \num{7}, \num{8} \si{\milli\joule\per\square\meter}; see Figure~\ref{fig:relaxed:fm}.
For $D=$ \num{0}, \dots, \num{6} \si{\milli\joule\per\square\meter}, the slight decrease of the total energy for increasing values of $D$ corresponds to an inward tilt of the magnetization on the boundary of the disk.
In the case of the skyrmion-like initial condition, the stable state is a quasi-uniform ferromagnetic state for the values $D=$ \num{0}, \num{1}, \num{2} \si{\milli\joule\per\square\meter}, a skyrmion for the values $D=$ \num{3}, \dots, \num{6} \si{\milli\joule\per\square\meter}, and a multidomain state for the values $D=$ \num{7}, \num{8} \si{\milli\joule\per\square\meter}; see Figure~\ref{fig:relaxed:skyrmion-like}.
The skyrmion size, i.e., the diameter of the circle $\{ m_3 = 0 \}$ in the $x_1 x_2$-plane, increases from the minimum value of circa \SI{14}{\nano\meter} for $D=$ \SI{3}{\milli\joule\per\square\meter} to the maximum value of circa \SI{48}{\nano\meter} for $D=$ \SI{6}{\milli\joule\per\square\meter}.
As observed in~\cite{scrtf2013}, the fact that for $D=$ \num{3}, \dots, \num{6} \si{\milli\joule\per\square\meter}, which are realistic values for the DMI constant, both the ferromagnetic state and the skyrmion state can be stabilized is very relevant for applications.
Indeed, this bistability can be exploited to code the information in future recording devices (the presence and the absence of a skyrmion can be used to encode one bit); see, e.g., \cite{fcs2013,tmztcf2014}.
\par
In Figure~\ref{fig:energyVSdm}, we plot the total energy of the relaxed state for different values of the DMI constant.
The energy values obtained with TPS1 (the results refer to the case $\theta=1$) and TPS2 are in perfect quantitative agreement with each other and with those reported in~\cite[Figure~1]{scrtf2013}.
The use of PF-TPS1 preserves the qualitative outcome of the experiment, but the quantitative agreement of the energy values with those of~\cite[Figure~1]{scrtf2013}, as a result of the violation of the pointwise constraint $\abs{\mm}=1$, is inevitably lost.
\subsection{Field-induced dynamics of skyrmions in nanodisks} \label{sec:numerics2}
We numerically investigate the stability and the induced dynamics of isolated magnetic skyrmions in helimagnetic materials in response to an applied field pulse.
The sample under consideration is a magnetic nanodisk of diameter \SI{140}{\nano\meter} ($x_1 x_2$-plane) and thickness \SI{10}{\nano\meter} ($x_3$-direction).
The effective field in~\eqref{eq:llg:physical} consists of exchange interaction, bulk DMI, applied external field, and stray field, i.e.,
\begin{equation*}
\Heff(\mm)
= \frac{2A}{\mu_0 \Ms} \Lapl\mm
- \frac{2D}{\mu_0 \Ms} \curl\mm
+ \Hext
+ \Hstray(\mm).
\end{equation*}
We use the material parameters of iron-germanium (\ch{FeGe}), i.e., $A=$ \SI{8.78e-12}{\joule\per\meter}, $D=$ \SI{1.58e-3}{\joule\per\square\meter}, and $\Ms=$ \SI{3.84e5}{\ampere\per\meter}; see, e.g., \cite{babcwcvhcsmf2017}.
\begin{figure}[b]
\captionsetup[subfigure]{labelfont=rm}
\centering
\begin{subfigure}{\textwidth}
\begin{minipage}{15cm}
\centering
\raisebox{-0.5\height}{\includegraphics[height=4cm]{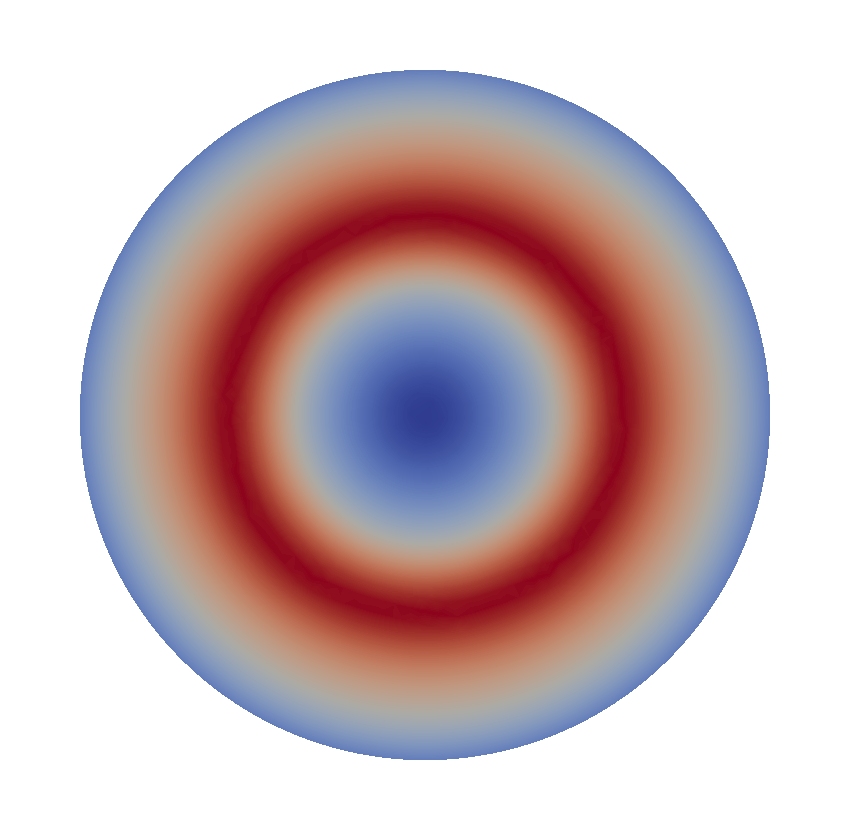}}
\raisebox{-0.5\height}{\includegraphics[height=3cm]{pics/legend.png}}
\hspace*{3mm}\raisebox{-0.5\height}{\includegraphics[height=4cm]{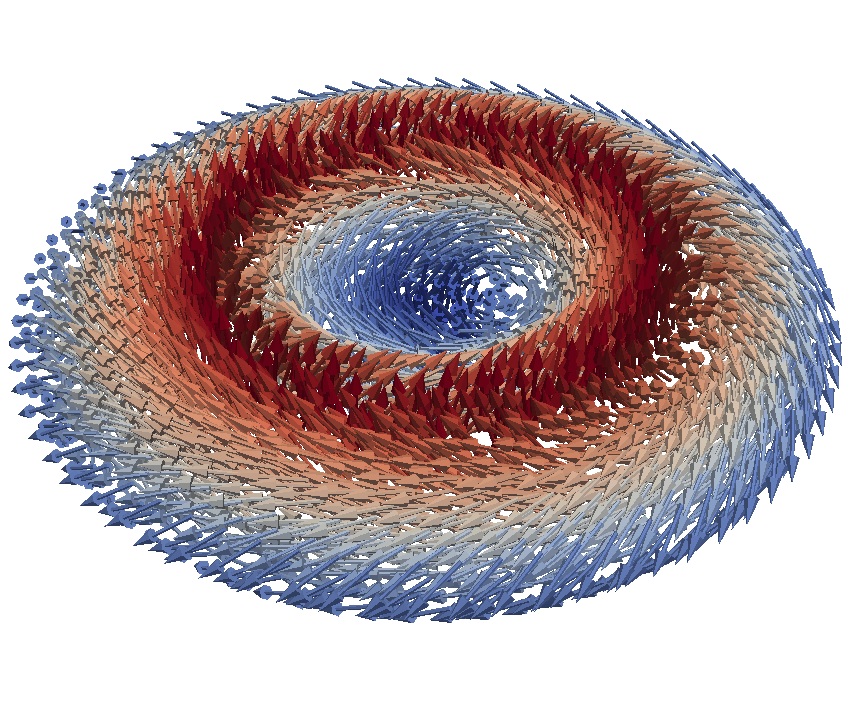}}
\end{minipage}
\caption{Relaxed skyrmion state that we use as initial condition obtained by relaxing a uniform out-of-plane state for \SI{1}{\nano\second}.}
\label{fig:start}
\end{subfigure}
\begin{subfigure}{\textwidth}
\begin{minipage}{15cm}
\centering
\raisebox{-0.5\height}{\includegraphics[height=4cm]{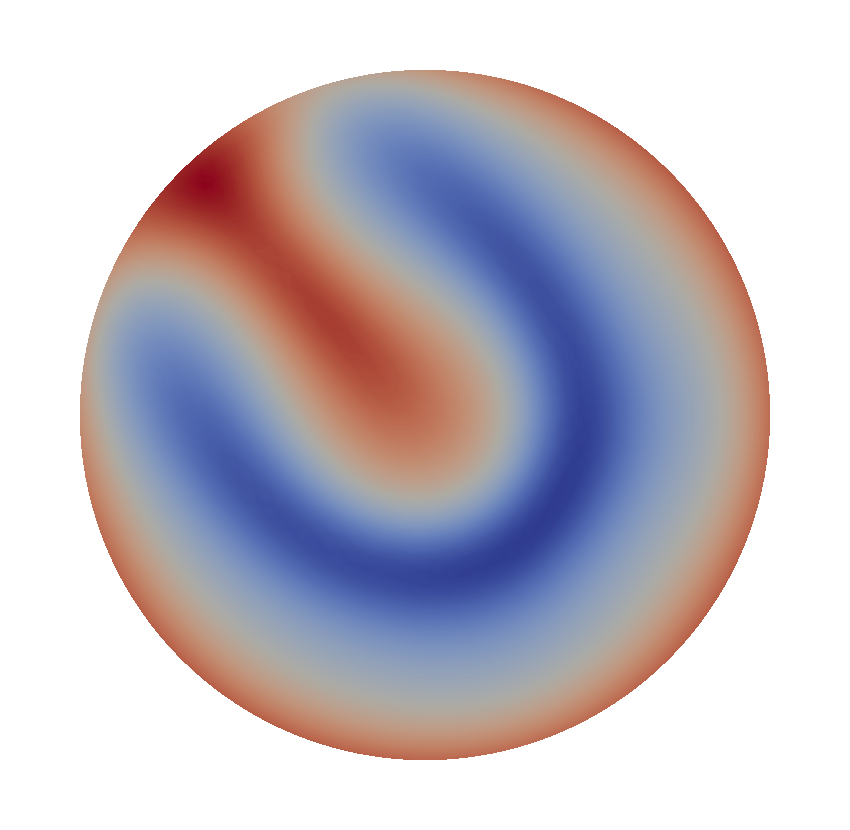}}
\raisebox{-0.5\height}{\includegraphics[height=3cm]{pics/legend.png}}
\hspace*{3mm}\raisebox{-0.5\height}{\includegraphics[height=4cm]{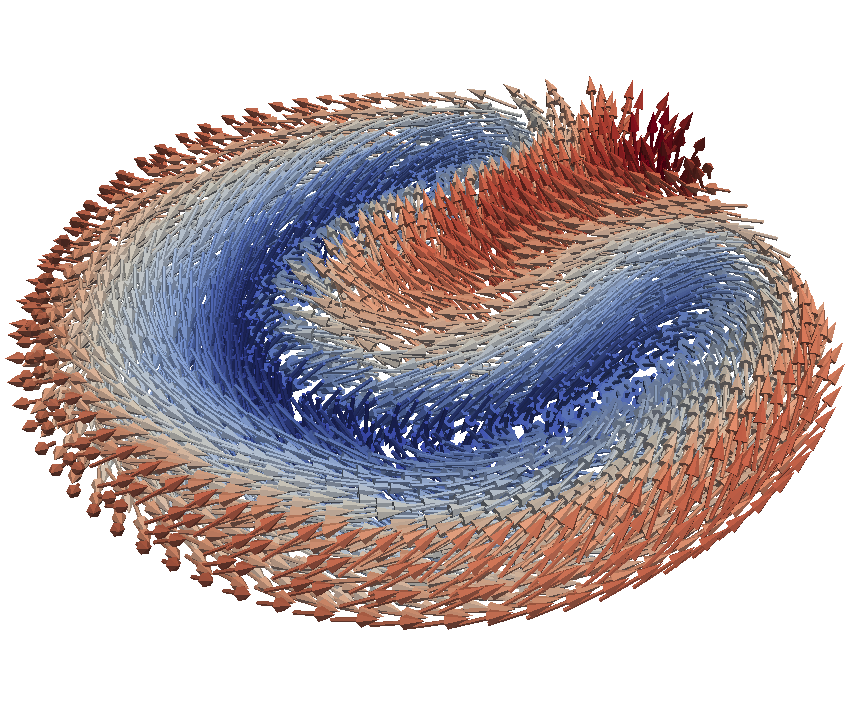}}
\end{minipage}
\caption{Metastable horseshoe state obtained by applying a field pulse of maximum intensity $\mu_0 H_{\mathrm{max}} =$ \SI{200}{\milli\tesla} to the skyrmion of~(a) and relaxing the system for \SI{10}{\nano\second}.}
\label{fig:horseshoe}
\end{subfigure}
\caption{Experiment of Section~\ref{sec:numerics2}.
Relaxed magnetization states: 2D view (left) and 3D view (right).
The pictures refer to the states computed with TPS2.}
\end{figure}
\begin{figure}[t]
\centering
\begin{tikzpicture}
\begin{axis}[
height = 50mm,
xlabel={\tiny $t$ (\si{\pico\second})},
xmin=0,
xmax=150,
ymin=0,
ymax=80,
xtick={0,40,110,150},
xticklabels={0,40,110,150},
ytick={0,60},
yticklabels={0,$H_{\mathrm{max}}$},
]
\addplot[blue,thick] coordinates {(0,0) (40,60)};
\addlegendentry{\tiny $H(t)$}
\addplot[blue,thick] coordinates {(40,60) (110,60)};
\addplot[blue,thick] coordinates {(110,60) (150,0)};
\addplot[dashed] coordinates {(40,0) (40,60)};
\addplot[dashed] coordinates {(110,0) (110,60)};
\end{axis}
\end{tikzpicture}
\caption{Experiment of Section~\ref{sec:numerics2}.
Structure of the applied field pulse:
The field intensity increases linearly in time for \SI{40}{ps} to reach the maximum value $H_{\mathrm{max}}$.
It is then constant and equal to $H_{\mathrm{max}}$ for \SI{70}{ps}.
Finally, it decreases linearly for \SI{40}{ps} to reach the value $0$.
}
\label{fig:field}
\end{figure}
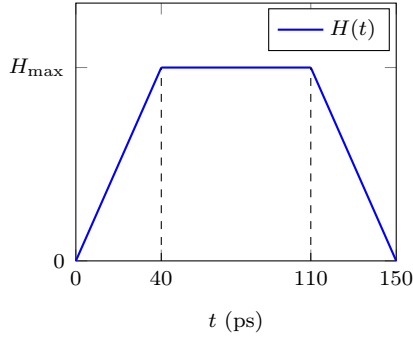
The initial condition for our experiment is obtained by setting $\Hext \equiv (0,0,0)$ and relaxing a uniform out-of-plane ferromagnetic state $\mm^0 \equiv (0,0,1)$ for \SI{3}{\nano\second}.
For the relaxation process, we choose the large value $\alpha=$ \num{1} for the Gilbert damping constant, since we are not interested in the precise magnetization dynamics.
The resulting relaxed state is the skyrmion depicted in Figure~\ref{fig:start}.
Starting from this configuration, we perturb the system from its equilibrium by applying an in-plane field pulse $\Hext(t) = (H(t),0,0)$ of maximum intensity $H_{\mathrm{max}}>0$ for \SI{150}{ps}; see Figure~\ref{fig:field}.
Then, we turn off the applied external field, i.e., $\Hext \equiv (0,0,0)$, and let the system relax to equilibrium. 
In order to capture all possible excitation modes, during the application of the field and the subsequent relaxation process, we set the value of the Gilbert damping constant to $\alpha=$ \num{0.002}, which is considerably smaller than the experimental value of $\alpha=$ \num{0.28} measured for \ch{FeGe}; see~\cite{babcwcvhcsmf2017}.
To probe the limit of the stability of the skyrmion, we test different values for the maximum intensity of the field $H_{\mathrm{max}}$, namely $\mu_0 H_{\mathrm{max}} =$ \num{1}, \num{2}, \num{5}, \num{10}, \num{20}, \num{50}, \num{100}, \num{200} \si{\milli\tesla}.
For the spatial discretization, we consider a regular partition of the nanodisk consisting of \num{36501} tetrahedra (mesh size of \SI{3}{\nano\meter}).
For the time discretization, we consider a uniform time-step size of \SI{0.1}{\pico\second}.
\par
In Figure~\ref{fig:average}, we plot the first \SI{10}{\nano\second} of the time evolution of the second component of the spatially averaged magnetization of the sample, i.e., $\langle m_2(t) \rangle = \abs{\Omega}^{-1} \int_{\Omega} m_2(\xx,t)\, \dx$.
We see that, for the values $\mu_0 H_{\mathrm{max}} =$ \num{1}, \num{2}, \num{5}, \num{10}, \num{20}, \num{50} \si{\milli\tesla}, the induced dynamics is a periodic damped precession of the skyrmion around the center of the sample, which comes back to the initial stable configuration by the relaxation process.
As expected, both the deflection from the stable symmetric initial state and the amplitude of the oscillations increase for larger values of $H_{\mathrm{max}}$.
For the value $\mu_0 H_{\mathrm{max}} =$ \SI{100}{\milli\tesla}, the skyrmion is critically deformed by the applied field pulse, but the initial stable configuration is recovered by the relaxation process.
Note that a different oscillating mode comes into play in this case.
For $\mu_0 H_{\mathrm{max}} =$ \SI{200}{\milli\tesla}, the skyrmion is destroyed.
After approximately \SI{3.5}{\nano\second} of chaotic dynamics, the magnetization configuration turns into a horseshoe state which then starts to rotate around the center of the sample; see Figure~\ref{fig:horseshoe}.
\par
As observed for the experiment of Section~\ref{sec:numerics1}, also in this case the results obtained with TPS1 and TPS2 are in full quantitative agreement with each other.
The use of PF-TPS1 preserves the qualitative outcome of the experiment, but the computed quantities, e.g., the amplitudes of the oscillations depicted in Figure~\ref{fig:average}, are slightly perturbed.
\par
The presented experiment is a preliminary study to investigate the stability and the dynamics of a skyrmion in the presence of a field pulse.
This is done to explore the possibility to use time-resolved scanning Kerr microscopy~\cite{kshmdah2017} which is based on the interplay between laser and field pulses to directly map the dynamics of magnetic skyrmions~\cite{kshmdlah2016}.
\begin{figure}[ht]
\centering
\begin{tikzpicture}
\pgfplotstableread{pics/ex2/data/001mT.dat}{\data}
\begin{axis}[
title = {$\mu_0 H_{\mathrm{max}} =$ \SI{1}{\milli\tesla}},
width = 55mm,
xlabel = {\tiny $t$ (\si{\nano\second})},
ylabel={\tiny $\langle m_2(t) \rangle$},
xmin = 0,
xmax = 10.15,
ymin = -0.5,
ymax = 0.5,
xtick = {0.15,1,2,3,4,5,6,7,8,9,10},
xticklabels = {0.15,1,2,3,4,5,6,7,8,9,10},
]
\addplot[blue,no marks,thick] table[x=t,y=my]{\data};
\end{axis}
\end{tikzpicture}
\qquad
\begin{tikzpicture}
\pgfplotstableread{pics/ex2/data/002mT.dat}{\data}
\begin{axis}[
title = {$\mu_0 H_{\mathrm{max}} =$ \SI{2}{\milli\tesla}},
width = 55mm,
xlabel = {\tiny $t$ (\si{\nano\second})},
ylabel={\tiny $\langle m_2(t) \rangle$},
xmin = 0,
xmax = 10.15,
ymin = -0.5,
ymax = 0.5,
xtick = {0.15,1,2,3,4,5,6,7,8,9,10},
xticklabels = {0.15,1,2,3,4,5,6,7,8,9,10},
]
\addplot[blue,no marks,thick] table[x=t,y=my]{\data};
\end{axis}
\end{tikzpicture}
\\
\begin{tikzpicture}
\pgfplotstableread{pics/ex2/data/005mT.dat}{\data}
\begin{axis}[
title = {$\mu_0 H_{\mathrm{max}} =$ \SI{5}{\milli\tesla}},
width = 55mm,
xlabel = {\tiny $t$ (\si{\nano\second})},
ylabel={\tiny $\langle m_2(t) \rangle$},
xmin = 0,
xmax = 10.15,
ymin = -0.5,
ymax = 0.5,
xtick = {0.15,1,2,3,4,5,6,7,8,9,10},
xticklabels = {0.15,1,2,3,4,5,6,7,8,9,10},
]
\addplot[blue,no marks,thick] table[x=t,y=my]{\data};
\end{axis}
\end{tikzpicture}
\qquad
\begin{tikzpicture}
\pgfplotstableread{pics/ex2/data/010mT.dat}{\data}
\begin{axis}[
title = {$\mu_0 H_{\mathrm{max}} =$ \SI{10}{\milli\tesla}},
width = 55mm,
xlabel = {\tiny $t$ (\si{\nano\second})},
ylabel={\tiny $\langle m_2(t) \rangle$},
xmin = 0,
xmax = 10.15,
ymin = -0.5,
ymax = 0.5,
xtick = {0.15,1,2,3,4,5,6,7,8,9,10},
xticklabels = {0.15,1,2,3,4,5,6,7,8,9,10},
]
\addplot[blue,no marks,thick] table[x=t,y=my]{\data};
\end{axis}
\end{tikzpicture}
\\
\begin{tikzpicture}
\pgfplotstableread{pics/ex2/data/020mT.dat}{\data}
\begin{axis}[
title = {$\mu_0 H_{\mathrm{max}} =$ \SI{20}{\milli\tesla}},
width = 55mm,
xlabel = {\tiny $t$ (\si{\nano\second})},
ylabel={\tiny $\langle m_2(t) \rangle$},
xmin = 0,
xmax = 10.15,
ymin = -0.5,
ymax = 0.5,
xtick = {0.15,1,2,3,4,5,6,7,8,9,10},
xticklabels = {0.15,1,2,3,4,5,6,7,8,9,10},
]
\addplot[blue,no marks,thick] table[x=t,y=my]{\data};
\end{axis}
\end{tikzpicture}
\qquad
\begin{tikzpicture}
\pgfplotstableread{pics/ex2/data/050mT.dat}{\data}
\begin{axis}[
title = {$\mu_0 H_{\mathrm{max}} =$ \SI{50}{\milli\tesla}},
width = 55mm,
xlabel = {\tiny $t$ (\si{\nano\second})},
ylabel={\tiny $\langle m_2(t) \rangle$},
xmin = 0,
xmax = 10.15,
ymin = -0.5,
ymax = 0.5,
xtick = {0.15,1,2,3,4,5,6,7,8,9,10},
xticklabels = {0.15,1,2,3,4,5,6,7,8,9,10},
]
\addplot[blue,no marks,thick] table[x=t,y=my]{\data};
\end{axis}
\end{tikzpicture}
\\
\begin{tikzpicture}
\pgfplotstableread{pics/ex2/data/100mT.dat}{\data}
\begin{axis}[
title = {$\mu_0 H_{\mathrm{max}} =$ \SI{100}{\milli\tesla}},
width = 55mm,
xlabel = {\tiny $t$ (\si{\nano\second})},
ylabel={\tiny $\langle m_2(t) \rangle$},
xmin = 0,
xmax = 10.15,
ymin = -0.5,
ymax = 0.5,
xtick = {0.15,1,2,3,4,5,6,7,8,9,10},
xticklabels = {0.15,1,2,3,4,5,6,7,8,9,10},
]
\addplot[blue,no marks,thick] table[x=t,y=my]{\data};
\end{axis}
\end{tikzpicture}
\qquad
\begin{tikzpicture}
\pgfplotstableread{pics/ex2/data/200mT.dat}{\data}
\begin{axis}[
title = {$\mu_0 H_{\mathrm{max}} =$ \SI{200}{\milli\tesla}},
width = 55mm,
xlabel = {\tiny $t$ (\si{\nano\second})},
ylabel={\tiny $\langle m_2(t) \rangle$},
xmin = 0,
xmax = 10.15,
ymin = -0.5,
ymax = 0.5,
xtick = {0.15,1,2,3,4,5,6,7,8,9,10},
xticklabels = {0.15,1,2,3,4,5,6,7,8,9,10},
]
\addplot[blue,no marks,thick] table[x=t,y=my]{\data};
\end{axis}
\end{tikzpicture}
\caption{Experiment of Section~\ref{sec:numerics2}.
Time evolution of $\langle m_2 \rangle$ for applied field pulses with different intensities.
The plots refer to the results computed with TPS2.
}
\label{fig:average}
\end{figure}
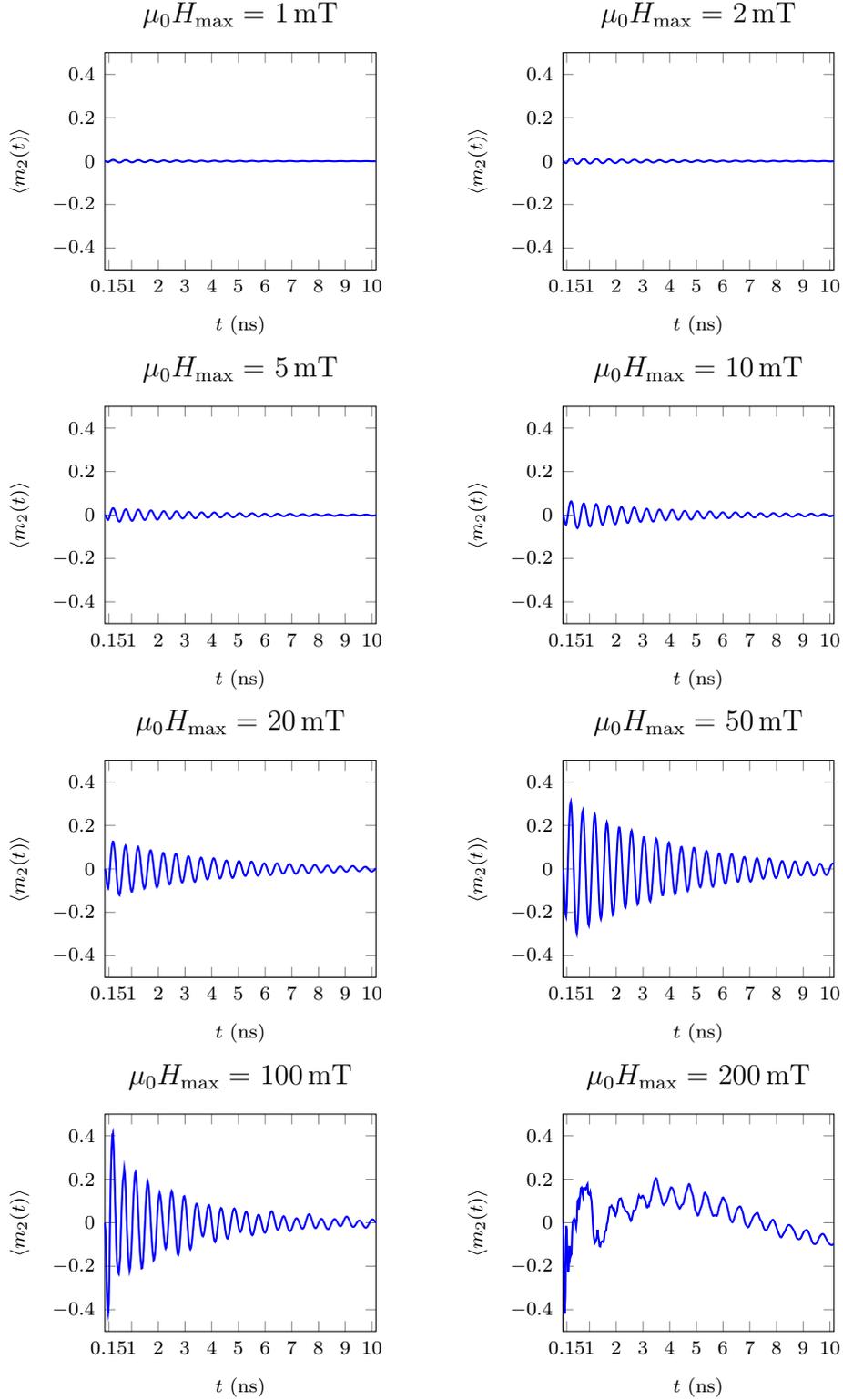
\section{Convergence analysis} \label{sec:convergence}
In this section, we show that all proposed algorithms are well-posed and we present the proof of Theorem~\ref{thm:main}.
To establish the convergence result, we use the standard energy method for proving existence of solutions of linear second-order parabolic problems; see, e.g., \cite[Section~7.1.2]{evans2010}.
The main difference is that, following~\cite{alouges2008a,ahpprs2014,akst2014}, the construction of approximate solutions is not obtained by applying the Galerkin method based on a basis of appropriately normalized eigenfunctions of the Laplace operator, but rather by using the finite element solutions delivered by the numerical schemes.
\subsection{Preliminaries}
We introduce some further notation and collect some auxiliary results.
We consider the nodal interpolant $\interp:C^0(\overline{\Omega})\to\mathcal{S}^1(\Th)$, which is defined by $\interp[v](\zz)=v(\zz)$ for all $\zz\in\Nh$ and $v \in C^0(\overline{\Omega})$.
It is well known that, for $\kappa$-shape-regular meshes and any integer $0 \leq m \leq 2$, the nodal interpolant satisfies the approximation property
\begin{equation} \label{eq:nodalInterpolant}
\norm[L^2(\Omega)]{D^m(v-\interp[v])} \leq C h^{2-m} \norm[L^2(\Omega)]{D^2 v} \quad \text{for all } v \in H^2(\Omega),
\end{equation}
where the constant $C>0$ depends only on $\kappa$.
We denote the vector-valued realization of the nodal interpolant by $\Interp:\CC^0(\overline{\Omega})\to\mathcal{S}^1(\Th)^3$.
The following classical inverse estimate requires the quasi-uniformity of the underlying family of meshes:
For any $1 \leq p \leq \infty$, it holds that
\begin{equation} \label{eq:inverse}
\norm[\LL^p(\Omega)]{\Grad\pphih}
\leq \Cinv h^{-1} \norm[\LL^p(\Omega)]{\pphih}
\quad \text{for all } \pphih\in\mathcal{S}^1(\Th)^3,
\end{equation}
where $\Cinv>0$ depends only on $\kappa$ and $p$. 
Standard scaling arguments show that, for any $1 \leq p < \infty$, we have the discrete norm equivalence
\begin{equation} \label{eq:discreteNormEquivalence}
\Cnorm^{-1} \norm[\LL^p(\Omega)]{\pphih}
\leq \left( h^3 \sum_{\zz\in\Nh} \abs{\pphih(\zz)}^p \right)^{1/p}
\leq \Cnorm \norm[\LL^p(\Omega)]{\pphih}
\quad \text{for all } \pphih\in\mathcal{S}^1(\Th)^3,
\end{equation}
where $\Cnorm>0$ depends only on $\kappa$ and $p$.
\par
To ensure that the approximate magnetization belongs to $\Mh$, TPS1 and TPS2 employ the nodal projection.
We refer to~\cite[Lemma~3.2]{bartels2005} for the proof that the nodal projection $\pphih \mapsto \Interp\big[\pphih/\abs{\pphih}\big]$ does not increase the exchange energy of a discrete function if the underlying mesh fulfills a weak acuteness condition.
Specifically, the angle condition~\eqref{eq:angleCondition} ensures that any $\pphih \in \mathcal{S}^1(\Th)^3$ with $\abs{\pphih(\zz)}\geq 1$ for all $\zz\in\Nh$ satisfies that
\begin{equation} \label{eq:nodalProjectionEnergy}
\norm[\LL^2(\Omega)]{\Grad\Interp\big[\pphih/\abs{\pphih}\big]} \leq \norm[\LL^2(\Omega)]{\Grad\pphih}.
\end{equation}
Owing to the application of the nodal projection, for all $0 \leq i \leq N-1$, the iterates of TPS1 and TPS2 satisfy $\norm[\LL^{\infty}(\Omega)]{\mmh^{i+1}}=1$ and the geometric estimates
\begin{equation*}
\abs{\mmh^{i+1}(\zz)-\mmh^i(\zz)} \leq k \abs{\vvh^i(\zz)}
\ \text{and} \
\abs{\mmh^{i+1}(\zz)-\mmh^i(\zz)-k\vvh^i(\zz)} \leq \frac{1}{2} k^2 \abs{\vvh^i(\zz)}^2
\end{equation*}
for any $\zz\in\Nh$; see~\cite{aj2006,bkp2008}.
With~\eqref{eq:discreteNormEquivalence}, these nodewise inequalities are turned into
\begin{subequations} \label{eq:geoEstCor}
\begin{align}
\label{eq:geoEstCor1}
\norm[\LL^p(\Omega)]{\mmh^{i+1}-\mmh^i} & \leq \Cgeo k \norm[\LL^p(\Omega)]{\vvh^i},\\
\label{eq:geoEstCor2}
\norm[\LL^p(\Omega)]{\mmh^{i+1}-\mmh^i-k\vvh^i} & \leq \Cgeo k^2 \norm[\LL^{2p}(\Omega)]{\vvh^i}^2,
\end{align}
\end{subequations}
where $\Cgeo>0$ depends only on $\kappa$ and $p$.
In the case of PF-TPS1, where the nodal projection is omitted, the geometric estimates~\eqref{eq:geoEstCor} become trivial, but the equality $\lVert \mmh^{i+1} \rVert_{\LL^{\infty}(\Omega)} = 1$ does not hold anymore.
However, for any $1 \leq j \leq N$, the linear time-stepping yields the recursive relation
\begin{equation} \label{eq:pftps1:recursive}
\vert\mmh^j(\zz)\vert^2 = \abs{\mmh^0(\zz)}^2 + k^2 \sum_{i=0}^{j-1} \abs{\vvh^i(\zz)}^2 \quad \text{for all } \zz\in\Nh.
\end{equation}
Together with~\eqref{eq:discreteNormEquivalence}, this leads to the estimate
\begin{equation} \label{eq:pftps1:L2bound}
\Cnorm^{-4} \lVert\mmh^j\lVert_{\LL^2(\Omega)}^2
\leq \norm[\LL^2(\Omega)]{\mmh^0}^2 + k^2 \sum_{i=0}^{j-1} \norm[\LL^2(\Omega)]{\vvh^i}^2.
\end{equation}
\subsection{Well-posedness}
In the following proposition, we prove that the three proposed algorithms are all well-posed.
\begin{proposition} \label{prop:wellposed}
Let $0 \leq i \leq N-1$.
There exists a unique solution $\vvh^i\in\Kh(\mmh^i)$ of~\eqref{eq:tps1} and~\eqref{eq:pftps1}.
There exists a threshold time-step size $k_0>0$, which depends only on $\alpha$, $\lex$, and $\ldm$, such that, if $k \leq k_0$, there exists a unique solution $\vvh^i\in\Kh(\mmh^i)$ of~\eqref{eq:tps2}.
The time-steppings of all algorithms are well-defined.
\end{proposition}
\begin{proof}
For any $0 \leq i \leq N-1$, let $a_{hk}^i(\cdot,\cdot)$ be the bilinear form appearing on the left-hand side of~\eqref{eq:tps1} and \eqref{eq:pftps1}.
For any $\pphih\in\mathcal{S}^1(\Th)^3$, it holds that
\begin{equation*}
a_{hk}^i(\pphih,\pphih)
= \alpha \norm[\LL^2(\Omega)]{\pphih}^2
+ \lex^2 \theta k \norm[\LL^2(\Omega)]{\Grad\pphih}^2.
\end{equation*}
Hence, the bilinear form is elliptic, even on the full space $\mathcal{S}^1(\Th)^3$.
Existence and uniqueness of the solution $\vvh^i \in \Kh(\mmh^i)$ of~\eqref{eq:tps1} and \eqref{eq:pftps1} thus follow from the Lax--Milgram theorem.
\par
Similarly, for any $0 \leq i \leq N-1$, let $b_{hk}^i(\cdot,\cdot)$ denote the bilinear form on the left-hand side of~\eqref{eq:tps2}.
For any $\eps>0$ and $\pphih\in\mathcal{S}^1(\Th)^3$, using~\eqref{eq:young} and~\eqref{eq:propertyCutOff}, we deduce that
\begin{equation*}
\begin{split}
b_{hk}^i(\pphih,\pphih)
& = \inner{W_{M(k)}(\lambda_h^i)\pphih}{\pphih}
+ \frac{\lex^2}{2} k [1 + \rho(k)] \norm[\LL^2(\Omega)]{\Grad\pphih}^2
+ \frac{\ldm}{2} k \inner{\pphih}{\curl\pphih} \\
& \geq \left(\frac{2 \alpha^2}{2 \alpha + M(k) k} - \frac{\ldm}{4\eps}k \right) \norm[\LL^2(\Omega)]{\pphih}^2
+ \frac{1}{2} \left(\lex^2 - \frac{\eps\ldm}{2} \right) k \norm[\LL^2(\Omega)]{\Grad\pphih}^2.
\end{split}
\end{equation*}
We choose $\eps=\lex^2/\ldm$.
With the properties~\eqref{eq:propertyOfM} of the cut-off function $M(k)$, it follows that both coefficients in front of the norms are positive if the time-step size $k$ is sufficiently small.
Then, the bilinear form is elliptic and~\eqref{eq:tps2} admits a unique solution $\vvh^i \in \Kh(\mmh^i)$.
\par
The linear time-stepping in step~\ref{item:pftps1-2} of PF-TPS1 is clearly well-defined.
In the case of TPS1 and TPS2, which include the nodal projection, since $\mmh^i\in\Mh$ and $\vvh \in \Kh(\mmh^i)$, it holds that
\begin{equation*}
\abs{\mmh^i(\zz) + k\vvh^i(\zz)}^2
= \abs{\mmh^i(\zz)}^2 + k^2\abs{\vvh^i(\zz)}^2 = 1 + k^2\abs{\vvh^i(\zz)}^2 \geq 1
\quad \text{for any } \zz\in\Nh.
\end{equation*}
The time-steppings of TPS1 and TPS2 are therefore also well defined.
\end{proof}
\subsection{Discrete energy law and stability}
In this section, we establish the discrete energy laws and study the stability for the discrete iterates delivered by the algorithms.
We first observe that, given $C_0>0$, assumption~\eqref{eq:convergenceMh0} provides some $h_0>0$ such that $\norm[\HH^1(\Omega)]{\mmh^0} \leq C_0$ for all $h \leq h_0$.
\par
In the following proposition, we prove the result for TPS1.
\begin{proposition}[Discrete energy law and stability of TPS1] \label{prop:tps1:energy}
Let $1 \leq j \leq N$.
There exists a constant $C>0$, which depends only on $\kappa$ and $\ldm$, such that the iterates of TPS1 satisfy the discrete energy law
\begin{equation} \label{eq:tps1:energy}
\E(\mmh^j)
+ \left(\alpha - C h^{-1} k \right) k \sum_{i=0}^{j-1} \norm[\LL^2(\Omega)]{\vvh^i}^2
+ \lex^2 (\theta - 1/2) k^2 \sum_{i=0}^{j-1} \norm[\LL^2(\Omega)]{\Grad\vvh^i}^2
\leq \E(\mmh^0).
\end{equation}
Moreover, there exists a constant $C'>0$, which depends only on $\alpha$, $\kappa$, and $\ldm$, such that, if $h \leq h_0$ and $k \leq C'h$, the iterates of TPS1 satisfy the stability estimate
\begin{equation} \label{eq:tps1:stability}
\norm[\HH^1(\Omega)]{\mmh^j}^2
+ k \sum_{i=0}^{j-1} \norm[\LL^2(\Omega)]{\vvh^i}^2
+ (\theta - 1/2) k^2 \sum_{i=0}^{j-1} \norm[\LL^2(\Omega)]{\Grad\vvh^i}^2
\leq C'',
\end{equation}
where the constant $C''>0$ depends only on $\alpha$, $C_0$, $\kappa$, $\lex$, $\ldm$, and $\abs{\Omega}$.
\end{proposition}
\begin{proof}
For any $0 \leq i \leq j-1$, we test~\eqref{eq:tps1} with $\pphih=\vvh^i \in \Kh(\mmh^i)$ to obtain the identity
\begin{equation*}
\begin{split}
& \alpha \norm[\LL^2(\Omega)]{\vvh^i}^2
+ \lex^2 \theta k \norm[\LL^2(\Omega)]{\Grad\vvh^i}^2 \\
& \quad = - \lex^2 \inner{\Grad\mmh^i}{\Grad\vvh^i}
- \frac{\ldm}{2} \inner{\curl\mmh^i}{\vvh^i}
- \frac{\ldm}{2} \inner{\curl\vvh^i}{\mmh^i}.
\end{split}
\end{equation*}
Since the angle condition~\eqref{eq:angleCondition} is satisfied, we obtain that
\begin{subequations} \label{eq:tps1:temp:energy}
\begin{equation}
\begin{split}
\frac{\lex^2}{2} \norm[\LL^2(\Omega)]{\Grad\mmh^{i+1}}^2
& \stackrel{\eqref{eq:nodalProjectionEnergy}}{\leq} \frac{\lex^2}{2} \norm[\LL^2(\Omega)]{\Grad\mmh^i + k \Grad\vvh^i}^2 \\
& \stackrel{\phantom{\eqref{eq:nodalProjectionEnergy}}}{=} \frac{\lex^2}{2} \norm[\LL^2(\Omega)]{\Grad\mmh^i}^2
+ \lex^2 k \inner{\Grad\mmh^i}{\Grad\vvh^i}
+ \frac{\lex^2}{2} k^2 \norm[\LL^2(\Omega)]{\Grad\vvh^i}^2.
\end{split}
\end{equation}
Hence, it follows that
\begin{equation}
\begin{split}
& \frac{\lex^2}{2} \norm[\LL^2(\Omega)]{\Grad\mmh^{i+1}}^2
- \frac{\lex^2}{2} \norm[\LL^2(\Omega)]{\Grad\mmh^i}^2
+ \alpha k \norm[\LL^2(\Omega)]{\vvh^i}^2
+ \lex^2 (\theta - 1/2) k^2 \norm[\LL^2(\Omega)]{\Grad\vvh^i}^2 \\
& \quad \leq - \frac{\ldm}{2} k \inner{\curl\mmh^i}{\vvh^i}
- \frac{\ldm}{2} k \inner{\curl\vvh^i}{\mmh^i}.
\end{split}
\end{equation}
We obtain the energy inequality
\begin{equation}
\begin{split}
& \E(\mmh^{i+1}) - \E(\mmh^i)
+ \alpha k \norm[\LL^2(\Omega)]{\vvh^i}^2
+ \lex^2 (\theta - 1/2) k^2 \norm[\LL^2(\Omega)]{\Grad\vvh^i}^2 \\
& \quad \leq - \frac{\ldm}{2} k \inner{\curl\mmh^i}{\vvh^i}
- \frac{\ldm}{2} k \inner{\curl\vvh^i}{\mmh^i} \\
& \qquad + \frac{\ldm}{2} \inner{\curl\mmh^{i+1}}{\mmh^{i+1}}
- \frac{\ldm}{2} \inner{\curl\mmh^i}{\mmh^i}.
\end{split}
\end{equation}
\end{subequations}
With some simple algebraic manipulations, we rewrite the last four terms of the right-hand side (those which involve the curl operator) as
\begin{equation} \label{eq:tps2:temp:energy}
\begin{split}
& \inner{\curl\mmh^{i+1}}{\mmh^{i+1}}
- \inner{\curl\mmh^i}{\mmh^i}
- k \inner{\curl\mmh^i}{\vvh^i}
- k \inner{\curl\vvh^i}{\mmh^i} \\
& \quad = \inner{\curl(\mmh^{i+1}-\mmh^i - k \vvh^i)}{\mmh^{i+1}}
+ \inner{\curl\mmh^i}{\mmh^{i+1} - \mmh^i - k \vvh^i} \\
& \qquad + k \inner{\curl\vvh^i}{\mmh^{i+1} - \mmh^i}.
\end{split}
\end{equation}
Using the inverse estimate~\eqref{eq:inverse} and the geometric estimates~\eqref{eq:geoEstCor}, we infer that
\begin{equation*}
\begin{split}
\lvert \inner{\curl(\mmh^{i+1}-\mmh^i - k \vvh^i)}{\mmh^{i+1}} \rvert
& \leq \norm[\LL^1(\Omega)]{\curl(\mmh^{i+1}-\mmh^i - k \vvh^i)} \norm[\LL^{\infty}(\Omega)]{\mmh^{i+1}} \\
& \leq \sqrt{2} \, \Cinv h^{-1} \norm[\LL^1(\Omega)]{\mmh^{i+1}-\mmh^i - k \vvh^i} \\
& \leq \sqrt{2} \, \Cinv \Cgeo h^{-1} k^2 \norm[\LL^2(\Omega)]{\vvh^i}^2.
\end{split}
\end{equation*}
Similarly, it holds that
\begin{equation*}
\begin{split}
\lvert \inner{\curl\mmh^i}{\mmh^{i+1} - \mmh^i - k \vvh^i} \rvert
& \leq \norm[\LL^{\infty}(\Omega)]{\curl\mmh^i} \norm[\LL^1(\Omega)]{\mmh^{i+1} - \mmh^i - k \vvh^i} \\
& \leq \sqrt{2} \, \Cinv \Cgeo h^{-1} k^2 \norm[\LL^2(\Omega)]{\vvh^i}^2
\end{split}
\end{equation*}
and
\begin{equation*}
\begin{split}
k \lvert \inner{\curl\vvh^i}{\mmh^{i+1} - \mmh^i} \rvert
& \leq k \norm[\LL^2(\Omega)]{\curl\vvh^i} \norm[\LL^2(\Omega)]{\mmh^{i+1} - \mmh^i} \\
& \leq \sqrt{2} \, \Cinv \Cgeo h^{-1} k^2 \norm[\LL^2(\Omega)]{\vvh^i}^2.
\end{split}
\end{equation*}
With $C= 3 \, \Cinv \Cgeo \ldm/\sqrt{2}$, we thus obtain that
\begin{equation*}
\E(\mmh^{i+1}) - \E(\mmh^i)
+ \alpha k \norm[\LL^2(\Omega)]{\vvh^i}^2
+ \lex^2 (\theta - 1/2) k^2 \norm[\LL^2(\Omega)]{\Grad\vvh^i}^2
\leq C h^{-1} k^2 \norm[\LL^2(\Omega)]{\vvh^i}^2.
\end{equation*}
Summation over $0 \leq i \leq j-1$ leads to~\eqref{eq:tps1:energy}.
\par
To show~\eqref{eq:tps1:stability}, we first note that $\norm[\LL^{\infty}(\Omega)]{\mmh^j}=1$ yields that
\begin{equation} \label{eq:tps1:L2bound}
\norm[\LL^2(\Omega)]{\mmh^j}^2 \leq \abs{\Omega}.
\end{equation}
We multiply~\eqref{eq:tps1:L2bound} by $\ldm^2 / \lex^2$ and add the resulting equation to~\eqref{eq:tps1:energy}.
Using the characterization~\eqref{eq:energyNormEquivalence} of the energy, we obtain that
\begin{equation*}
\begin{split}
& \frac{\lex^2}{4} \norm[\LL^2(\Omega)]{\Grad\mmh^j}^2
+ \frac{\ldm^2}{2 \lex^2} \norm[\LL^2(\Omega)]{\mmh^j}^2
+ \left(\alpha - C h^{-1} k \right) k \sum_{i=0}^{j-1} \norm[\LL^2(\Omega)]{\vvh^i}^2 \\
& \quad + \lex^2 (\theta - 1/2) k^2 \sum_{i=0}^{j-1} \norm[\LL^2(\Omega)]{\Grad\vvh^i}^2
\leq \frac{\lex^2 + \ldm^2}{2} \norm[\LL^2(\Omega)]{\Grad\mmh^0}^2
+ \frac{1}{4} \norm[\LL^2(\Omega)]{\mmh^0}^2
+ \frac{\ldm^2}{\lex^2} \abs{\Omega}.
\end{split}
\end{equation*}
Let $C' = \alpha/(2C)$ and $k \leq C' h$.
Since $\theta \geq 1/2$, all terms on the left-hand side are nonnegative.
We obtain~\eqref{eq:tps1:stability}, where the constant $C''>0$ (which we do not compute explicitly) depends only on $\alpha$, $C_0$, $\kappa$, $\lex$, $\ldm$, and $\abs{\Omega}$.
\end{proof}
In the following proposition, we prove the corresponding result for PF-TPS1.
\begin{proposition}[Discrete energy law and stability of PF-TPS1] \label{prop:pftps1:energy}
Let $1 \leq j \leq N$ and $1/2 < \theta \leq 1$.
The iterates of PF-TPS1 satisfy the discrete energy law
\begin{equation} \label{eq:pftps1:energy}
\begin{split}
& \E(\mmh^j)
+ \alpha k \sum_{i=0}^{j-1} \norm[\LL^2(\Omega)]{\vvh^i}^2
+ \lex^2 (\theta - 1/2) k^2 \sum_{i=0}^{j-1} \norm[\LL^2(\Omega)]{\Grad\vvh^i}^2 \\
& \quad = \E(\mmh^0)
+ \frac{\ldm}{2} k^2 \sum_{i=0}^{j-1} \inner{\curl\vvh^i}{\vvh^i}.
\end{split}
\end{equation}
Moreover, there exists a threshold time-step size $k_0>0$, which depends only on $\alpha$, $\kappa$, $\lex$, $\ldm$, and $\theta$, such that, if $h \leq h_0$ and $k \leq k_0$, the iterates of PF-TPS1 satisfy the stability estimate
\begin{equation} \label{eq:pftps1:stability}
\norm[\HH^1(\Omega)]{\mmh^j}^2
+ k \sum_{i=0}^{j-1} \norm[\LL^2(\Omega)]{\vvh^i}^2
+ k^2 \sum_{i=0}^{j-1} \norm[\LL^2(\Omega)]{\Grad\vvh^i}^2
\leq C,
\end{equation}
where the constant $C>0$ depends only on $\alpha$, $C_0$, $\kappa$, $\lex$, $\ldm$, and $\theta$.
\end{proposition}
\begin{proof}
Let $0 \leq i \leq j-1$.
We follow the argument of the proof of Proposition~\ref{prop:tps1:energy}:
Due to the linear time-stepping of PF-TPS1, all the computations in~\eqref{eq:tps1:temp:energy} hold with equality sign and without resorting to the angle condition~\eqref{eq:angleCondition}.
Moreover, all but the last term on the right-hand side of~\eqref{eq:tps2:temp:energy} vanish.
As a result, we obtain the energy identity
\begin{equation*}
\E(\mmh^{i+1}) - \E(\mmh^i)
+ \alpha k \norm[\LL^2(\Omega)]{\vvh^i}^2
+ \lex^2 (\theta - 1/2) k^2 \norm[\LL^2(\Omega)]{\Grad\vvh^i}^2
= \frac{\ldm}{2} k^2 \inner{\curl\vvh^i}{\vvh^i}.
\end{equation*}
Summing this identity over $0 \leq i \leq j-1$, we prove~\eqref{eq:pftps1:energy}.
To estimate the right-hand side, we apply the weighted Young inequality~\eqref{eq:young}, which, for any $\eps>0$, yields that
\begin{equation*}
\inner{\curl\vvh^i}{\vvh^i}
\leq \sqrt{2} \norm[\LL^2(\Omega)]{\Grad\vvh^i} \norm[\LL^2(\Omega)]{\vvh^i}
\leq \eps \norm[\LL^2(\Omega)]{\Grad\vvh^i}^2 + \frac{1}{2\eps} \norm[\LL^2(\Omega)]{\vvh^i}^2.
\end{equation*}
Choosing $\eps=\lex^2(\theta-1/2)/\ldm$, the characterization~\eqref{eq:energyNormEquivalence} of the energy shows that
\begin{equation} \label{eq:pftps1:temp:energy}
\begin{split}
& \frac{\lex^2}{4} \norm[\LL^2(\Omega)]{\Grad\mmh^j}^2
- \frac{\ldm^2}{2 \lex^2} \norm[\LL^2(\Omega)]{\mmh^j}^2
+ \left( \alpha - \frac{\ldm^2}{2 \lex^2(2\theta-1)}k \right) k \sum_{i=0}^{j-1} \norm[\LL^2(\Omega)]{\vvh^i}^2 \\
& \quad + \frac{\lex^2}{2} (\theta - 1/2) k^2 \sum_{i=0}^{j-1} \norm[\LL^2(\Omega)]{\Grad\vvh^i}^2
\leq \frac{\lex^2 + \ldm^2}{2} \norm[\LL^2(\Omega)]{\Grad\mmh^0}^2
+ \frac{1}{4} \norm[\LL^2(\Omega)]{\mmh^0}^2.
\end{split}
\end{equation}
We multiply~\eqref{eq:pftps1:L2bound} by $\Cnorm^4 \ldm^2 / \lex^2$ and add the resulting equation to~\eqref{eq:pftps1:temp:energy} to obtain that
\begin{equation*}
\begin{split}
& \frac{\lex^2}{4} \norm[\LL^2(\Omega)]{\Grad\mmh^j}^2
+ \frac{\ldm^2}{2 \lex^2} \norm[\LL^2(\Omega)]{\mmh^j}^2
+ \left( \alpha - \frac{\ldm^2 [1 + 2 \Cnorm^4 (2 \theta-1)]}{2 \lex^2(2\theta-1)} k \right) k \sum_{i=0}^{j-1} \norm[\LL^2(\Omega)]{\vvh^i}^2 \\
& \quad + \frac{\lex^2}{2} (\theta - 1/2) k^2 \sum_{i=0}^{j-1} \norm[\LL^2(\Omega)]{\Grad\vvh^i}^2 \\
& \qquad \leq \frac{\lex^2 + \ldm^2}{2} \norm[\LL^2(\Omega)]{\Grad\mmh^0}^2
+ \frac{\lex^2 + 4\Cnorm^4 \ldm^2}{4 \lex^2} \norm[\LL^2(\Omega)]{\mmh^0}^2.
\end{split}
\end{equation*}
If $k \leq k_0 = \alpha \, \lex^2(2 \theta-1)/\{\ldm^2 [1 + 2 \Cnorm^4 (2 \theta-1)]\}$, then all terms on the left-hand side are nonnegative.
This leads to~\eqref{eq:pftps1:stability}, where the (explicitly computable) constant $C>0$ depends only on $\alpha$, $C_0$, $\kappa$, $\lex$, $\ldm$, and $\theta$.
\end{proof}
Finally, in the following proposition, we prove the stability result for TPS2.
\begin{proposition}[Discrete energy law and stability of TPS2] \label{prop:tps2:energy}
Let $1 \leq j \leq N$.
Suppose that the time-step size $k$ is sufficiently small, so that~\eqref{eq:tps2} is well-posed by Proposition~\ref{prop:wellposed}.
Then, there exists a constant $C>0$, which depends only on $\kappa$ and $\ldm$, such that the iterates of TPS2 satisfy the discrete energy law
\begin{equation} \label{eq:tps2:energy}
\begin{split}
& \E(\mmh^j)
+ k \sum_{i=0}^{j-1} \inner{W_{M(k)}(\lambda_h^i)\vvh^i}{\vvh^i}
+ \frac{\lex^2}{2} \rho(k) k^2 \sum_{i=0}^{j-1} \norm[\LL^2(\Omega)]{\Grad\vvh^i}^2 \\
& \quad \leq \E(\mmh^0) + C h^{-1} k^2 \sum_{i=0}^{j-1} \norm[\LL^2(\Omega)]{\vvh^i}^2.
\end{split}
\end{equation}
Moreover, there exists a threshold time-step size $k_0>0$, which depends only on $\alpha$, $\lex$, and $\ldm$, and a constant $C'>0$, which depends only on $\alpha$, $\kappa$, and $\ldm$, such that, if $h \leq h_0$, $k \leq k_0$, and $k \leq C'h$, the iterates of TPS2 satisfy the stability estimate
\begin{equation} \label{eq:tps2:stability}
\norm[\HH^1(\Omega)]{\mmh^j}^2
+ k \sum_{i=0}^{j-1} \norm[\LL^2(\Omega)]{\vvh^i}^2
+ \rho(k) k^2 \sum_{i=0}^{j-1} \norm[\LL^2(\Omega)]{\Grad\vvh^i}^2
\leq C'',
\end{equation}
where the constant $C''>0$ depends only on $\alpha$, $C_0$, $\kappa$, $\lex$, $\ldm$, and $\abs{\Omega}$.
\end{proposition}
\begin{proof}
Let $0 \leq i \leq j-1$.
We follow step by step the argument of the proof of Proposition~\ref{prop:tps1:energy} to obtain the inequality
\begin{equation*}
\begin{split}
& \E(\mmh^{i+1}) - \E(\mmh^i)
+ k \inner{W_{M(k)}(\lambda_h^i)\vvh^i}{\vvh^i}
+ \frac{\lex^2}{2} \rho(k) k^2 \norm[\LL^2(\Omega)]{\Grad\vvh^i}^2 \\
& \quad \leq 
- \frac{\ldm}{2} k^2 \inner{\curl\vvh^i}{\vvh^i}
- \frac{\ldm}{2} k \inner{\curl\mmh^i}{\vvh^i}
- \frac{\ldm}{2} k \inner{\mmh^i}{\curl\vvh^i} \\
& \qquad + \frac{\ldm}{2} \inner{\curl\mmh^{i+1}}{\mmh^{i+1}}
- \frac{\ldm}{2} \inner{\curl\mmh^i}{\mmh^i}.
\end{split}
\end{equation*}
We reformulate the terms of the right-hand side which involve the curl operator, i.e.,
\begin{equation*}
\begin{split}
& \inner{\curl\mmh^{i+1}}{\mmh^{i+1}}
- \inner{\curl\mmh^i}{\mmh^i + k \vvh^i}
- k \inner{\curl\vvh^i}{\mmh^i}
- k^2 \inner{\curl\vvh^i}{\vvh^i} \\
& \quad = \inner{\curl(\mmh^{i+1}-\mmh^i - k \vvh^i)}{\mmh^{i+1}}
+ \inner{\curl\mmh^i}{\mmh^{i+1} - \mmh^i - k \vvh^i} \\
& \qquad + k \inner{\curl\vvh^i}{\mmh^{i+1} - \mmh^i}
- k^2 \inner{\curl\vvh^i}{\vvh^i}
\end{split}
\end{equation*}
and proceed with their direct estimation:
Using~\eqref{eq:inverse} and~\eqref{eq:geoEstCor}, we obtain that
\begin{align*}
\lvert \inner{\curl(\mmh^{i+1}-\mmh^i - k \vvh^i)}{\mmh^{i+1}} \rvert
& \leq \sqrt{2} \, \Cinv \Cgeo h^{-1} k^2 \norm[\LL^2(\Omega)]{\vvh^i}^2, \\
\lvert \inner{\curl\mmh^i}{\mmh^{i+1} - \mmh^i - k \vvh^i} \rvert
& \leq \sqrt{2} \,\Cinv \Cgeo h^{-1} k^2 \norm[\LL^2(\Omega)]{\vvh^i}^2, \\
k \lvert \inner{\curl\vvh^i}{\mmh^{i+1} - \mmh^i} \rvert
& \leq \sqrt{2} \,\Cinv \Cgeo h^{-1} k^2 \norm[\LL^2(\Omega)]{\vvh^i}^2, \\
k^2 \lvert \inner{\curl\vvh^i}{\vvh^i} \rvert
& \leq \sqrt{2} \,\Cinv h^{-1} k^2 \norm[\LL^2(\Omega)]{\vvh^i}^2.
\end{align*}
It follows that
\begin{equation*}
\E(\mmh^{i+1}) - \E(\mmh^i)
+ k \inner{W_{M(k)}(\lambda_h^i)\vvh^i}{\vvh^i}
+ \frac{\lex^2}{2} \rho(k) k^2 \norm[\LL^2(\Omega)]{\Grad\vvh^i}^2
\leq C h^{-1} k^2 \norm[\LL^2(\Omega)]{\vvh^i}^2,
\end{equation*}
with $C = \Cinv(3\Cgeo+1)\ldm/\sqrt{2}$.
Summation over $0 \leq i \leq j-1$ leads to~\eqref{eq:tps2:energy}.
\par
If $k$ is sufficiently small, it holds that $W_{M(k)} (\cdot) \geq \alpha/2$; see~\eqref{eq:propertyCutOff}--\eqref{eq:propertyOfM}. 
With the characterization~\eqref{eq:energyNormEquivalence} of the energy and the inequality $\norm[\LL^2(\Omega)]{\mmh^j}^2 \leq \abs{\Omega}$, we obtain that
\begin{equation*}
\begin{split}
& \frac{\lex^2}{4} \norm[\LL^2(\Omega)]{\Grad\mmh^j}^2
+ \frac{\ldm^2}{2 \lex^2} \norm[\LL^2(\Omega)]{\mmh^j}^2
+ \frac{\alpha - 2 C h^{-1} k}{2} k \sum_{i=0}^{j-1} \norm[\LL^2(\Omega)]{\vvh^i}^2 \\
& \quad + \frac{\lex^2}{2} \rho(k) k^2 \sum_{i=0}^{j-1} \norm[\LL^2(\Omega)]{\Grad\vvh^i}^2
\leq \frac{\lex^2 + \ldm^2}{2} \norm[\LL^2(\Omega)]{\Grad\mmh^0}^2
+ \frac{1}{4} \norm[\LL^2(\Omega)]{\mmh^0}^2
+ \frac{\ldm^2}{\lex^2} \abs{\Omega}.
\end{split}
\end{equation*}
Let $C'= \alpha/(4C)$.
If $k \leq C' h$, then all terms on the left-hand side are nonnegative.
Hence, we obtain~\eqref{eq:tps2:stability}, where the constant $C''>0$ (which we do not compute explicitly) depends only on $\alpha$, $C_0$, $\kappa$, $\lex$, $\ldm$, and $\abs{\Omega}$.
\end{proof}
\subsection{Extraction of weakly convergent subsequences}
Exploiting the established stability estimates of the three algorithms, we are now able to prove that the time reconstructions defined by~\eqref{eq:timeApprox} are uniformly bounded.
\begin{proposition} \label{prop:boundedness}
Suppose that the assumptions of Theorem~\ref{thm:main} are satisfied.
For any algorithm, if $h$ and $k$ are sufficiently small, the sequences $\{\mmhk \}$, $\{\mmhk^\pm \}$, and $\{\vvhk^- \}$ are uniformly bounded in the sense that
\begin{equation} \label{eq:boundedness}
\norm[L^{\infty}(0,T;\HH^1(\Omega))]{\mmhk}
+ \norm[L^{\infty}(0,T;\HH^1(\Omega))]{\mmhk^\pm}
+ \norm[\LL^2(\Omega_T)]{\mmhkt}
+ \norm[\LL^2(\Omega_T)]{\vvhk^-} \leq C.
\end{equation}
The constant $C>0$ is independent of $h$ and $k$.
Moreover, it holds that
\begin{equation} \label{eq:GradVto0}
\lim_{h,k \to 0} k \norm[\LL^2(\Omega_T)]{\Grad\vvhk^-} = 0.
\end{equation}
\end{proposition}
\begin{proof}
The estimate~\eqref{eq:boundedness} follows directly from~\eqref{eq:tps1:stability}, \eqref{eq:pftps1:stability}, and~\eqref{eq:tps2:stability}.
For TPS1 and TPS2, also the geometric estimate~\eqref{eq:geoEstCor1} is used to conclude that $\norm[\LL^2(\Omega_T)]{\mmhkt} \leq C$.
\par
The convergence~\eqref{eq:GradVto0} for PF-TPS1 follows from Proposition~\ref{prop:pftps1:energy}.
Indeed, it holds that
\begin{equation*}
k^2 \norm[\LL^2(\Omega_T)]{\Grad\vvhk^-}^2
=  k^3 \sum_{i=0}^{N-1} \norm[\LL^2(\Omega)]{\Grad\vvh^i}^2
\stackrel{\eqref{eq:pftps1:stability}}{\leq} C k.
\end{equation*}
For TPS1 (resp.\ TPS2), we first resort to an inverse estimate to obtain that
\begin{equation*}
k^2 \norm[\LL^2(\Omega_T)]{\Grad\vvhk^-}^2
=  k^3 \sum_{i=0}^{N-1} \norm[\LL^2(\Omega)]{\Grad\vvh^i}^2
\stackrel{\eqref{eq:inverse}}{\leq} \Cinv h^{-2} k^3 \sum_{i=0}^{N-1} \norm[\LL^2(\Omega)]{\vvh^i}^2.
\end{equation*}
The result then follows from Proposition~\ref{prop:tps1:energy} (resp.\ Proposition~\ref{prop:tps2:energy}) and the fact that $k/h \to 0$ as $h,k \to 0$ by assumption.
\end{proof}
With this result, we can now extract weakly convergent subsequences.
\begin{proposition} \label{prop:convergenceSubsequences}
Suppose that the assumptions of Theorem~\ref{thm:main} are satisfied.
Then, for any algorithm, there exists $\mm\in \HH^1(\Omega_T) \cap L^{\infty}(0,T;\HH^1(\Omega))$, which satisfies $\abs{\mm}=1$ a.e.\ in $\Omega_T$, such that the sequences of time reconstructions $\{\mmhk \}$, $\{\mmhk^\pm \}$, and $\{\vvhk^- \}$ admit subsequences (not relabeled) for which it holds that
\begin{subequations} \label{eq:convergences}
\begin{align}
\label{eq:convergences1}
\mmhk \weakto \mm & \quad \text{ in } \HH^1(\Omega_T),\\
\mmhk \to \mm & \quad \text{ in } \HH^s(\Omega_T) \text{ for all } 0 < s < 1, \\
\mmhk \to \mm & \quad \text{ in } L^2(0,T;\HH^s(\Omega)) \text{ for all } 0 < s < 1, \\
\mmhk, \mmhk^\pm \to \mm & \quad \text{ in } \LL^2(\Omega_T), \\
\mmhk, \mmhk^\pm \to \mm & \quad \text{ pointwise a.e.\ in } \Omega_T, \\
\label{eq:convergences6}
\mmhk, \mmhk^\pm \weakstarto \mm & \quad \text{ in } L^{\infty}(0,T;\HH^1(\Omega)), \\
\label{eq:convergences7}
\vvhk^- \weakto \mmt & \quad \text{ in } \LL^2(\Omega_T)
\end{align}
\end{subequations}
as $h,k \to 0$.
\end{proposition}
\begin{proof}
For the sake of clarity, we divide the proof into three steps.
\begin{itemize}
\item \textbf{Step 1:} Proof of the convergence results~\eqref{eq:convergences1}--\eqref{eq:convergences6}.
\end{itemize}
The uniform boundedness~\eqref{eq:boundedness} established by Proposition~\ref{prop:boundedness} allows us to extract weakly convergent subsequences (not relabeled) of $\left\{\mmhk\right\}$, $\left\{\mmhk^{\pm}\right\}$, with possibly different limits, and $\left\{\vvhk^-\right\}$ in $\HH^1(\Omega_T)$,  $L^2(0,T;\HH^1(\Omega))$, and $\LL^2(\Omega_T)$, respectively.
\par
Let $\mm\in\HH^1(\Omega_T)$ denote the weak limit of $\left\{\mmhk\right\}$ in $\HH^1(\Omega_T)$.
The continuous inclusions $\HH^1(\Omega_T) \subset L^2(0,T;\HH^1(\Omega)) \subset \LL^2(\Omega_T)$ and the compact embedding $\HH^1(\Omega_T) \Subset \LL^2(\Omega_T)$ show that $\mmhk \weakto \mm$ in $L^2(0,T;\HH^1(\Omega))$ and $\mmhk \to \mm$ in $\LL^2(\Omega_T)$.
In particular, upon extraction of a further subsequence, we obtain that $\mmhk \to \mm$ pointwise a.e.\ in $\Omega_T$.
\par
Let $0 < s < 1$.
From the interpolation result $[\LL^2(\Omega_T),\HH^1(\Omega_T)]_s = \HH^s(\Omega_T)$,
we obtain the compact embedding $\HH^1(\Omega_T) \Subset \HH^s(\Omega_T)$; see, e.g., \cite[Theorem~6.4.5 and Theorem~3.8.1]{bl1976}.
Since $[\LL^2(\Omega_T),L^2(0,T;\HH^1(\Omega))]_s = L^2(0,T;\HH^s(\Omega))$, which follows, e.g., by~\cite[Theorem~5.1.2]{bl1976},
we deduce that the inclusion $\HH^s(\Omega_T) \subset L^2(0,T;\HH^s(\Omega))$ is continuous.
Hence, it holds that $\HH^1(\Omega_T) \Subset \HH^s(\Omega_T) \subset L^2(0,T;\HH^s(\Omega))$,
from which we conclude that $\mmhk \to \mm$ in both $\HH^s(\Omega_T)$ and $L^2(0,T;\HH^s(\Omega))$.
Moreover, since
\begin{equation*}
\norm[\LL^2(\Omega_T)]{\mmhk-\mmhk^{\pm}}
\stackrel{\eqref{eq:timeApprox}}{\leq} k \norm[\LL^2(\Omega_T)]{\mmhkt}
\stackrel{\eqref{eq:boundedness}}{\lesssim} k,
\end{equation*}
it follows that $\mmhk^{\pm} \weakto \mm$ in $L^2(0,T;\HH^1(\Omega))$ as well as $\mmhk^{\pm} \to \mm$ in $\LL^2(\Omega_T)$ and pointwise a.e.\ in $\Omega_T$.
Finally, since the sequences $\left\{\mmhk\right\}$ and $\left\{\mmhk^{\pm}\right\}$ are uniformly bounded also in $L^{\infty}(0,T;\HH^1(\Omega))$, we can extract further weakly-star convergent subsequences, whose limits coincide with the weak limits in $L^2(0,T;\HH^1(\Omega))$, i.e., it holds that $\mmhk,\mmhk^\pm \weakstarto \mm$ in $L^{\infty}(0,T;\HH^1(\Omega))$.
\begin{itemize}
\item \textbf{Step 2:} Proof of~\eqref{eq:convergences7}.
\end{itemize}
Let $\vv\in\LL^2(\Omega_T)$ such that $\vvhk^-\weakto\vv$ in $\LL^2(\Omega_T)$.
In the case of TPS1 and TPS2, which includes the nodal projection, it holds that
\begin{equation*}
\norm[\LL^1(\Omega_T)]{\mmt-\vv}
\leq \liminf_{h,k \to 0} \norm[\LL^1(\Omega_T)]{\mmhkt-\vvhk^-}
\stackrel{\eqref{eq:geoEstCor2}}{\leq}
\Cgeo^2 k \norm[\LL^2(\Omega_T)]{\vvhk^-}^2
\stackrel{\eqref{eq:boundedness}}{\lesssim} k,
\end{equation*}
which shows that $\vv = \mmt$ a.e.\ in $\Omega_T$.
For PF-TPS1, the result directly follows from the equality $\mmh^{i+1}=\mmh^i + k \vvh^i$.
\begin{itemize}
\item \textbf{Step 3:} $\mm$ satisfies $\abs{\mm}=1$ a.e.\ in $\Omega_T$.
\end{itemize}
In the case of TPS1 and TPS2, since $\interp\big[\abs{\mmh^i}^2\big]=1$ and $\Grad\mmh^i$ is piecewise constant for all $0 \leq i \leq N-1$, it holds that
\begin{equation*}
\big\lVert{\abs{\mmhk^-}^2-1}\big\rVert_{L^2(\Omega_T)}
\lesssim h \norm[\LL^2(\Omega_T)]{\Grad\mmhk^-}
\stackrel{\eqref{eq:boundedness}}{\lesssim} h,
\end{equation*}
which yields the convergence $\abs{\mmhk^-}^2 \to 1$ in $L^2(\Omega_T)$.
Since $\mmhk^- \to \mm$ pointwise a.e.\ in $\Omega_T$, we deduce that $\abs{\mm}=1$ a.e.\ in $\Omega_T$.
\par
In the case of PF-TPS1, we start with a triangle inequality, which shows that
\begin{equation*}
\begin{split}
& \big\lVert{\abs{\mm}^2-1}\big\rVert_{L^1(\Omega_T)} \\
& \quad \leq \big\lVert{\abs{\mm}^2 - \abs{\mmhk^+}^2}\big\rVert_{L^1(\Omega_T)}
+ \big\lVert{\abs{\mmhk^+}^2-\interp\big[\abs{\mmhk^+}^2\big]}\big\rVert_{L^1(\Omega_T)}
+ \big\lVert{\interp\big[\abs{\mmhk^+}^2\big]-1}\big\rVert_{L^1(\Omega_T)}.
\end{split}
\end{equation*}
The first two terms on the right-hand side converge to $0$.
Indeed, on the one hand, it holds that
\begin{equation*}
\big\lVert{\abs{\mm}^2 - \abs{\mmhk^+}^2}\big\rVert_{L^1(\Omega_T)}
\leq \norm[\LL^2(\Omega_T)]{\mm + \mmhk^+}
\norm[\LL^2(\Omega_T)]{\mm - \mmhk^+}
\lesssim \norm[\LL^2(\Omega_T)]{\mm - \mmhk^+}
\end{equation*}
and $\mmhk^+ \to \mm$ in $\LL^2(\Omega_T)$.
On the other hand, using the approximation properties~\eqref{eq:nodalInterpolant} of the nodal interpolant and the fact that $\Grad\mmh^{i+1}$ is piecewise constant, one shows that
\begin{equation*}
\big\lVert{\abs{\mmhk^+}^2-\interp\big[\abs{\mmhk^+}^2\big]}\big\rVert_{L^1(\Omega_T)}
\lesssim h^2 \norm[\LL^2(\Omega_T)]{\Grad\mmhk^+}^2
\stackrel{\eqref{eq:boundedness}}{\lesssim} h^2.
\end{equation*}
To conclude, it remains to show that
\begin{equation} \label{eq:InterpolantConvergesInL1}
\big\lVert{\interp\big[\abs{\mmhk^+}^2\big]-1}\big\rVert_{L^1(\Omega_T)} \to 0.
\end{equation}
For any $t \in (0,T)$, let $0 \leq i \leq N-1$ such that $t \in [t_i,t_{i+1})$.
It holds that
\begin{equation*}
\begin{split}
& \big\lVert{\interp\big[\abs{\mmhk^+(t)}^2\big]-1}\big\rVert_{L^1(\Omega)}
= \big\lVert{\interp\big[\abs{\mmh^{i+1}}^2\big]-1}\big\rVert_{L^1(\Omega)} \\
& \ \leq \big\lVert{\interp\big[\abs{\mmh^{i+1}}^2\big]-\interp\big[\abs{\mmh^0}^2\big]}\big\rVert_{L^1(\Omega)}
+ \big\lVert{\interp\big[\abs{\mmh^0}^2\big]-\abs{\mmh^0}^2}\big\rVert_{L^1(\Omega)}
+ \big\lVert{\abs{\mmh^0}^2-1}\big\rVert_{L^1(\Omega)}.
\end{split}
\end{equation*}
For the first term on the right-hand side, it holds that
\begin{equation} \label{eq:ConstraintLinearDecay}
\begin{split}
\big\lVert{\interp\big[\abs{\mmh^{i+1}}^2\big]-\interp\big[\abs{\mmh^0}^2\big]}\big\rVert_{L^1(\Omega)}
& \stackrel{\eqref{eq:discreteNormEquivalence}}{\lesssim} h^3 \sum_{\zz \in \Nh} \abs{\abs{\mmh^{i+1}(\zz)}^2 - \abs{\mmh^0(\zz)}^2} \\
& \stackrel{\eqref{eq:pftps1:recursive}}{=} h^3 \sum_{\zz \in \Nh} k^2 \sum_{\ell=0}^i \abs{\vvh^{\ell}(\zz)}^2 \\
& \stackrel{\eqref{eq:discreteNormEquivalence}}{\lesssim} k^2 \sum_{\ell=0}^i \norm[\LL^2(\Omega)]{\vvh^{\ell}}^2
\stackrel{\eqref{eq:pftps1:stability}}{\leq} C k.
\end{split}
\end{equation}
Using the approximation properties of $\interp$, we estimate the second term by
\begin{equation*}
\big\lVert{\interp\big[\abs{\mmh^0}^2\big]-\abs{\mmh^0}^2}\big\rVert_{L^1(\Omega)}
\lesssim h^2 \norm[\LL^2(\Omega)]{\Grad\mmh^0}^2
\stackrel{\eqref{eq:convergenceMh0}}{\lesssim} h^2.
\end{equation*}
Finally, since $\abs{\mm^0}=1$ a.e.\ in $\Omega$ by assumption, the third term satisfies that
\begin{equation*}
\begin{split}
\big\lVert{\abs{\mmh^0}^2-1}\big\rVert_{L^1(\Omega)}
& = \big\lVert{\abs{\mmh^0}^2-\abs{\mm^0}^2}\big\rVert_{L^1(\Omega)}
\leq \norm[\LL^2(\Omega)]{\mmh^0+\mm^0} \norm[\LL^2(\Omega)]{\mmh^0-\mm^0} \\
& \lesssim \norm[\LL^2(\Omega)]{\mmh^0-\mm^0}.
\end{split}
\end{equation*}
Thanks to~\eqref{eq:convergenceMh0}, this yields the convergence $\abs{\mmh^0}^2 \to 1$ in $L^1(\Omega)$.
Altogether, this proves~\eqref{eq:InterpolantConvergesInL1} and thus concludes the proof.
\end{proof}
\subsection{Identification of the limit with a weak solution of LLG}
We start with establishing an auxiliary convergence result for the time reconstructions obtained by PF-TPS1.
\begin{lemma} \label{lem:auxiliaryConv}
Suppose that the assumptions of Theorem~\ref{thm:main} are satisfied and let $\{ \mmhk^{\pm} \}$ be the time reconstructions generated by PF-TPS1.
For all $0 < s < 1$, it holds that
\begin{equation*} 
\mmhk^{\pm} \to \mm \quad \text{in } L^2(0,T;\HH^s(\Omega))
\quad \text{as } h,k \to 0.
\end{equation*}
\end{lemma}
\begin{proof}
Let $0 < s < 1$.
It holds that
\begin{equation*}
\begin{split}
& \norm[L^2(0,T;\HH^s(\Omega))]{\mmhk-\mmhk^{\pm}}^2 \\
& \quad = \int_0^T \norm[\HH^s(\Omega)]{\mmhk(t)-\mmhk^{\pm}(t)}^2 \dt \stackrel{\eqref{eq:timeApprox}}{\leq} k^2 \int_0^T \norm[\HH^s(\Omega)]{\mmhkt(t)}^2 \dt \\
& \quad = k^2 \sum_{i=0}^{N-1} \int_{t_i}^{t_{i+1}} \norm[\HH^s(\Omega)]{\mmhkt(t)}^2 \dt
= k^3 \sum_{i=0}^{N-1} \norm[\HH^s(\Omega)]{(\mmh^{i+1}-\mmh^i)/k}^2 \\
& \quad = k^3 \sum_{i=0}^{N-1} \norm[\HH^s(\Omega)]{\vvh^i}^2
\lesssim k^3 \sum_{i=0}^{N-1} \norm[\HH^1(\Omega)]{\vvh^i}^2
= k^2 \norm[\LL^2(\Omega_T)]{\vvhk^-}^2 + k^2 \norm[\LL^2(\Omega_T)]{\Grad\vvhk^-}^2.
\end{split}
\end{equation*}
Since $\mmhk \to \mm$ in $L^2(0,T;\HH^s(\Omega))$ by Proposition~\ref{prop:convergenceSubsequences}, the result follows from~\eqref{eq:boundedness}--\eqref{eq:GradVto0}.
\end{proof}
We have collected all ingredients to finalize the proof of Theorem~\ref{thm:main}.
\begin{proof}[Proof of Theorem~\ref{thm:main}]
By Proposition~\ref{prop:convergenceSubsequences}, for any algorithm, we deduce the desired convergence towards a function $\mm\in L^{\infty}(0,T;\HH^1(\Omega)) \cap H^1(0,T;\LL^2(\Omega))$ satisfying $\abs{\mm}=1$ a.e.\ in $\Omega_T$.
Since $\mmhk \weakto \mm$ in $\HH^1(\Omega_T)$, we also have the weak convergence of the traces, i.e., $\mmhk(0) \weakto \mm(0)$ in $\HH^{1/2}(\Omega)$.
By assumption~\eqref{eq:convergenceMh0}, we deduce that $\mm(0)=\mm^0$ in the sense of traces.
It remains to show that $\mm$ fulfills the variational formulation~\eqref{eq:weak:variational} and the energy inequality~\eqref{eq:weak:energyLaw}.
For the sake of clarity, we consider the three algorithms separately.
\begin{itemize}
\item \textbf{Step 1:} Proof of the result for TPS1.
\end{itemize}
Let $\vvphi\in\CC^{\infty}(\overline{\Omega_T})$ be an arbitrary test function.
For any $0\leq i \leq N-1$ and $t \in (t_i,t_{i+1})$, we test~\eqref{eq:tps1} with $\pphih=\Interp[\mmhk^-(t)\times\vvphi(t)]\in\Kh(\mmh^i)$.
Integrating in time over $t \in (t_i,t_{i+1})$, summing over $0\leq i \leq N-1$, and using the approximation property~\eqref{eq:nodalInterpolant} of the nodal interpolant, we obtain the identity
\begin{equation} \label{eq:tps1:schemeIntegrated}
\begin{split}
& \alpha \int_0^T\inner{\vvhk^-(t)}{\mmhk^-(t)\times\vvphi(t)} \, \dt
+ \int_0^T \inner{\mmhk^-(t)\times\vvhk^-(t)}{\mmhk^-(t)\times\vvphi(t)} \, \dt \\
& \ + \lex^2 \int_0^T \inner{\Grad[\mmhk^-(t)+\theta k\vvhk^-(t)]}{\Grad[\mmhk^-(t)\times\vvphi(t)]} \, \dt 
+ \mathcal{O}(h) \\
& \quad = - \frac{\ldm}{2} \int_0^T \inner{\curl\mmhk^-(t)}{\mmhk^-(t)\times\vvphi(t)} \, \dt \\
& \qquad - \frac{\ldm}{2} \int_0^T \inner{\mmhk^-(t)}{\curl[\mmhk^-(t)\times\vvphi(t)]} \, \dt.
\end{split}
\end{equation}
Using the available convergence results, we would like to pass the latter to the limit as $h,k \to 0$ to obtain~\eqref{eq:weak:variational}.
For the left-hand side, it holds that
\begin{align*}
\alpha \int_0^T\inner{\vvhk^-(t)}{\mmhk^-(t)\times\vvphi(t)} \, \dt
& \to
- \alpha \int_0^T\inner{\mm(t)\times\mmt(t)}{\vvphi(t)} \, \dt, \\
\int_0^T \inner{\mmhk^-(t)\times\vvhk^-(t)}{\mmhk^-(t)\times\vvphi(t)} \, \dt
& \to 
\int_0^T \inner{\mmt(t)}{\vvphi(t)} \, \dt, \\
\lex^2 \int_0^T \inner{\Grad[\mmhk^-(t)+\theta k\vvhk^-(t)]}{\Grad\mmhk^-(t)\times\vvphi(t)} \, \dt
& \to
\lex^2 \int_0^T \inner{\mm(t) \times \Grad\mm(t)}{\Grad\vvphi(t)} \, \dt;
\end{align*}
see~\cite{alouges2008a,bffgpprs2014} for details.
For the first term on the right-hand side, since $\curl\mmhk^- \weakto \curl\mm$ and $\mmhk^-\times\vvphi \to \mm\times\vvphi$ in $\LL^2(\Omega_T)$, it holds that
\begin{equation*}
- \frac{\ldm}{2} \int_0^T \inner{\curl\mmhk^-(t)}{\mmhk^-(t)\times\vvphi(t)} \, \dt
\to - \frac{\ldm}{2} \int_0^T \inner{\curl\mm(t)}{\mm(t)\times\vvphi(t)} \, \dt.
\end{equation*}
Since $\curl(\mmhk^-\times\vvphi) \weakto \curl(\mm\times\vvphi)$ in $\LL^2(\Omega_T)$, it follows that
\begin{equation*}
- \frac{\ldm}{2} \int_0^T \inner{\mmhk^-(t)}{\curl[\mmhk^-(t)\times\vvphi(t)]} \, \dt
\to - \frac{\ldm}{2} \int_0^T \inner{\mm(t)}{\curl[\mm(t)\times\vvphi(t)]} \, \dt.
\end{equation*}
By~\eqref{eq:green}, it holds that
\begin{equation*}
\begin{split}
& - \frac{\ldm}{2} \int_0^T \inner{\curl\mm(t)}{\mm(t)\times\vvphi(t)} \, \dt
- \frac{\ldm}{2} \int_0^T \inner{\mm(t)}{\curl[\mm(t)\times\vvphi(t)]} \, \dt \\
& \quad = - \ldm \int_0^T \inner{\curl\mm(t)}{\mm(t)\times\vvphi(t)} \, \dt
- \frac{\ldm}{2} \int_0^T \edual{\gamma_T[\mm(t)]}{\mm(t)\times\vvphi(t)} \, \dt,
\end{split}
\end{equation*}
which proves~\eqref{eq:weak:variational} for any smooth test function $\vvphi$.
The desired result then follows by density.
\par
The energy inequality~\eqref{eq:weak:energyLaw} is obtained by passing~\eqref{eq:tps1:energy} to the limit as $h,k \to 0$ and using the available convergence results~\eqref{eq:convergences}, assumption~\eqref{eq:convergenceMh0} on the initial condition, the fact that $k/h \to 0$, in combination with standard lower semicontinuity arguments.
\begin{itemize}
\item \textbf{Step 2:} Proof of the result for PF-TPS1.
\end{itemize}
The proof follows the lines of the one for TPS1 discussed in Step~1.
In the proof of the variational formulation~\eqref{eq:weak:variational}, the only difference is the convergence of the second term on the left-hand side of~\eqref{eq:tps1:schemeIntegrated}, which is more subtle here, since omitting the nodal projection the uniform boundedness of $\mmhk^-$ in $\LL^{\infty}(\Omega_T)$ is lost.
To show the desired convergence, we start with recalling the so-called Lagrange identity
\begin{equation} \label{eq:LagrangeIdentity}
(\vec{a}\times\vec{b})\cdot(\vec{c}\times\vec{d})=(\vec{a}\cdot\vec{c})(\vec{b}\cdot\vec{d}) - (\vec{a}\cdot\vec{d})(\vec{b}\cdot\vec{c}) \quad \text{for all } \vec{a},\vec{b},\vec{c},\vec{d} \in\real^3
\end{equation}
and the continuous embedding $\HH^s(\Omega) \subset \LL^4(\Omega)$, which holds for all $s \geq 3/4$.
Choosing an arbitrary $3/4 \leq s < 1$, we obtain the estimate
\begin{equation*}
\begin{split}
\big\lVert{\abs{\mmhk^-}^2-1}\big\rVert_{L^2(\Omega_T)}^2
& = \big\lVert{\abs{\mmhk^-}^2-\abs{\mm}^2}\big\rVert_{L^2(\Omega_T)}^2
= \int_0^T \big\lVert{\abs{\mmhk^-(t)}^2- \abs{\mm(t)}^2}\big\rVert_{L^2(\Omega)}^2 \dt \\
& = \int_0^T \big\lVert{[\mmhk^-(t)+\mm(t)]\cdot[\mmhk^-(t)-\mm(t)]}\big\rVert_{L^2(\Omega)}^2 \dt \\
& \leq \int_0^T \norm[\LL^4(\Omega)]{\mmhk^-(t)+\mm(t)}^2 \norm[\LL^4(\Omega)]{\mmhk^-(t)-\mm(t)}^2 \dt \\
& \leq \int_0^T \norm[\HH^1(\Omega)]{\mmhk^-(t)+\mm(t)}^2 \norm[\HH^s(\Omega)]{\mmhk^-(t)-\mm(t)}^2 \dt \\
& \leq \norm[L^{\infty}(0,T;\HH^1(\Omega))]{\mmhk^-+\mm}^2 \norm[L^2(0,T;\HH^s(\Omega))]{\mmhk^--\mm}^2 \\
& \lesssim \norm[L^2(0,T;\HH^s(\Omega))]{\mmhk^--\mm}^2.
\end{split}
\end{equation*}
Thanks to Lemma~\ref{lem:auxiliaryConv}, we deduce that $\abs{\mmhk^-}^2 \to 1$ in $L^2(\Omega_T)$ as $h,k \to 0$.
Together with the weak convergence $\vvhk^-\cdot\vvphi \weakto \mmt\cdot\vvphi$ in $L^2(\Omega_T)$, it follows that
\begin{equation*}
\begin{split}
& \int_0^T \inner{\mmhk^-(t)\times\vvhk^-(t)}{\mmhk^-(t)\times\vvphi(t)} \, \dt \\
& \quad \stackrel{\eqref{eq:LagrangeIdentity}}{=} \int_0^T \inner{\abs{\mmhk^-(t)}^2}{\vvhk^-(t)\cdot\vvphi(t)} \, \dt
\to \int_0^T \inner{\mmt(t)}{\vvphi(t)} \, \dt.
\end{split}
\end{equation*}
Finally, passing the discrete energy law~\eqref{eq:pftps1:energy} to the limit as $h,k \to 0$, thanks to~\eqref{eq:convergenceMh0}, \eqref{eq:GradVto0}, the available convergence results~\eqref{eq:convergences}, and standard lower semicontinuity arguments, we obtain~\eqref{eq:weak:energyLaw}.
\begin{itemize}
\item \textbf{Step 3:} Proof of the result for TPS2.
\end{itemize}
The verification of the variational formulation~\eqref{eq:weak:variational} follows by the same method used in Step~1 for TPS1.
Given an arbitrary $\vvphi\in\CC^{\infty}(\overline{\Omega_T})$, for any $0\leq i \leq N-1$ and $t \in (t_i,t_{i+1})$, we test~\eqref{eq:tps2} with $\pphih=\Interp[\mmhk^-(t)\times\vvphi(t)]\in\Kh(\mmh^i)$ to obtain
\begin{equation} \label{eq:tps2:schemeIntegrated}
\begin{split}
& \alpha \int_0^T\inner{W_{M(k)}(\lambda_{hk}^-(t))\vvhk^-(t)}{\mmhk^-(t)\times\vvphi(t)} \, \dt \\
& \ + \int_0^T \inner{\mmhk^-(t)\times\vvhk^-(t)}{\mmhk^-(t)\times\vvphi(t)} \, \dt \\
& \ + \lex^2 \int_0^T \inner{\Grad[\mmhk^-(t)+ (1 + \rho(k))(k/2)\vvhk^-(t)]}{\Grad[\mmhk^-(t)\times\vvphi(t)]} \, \dt 
+ \mathcal{O}(h) \\
& \quad = - \frac{\ldm}{2} \int_0^T \inner{\curl[\mmhk^-(t)+(k/2)\vvhk^-(t)]}{\mmhk^-(t)\times\vvphi(t)} \, \dt \\
& \qquad - \frac{\ldm}{2} \int_0^T \inner{\mmhk^-(t)+(k/2)\vvhk^-(t)}{\curl[\mmhk^-(t)\times\vvphi(t)]} \, \dt,
\end{split}
\end{equation}
where, in analogy with~\eqref{eq:timeApprox}, we define the piecewise time reconstruction $\lambda_{hk}^-$ by $\lambda_{hk}^-(t) := \lambda_h^i$ for all $0 \leq i \leq N-1$ and $t \in [t_i,t_{i+1})$.
\par
With the available convergence result~\eqref{eq:convergences} and the convergence properties of $W_{M(k)}(\cdot)$ and $\rho(\cdot)$, each of the three terms on the left-hand side converges towards the corresponding term of~\eqref{eq:weak:variational} as $h,k \to 0$; see~\cite{akst2014} for details.
We discuss the convergence of the two terms on the right-hand side.
Since $\curl\mmhk^- \weakto \curl\mm$ and $\mmhk^-\times\vvphi \to \mm\times\vvphi$ in $\LL^2(\Omega_T)$, it holds that
\begin{equation*}
- \frac{\ldm}{2} \int_0^T \inner{\curl\mmhk^-(t)}{\mmhk^-(t)\times\vvphi(t)} \, \dt
\to
- \frac{\ldm}{2} \int_0^T \inner{\curl\mm(t)}{\mm(t)\times\vvphi(t)} \, \dt.
\end{equation*}
Moreover, it holds that
\begin{equation*}
- \frac{\ldm}{4} k \int_0^T \inner{\curl\vvhk^-(t)}{\mmhk^-(t)\times\vvphi(t)} \, \dt \to 0,
\end{equation*}
which follows from~\eqref{eq:GradVto0}, since
\begin{equation*}
\begin{split}
\abs{k \int_0^T \inner{\curl\vvhk^-(t)}{\mmhk^-(t)\times\vvphi(t)} \, \dt}
& \lesssim k \norm[\LL^2(\Omega_T)]{\Grad\vvhk^-} \norm[\LL^{\infty}(\Omega_T)]{\mmhk^-} \norm[\LL^2(\Omega_T)]{\vvphi} \\
& \lesssim k \norm[\LL^2(\Omega_T)]{\Grad\vvhk^-}.
\end{split}
\end{equation*}
Hence, the first term on the right-hand side of~\eqref{eq:tps2:schemeIntegrated} converges towards
\begin{equation*}
- \frac{\ldm}{2} \int_0^T \inner{\curl\mm(t)}{\mm(t)\times\vvphi(t)} \, \dt.
\end{equation*}
Similarly, we show that the second term on the right-hand side converges towards
\begin{equation*}
- \frac{\ldm}{2} \int_0^T \inner{\mm(t)}{\curl[\mm(t)\times\vvphi(t)]} \, \dt.
\end{equation*}
As shown in Step~1 for TPS1, it holds that
\begin{equation*}
- \frac{\ldm}{2} \int_0^T \inner{\mmhk^-(t)}{\curl[\mmhk^-(t)\times\vvphi(t)]} \, \dt
\to
- \frac{\ldm}{2} \int_0^T \inner{\mm(t)}{\curl[\mm(t)\times\vvphi(t)]} \, \dt.
\end{equation*}
On the other hand, we have that
\begin{equation*}
\abs{ k \int_0^T \inner{\vvhk^-(t)}{\curl[\mmhk^-(t)\times\vvphi(t)]} \, \dt}
\lesssim k \norm[\LL^2(\Omega_T)]{\vvhk^-} \norm[\HH^1(\Omega_T)]{\mmhk^-} \norm[\WW^{1,\infty}(\Omega_T)]{\vvphi}
\lesssim k,
\end{equation*}
which shows that
\begin{equation*}
- \frac{\ldm}{4} k \int_0^T \inner{\vvhk^-(t)}{\curl[\mmhk^-(t)\times\vvphi(t)]} \, \dt \to 0.
\end{equation*}
This proves~\eqref{eq:weak:variational} for any smooth test function $\vvphi$.
By density, we obtain the desired result.
\par
Finally, the energy inequality~\eqref{eq:weak:energyLaw} is obtained by passing to the limit as $h,k \to 0$ the discrete energy law~\eqref{eq:tps2:energy} and using standard lower semicontinuity arguments.
\end{proof}
\bibliographystyle{acm}
\bibliography{ref}
\end{document}